\documentclass[reqno,a4paper,twoside]{amsart}
\usepackage{enumitem}

\setenumerate{label=\textnormal{(\arabic*)}}

\usepackage{amsmath,amssymb,dsfont,verbatim,mathtools,bm,geometry,fge,endnotes}

\usepackage{booktabs}
\newcommand{\ra}[1]{\renewcommand{\arraystretch}{#1}}
\usepackage{multirow}
\usepackage[latin1]{inputenc}
\usepackage[raggedright]{titlesec}
\usepackage{tikz}
\usepackage{float,subfig}
\usepackage{amsrefs}
\usepackage[foot]{amsaddr}

\titleformat{\chapter}[display]
{\normalfont\huge\bfseries}{\chaptertitlename\\thechapter}{20pt}{\Huge}
\titleformat{\section}
{\normalfont\Large\bfseries\center}{\thesection}{1em}{}
\titleformat{\subsection}
{\normalfont\large\bfseries}{\thesubsection}{1em}{}
\titleformat{\subsubsection}[runin]
{\normalfont\normalsize\bfseries}{\thesubsubsection}{1em}{}
\titleformat{\paragraph}[runin]
{\normalfont\normalsize\bfseries}{\theparagraph}{1em}{}
\titleformat{\subparagraph}[runin]
{\normalfont\normalsize\bfseries}{\thesubparagraph}{1em}{}
\titlespacing*{\chapter} {0pt}{50pt}{40pt}
\titlespacing*{\section} {0pt}{3.5ex plus 1ex minus .2ex}{2.3ex plus .2ex}
\titlespacing*{\subsection} {0pt}{3.25ex plus 1ex minus .2ex}{1.5ex plus .2ex}
\titlespacing*{\subsubsection}{0pt}{3.25ex plus 1ex minus .2ex}{1.5ex plus .2ex}
\titlespacing*{\paragraph} {0pt}{3.25ex plus 1ex minus .2ex}{1em}
\titlespacing*{\subparagraph} {\parindent}{3.25ex plus 1ex minus .2ex}{1em}
\usepackage{titletoc}
\contentsmargin{2.55em}
\dottedcontents{section}[1.5em]{}{2.8em}{1pc}
\dottedcontents{subsection}[3.7em]{}{3.4em}{1pc}
\dottedcontents{subsubsection}[7.6em]{}{3.2em}{1pc}
\dottedcontents{paragraph}[10.3em]{}{3.2em}{1pc}

\newcommand{\hs}{\hspace{-0.5pt}}

\def\xcirc{\hs\circ\hs}

\newtheorem{theorem}{Theorem}[section]
\newtheorem{lemma}[theorem]{Lemma}
\newtheorem{proposition}[theorem]{Proposition}
\newtheorem{corollary}[theorem]{Corollary}

\theoremstyle{definition}
\newtheorem{definition}[theorem]{Definition}

\newtheorem{notation}[theorem]{Notation}
\newtheorem{example}[theorem]{Example}

\theoremstyle{remark}
\newtheorem{remark}[theorem]{Remark}

\DeclareMathOperator{\Aut}{Aut}
\DeclareMathOperator{\siq}{SIQC}
\DeclareMathOperator{\CH}{CH}
\DeclareMathOperator{\ord}{ord}
\DeclareMathOperator{\J}{J}
\DeclareMathOperator{\Jac}{Jac}
\DeclareMathOperator{\ide}{id}
\DeclareMathOperator{\Ext}{Ext}
\DeclareMathOperator{\GL}{GL}
\DeclareMathOperator{\Supp}{Supp}
\DeclareMathOperator{\Z}{Z}
\DeclareMathOperator{\en}{en}
\DeclareMathOperator{\En}{En}
\DeclareMathOperator{\factors}{factors}
\DeclareMathOperator{\N}{N}
\DeclareMathOperator{\Ss}{S}
\DeclareMathOperator{\st}{st}
\DeclareMathOperator{\St}{St}
\DeclareMathOperator{\Cone}{Cone}
\DeclareMathOperator{\PE}{PE}
\DeclareMathOperator{\val}{val}
\DeclareMathOperator{\dir}{dir}
\DeclareMathOperator{\Dirsup}{Valsup}
\DeclareMathOperator{\Dirinf}{Valinf}
\DeclareMathOperator{\Succ}{Succ}
\DeclareMathOperator{\Pred}{Pred}
\DeclareMathOperator{\Dir}{Dir}
\DeclareMathOperator{\HH}{H}
\DeclareMathOperator{\linspan}{linspan}
\DeclareMathOperator{\lcm}{lcm}
\DeclareMathOperator{\End}{End}
\newcommand{\ov}{\overline}
\newcommand{\ot}{\otimes}
\newcommand{\wh}{\widehat}
\newcommand{\wt}{\widetilde}
\newcommand{\ep}{\epsilon}
\newcommand{\De}{\Delta}
\newcommand{\urho}{\underline{\rho}}
\newcommand{\usigma}{\underline{\sigma}}
\newcommand{\uDir}{\underline{\Dir}}
\newcommand{\BigFig}[1]{\parbox{12pt}{\Huge #1}}
\newcommand{\BigZero}{\BigFig{0}}
\newcommand{\dpu}{\mathbin{:}}
\renewcommand{\theequation}{\thesection.\arabic{equation}}

\begin{document}
\title[Bi-traceable graphs, three longest paths and Hippchen's conjecture]{Bi-traceable graphs, the intersection of three longest paths and 
Hippchen's conjecture}

\author[Juan Gutierrez]{Juan Gutierrez$^{1}$}
\address{$^1$ Departamento de Ciencia de la Computaci\'on
Universidad de Ingenier\'ia y Tecnolog\'ia (UTEC)}
\email{jgutierreza@utec.edu.pe}

\author[Christian Valqui]{Christian Valqui$^{2,}$$^3$}
\address{$^2$Pontificia Universidad Cat\'olica del Per\'u, Secci\'on Matem\'aticas, PUCP, Av. Universitaria 1801, San Miguel, Lima 32, Per\'u.}
\address{$^3$Instituto de Matem\'atica y Ciencias Afines (IMCA) Calle Los Bi\'ologos 245. Urb San C\'esar.
La Molina, Lima 12, Per\'u.}
\email{cvalqui@pucp.edu.pe}
\thanks{Christian Valqui was supported by PUCP-DGI-CAP-2020-818.}


\maketitle

\begin{abstract}
  Let $P,Q$ be longest paths in a simple graph. We analyze the possible connections between the components of
   $P\cup Q\setminus (V(P)\cap V(Q))$ and introduce the notion of a bi-traceable graph. 
   We use the results for all the possible configurations
   of the intersection points when $\#V(P)\cap V(Q)\le 5$ in order to prove that if the intersection of three longest paths $P,Q,R$ is empty, then
   $\#(V(P)\cap V(Q))\ge 6$. We also prove Hippchen's conjecture for $k\le 6$: If a graph $G$ is $k$-connected for $k\le 6$, and $P$ and $Q$ 
   are longest paths in $G$, then $\#(V(P)\cap V(Q))\ge 6$.
\end{abstract}

\tableofcontents

\section*{Introduction}
Different configurations of the intersection points of longest paths in simple graphs have been analyzed by various authors with different purposes.
Axenovich in~\cite{A} used two configurations $Q_1$ and $Q_2$ that described forbidden connections in $G\setminus V(P)\cap V(Q)$, in order to prove
that in outerplanar graphs every three longest paths share a common point.  Hence, for outerplanar graphs the following conjecture is
true.\vspace{0.3cm}

{\bf Conjecture 1:} {\it For every connected graph, any three of its longest paths have a common vertex.\vspace{0.3cm}}

In~\cite{dR} this conjecture is proven for a broader class of graphs and in~\cite{FFNO} an approach to this conjecture using certain distance
functions is given. It is known that
two longest paths in a connected graph have necessarily a common vertex, but there are connected graphs, where seven longest paths do not have a
common vertex (See for example~\cite{S}). The equivalent question for 3,4,5 and 6 longest paths is still open. See~\cite{SZZ} for a survey on this 
and similar problems.

On the other hand Hippchen in~\cite{H} analyzed forbidden connections in $G\setminus V(P)\cap V(Q)$ when $\#(V(P)\cap V(Q))=2$ in order to prove 
that  in $3$-connected  graphs the intersection of two longest path contains at least 3 vertices. He then stated the following
conjecture:\vspace{0.3cm}

{\bf Conjecture 2:} {\it The intersection of two longest paths in a $k$-connected simple graph has cardinality at least $k$.\vspace{0.3cm}}

This conjecture was an adaptation of a similar conjecture for longest cycles instead of paths, which appeared first in~\cite{Gr}*{p.188} and was
attributed to Scott Smith, and where some partial results are known (See for example~\cite{G}, \cite{ST} and~\cite{Ch}).

In~\cite{G} the first author proved Hippchen's conjecture for $k=4$, and in~\cite{CCP} Hippchen's conjecture is proven for $k=5$. In both cases the
possible configurations and forbidden connections of the graph induced by $V(P)\cap V(Q)$ is analyzed.

The configuration of the intersection vertices of longest cycles (instead of paths) has been considered for example  in~\cite{B79},  
\cite{Gr} and~\cite{ST}. Babai~\cite{B79}*{Lemma 2} uses a forbidden connection in order to prove that in a 3-connected graph, any 
two longest cycles have at least three points in common. Groetschel in~\cite{Gr}  proved that the complement of the intersection 
of two longest cycles is  not connected when there are at most 5 intersection points, Steward and al. used a computer program
in~\cite{ST} to analyze the different configurations of the intersection points in order to show that this remains true for  
$k=6,7$ (for $k=8$ there is a counterexample).

In the present paper we will make a complete analysis of the possible configurations of the intersection points of two longest paths $P$ and $Q$, 
when there are $\ell\le 5$ intersection points. We will analyze the different possibilities for the sequential order of the intersection vertices in
each of the two longest paths $P$ and $Q$, which is given by a permutation of the $\ell$ intersection vertices. Two permutations give the same
configuration  if they can be transformed into each other by exchanging $P$ and $Q$, or by changing the direction in which we travel through the
points in each path.

If the intersection is a single point or has two points, there is only one configuration. If the intersection has three points, there are two
configurations, already described
in~\cite{G}. If the intersection has 4 points, then the 24 different permutations give 7 different configurations, which correspond to 7 different
cases described in~\cite{CCP}.  When the intersection has 5 points, we find 23 different cases arising from the 120 different permutations.

For each configuration we analyze in detail the possible connections between the components of the graph $(P\cup Q)\setminus (V(P)\cap V(Q))$. We 
find a large class of graphs $P\cup Q$, such that in each exterior swap unit (see Definition~\ref{def ESU}) $P\setminus (V(P)\cap V(Q))$ cannot be
 connected with $Q\setminus (V(P)\cap V(Q))$ outside $V(P)\cap V(Q)$ (See Proposition~\ref{prop no path exists} and Remark~\ref{large class}).  In
  these graphs, $V(P)\cap V(Q)$ is a separator, if we additionally assume that $V(P)\ne V(Q)$. For graphs not in this class we analyze in detail
  whether two components of $(P\cup Q)\setminus (V(P)\cap V(Q))$ can be connected such that $P$ and $Q$ are still longest paths.

We will use our results in order to approach three different problems on longest paths, the intersection of three longest paths, 
Hippchen's conjecture and a path version of Groetschel's result on the connectedness of the complement of the intersection
of two longest paths.  One of our main results is the following result on a possible counterexample of Conjecture 1 in 
Theorem~\ref{teorema three paths mayor que 5}:\vspace{0.3cm}

{\it If the intersection of three longest paths $P,Q,R$ is empty, then $\#(V(P)\cap V(Q))\ge 6$.\vspace{0.3cm}}

\noindent By symmetry, this result implies that in this case we also have $\#(V(R)\cap V(Q))\ge 6$ and $\#(V(P)\cap V(R))\ge 6$. Hence we can 
rephrase our result in the following way: \vspace{0.3cm}

{\it Let $P,Q,R$ be three longest paths in a graph. If one of $V(P)\cap V(Q)$, $V(P)\cap V(R)$ or $V(Q)\cap V(R)$ has less than 6 points, then the
three paths have a common vertex.\vspace{0.3cm}}

On the other hand,   if Conjecture 1 is false and there exist three longest paths in a graph with empty intersection, then with our methods one can
carry out a systematic search for a counterexample.

As second main result of the present paper we prove Hippchen's conjecture for $k=6$ (Corollary~\ref{Corollary Hippchen 6}) and give a new  proof for
$k=5$ (Coro\-llary~\ref{Corollary Hippchen 5}). The same methods should be useful in order to prove Hippchen's conjecture for higher $k$, or to prove
it completely.

The third problem to which we apply our methods is the path-version of~\cite{Gr}*{Theorem 1.2(a)} (see also~\cite{ST}). We arrive at the
following results (Theorem~\ref{teorema principal} and Corollary~\ref{Corollary Gr4}):\vspace{0.3cm}

\noindent {\it $[Gr(\ell\le 5)]$ Assume that $P$ and $Q$ are two longest paths in a simple graph $G$. If $V(Q)\ne V(P)$ and $V(P)\cap V(Q)$ has
cardinality $\ell \le 5$, then
it is a separator (called an articulation set in~\cite{Gr}), which means that the complement is not connected.\vspace{0.3cm}

\noindent $[Gr(n\le 7)]$ If $V(Q)\ne V(P)$ and $n=|V(G)|\le 7$ then $V(Q)\cap V(P)$ is a separator.\vspace{0.3cm}}

\noindent Open question: Which are the maximal $\ell$ and $n$ such that the above results remain true?
In an hypotraceable graph we can choose two longest paths $P$ and $Q$ that leave out two connected vertices, and so $V(Q)\cap V(P)$ is a not separator
in that case. Since there exists a hypotraceable graph on 34 vertices (see~\cite{Th}), we know automatically
that $n_{max}\le 33$ and $\ell_{max}\le 31$.

Now consider the
following simple graph with 11 vertices, that has two longest path $P$ and $Q$ of length 9, which satisfy $V(Q)\ne V(P)$ and
moreover, the complement of $V(Q)\cap V(P)$ is connected. Since $\# (V(P)\cap V(Q))=9$ we have $n=11$ and $\ell=9$.

  \begin{tikzpicture}[scale=1]
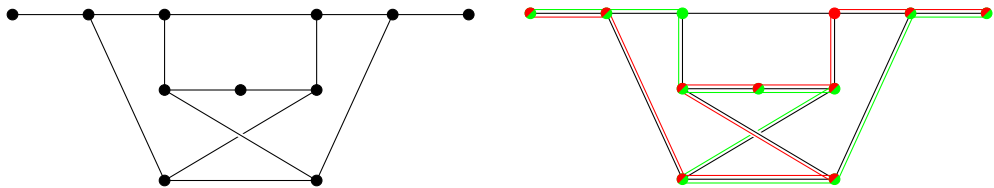
  \draw(3,0.2) node {Simple graph, $n=11$ vertices};
\draw(6,2.3) node {};
\filldraw [black]  (0,3)    circle (2pt)
[black]  (2,0.8)    circle (2pt)
[black]  (1,3)    circle (2pt)
[black]  (2,2)    circle (2pt)
[black]  (2,3)    circle (2pt)
[black]  (3,2)    circle (2pt)
[black]  (4,2)    circle (2pt)
[black]  (4,3)    circle (2pt)
[black]  (4,0.8)    circle (2pt)
[black]  (5,3)    circle (2pt)
[black]  (6,3)    circle (2pt);
\draw[-] (0,3)--(1,3);
\draw[-] (1,3)--(2,3);
\draw[-] (2,0.8)--(1,3);
\draw[-] (2,0.8)--(4,2);
\draw[-] (2,2)--(2,3);
\draw[-] (2,2)--(4,2);
\draw[-,white,line width=2pt] (2.5,1.7)--(3.5,1.1);
\draw[-] (2,2)--(4,0.8);
\draw[-] (4,2)--(4,3);
\draw[-] (4,3)--(5,3);
\draw[-] (5,3)--(6,3);
\draw[-] (4,0.8)--(5,3);
\draw[-] (2,3)--(4,3);
\draw[-] (2,0.8)--(4,0.8);
\draw (2,3.5) node {};
\end{tikzpicture}
  \begin{tikzpicture}[scale=1]
  \draw(3,0.2) node {Two longest paths, $\ell=\#(V(P)\cap V(Q))=9$};
\draw(6,2.3) node {};
\draw[-] (0,3)--(1,3);
\draw[-] (1,3)--(2,3);
\draw[-] (2,0.8)--(1,3);
\draw[-] (2,0.8)--(4,2);
\draw[-] (2,2)--(2,3);
\draw[-] (2,2)--(4,2);
\draw[-] (4,2)--(4,3);
\draw[-] (4,3)--(5,3);
\draw[-] (5,3)--(6,3);
\draw[-] (4,0.8)--(5,3);
\draw[-] (2,3)--(4,3);
\draw[-] (2,0.8)--(4,0.8);
\draw[-,red] (0,2.95)--(1,2.95);
\draw[-,red] (2.03,0.83)--(1.03,3.03);
\draw[-,red] (2,2.05)--(4,2.05);
\draw[-,green] (2,0.85)--(4,2.05);
\draw[-,white,line width=3pt] (2.5,1.68)--(3.5,1.08);
\draw[-] (2,2)--(4,0.8);
\draw[-,red] (2,1.95)--(4,0.75);
\draw[-,red] (3.95,2)--(3.95,3);
\draw[-,red] (4,3.05)--(5,3.05);
\draw[-,red] (5,3.05)--(6,3.05);
\draw[-,red] (2,0.85)--(4,0.85);
\draw[-,green] (0,3.05)--(1,3.05);
\draw[-,green] (1,3.05)--(2,3.05);
\draw[-,green] (2,0.75)--(4,0.75);
\draw[-,green] (1.95,2)--(1.95,3);
\draw[-,green] (2,1.95)--(4,1.95);
\draw[-,green] (5,2.95)--(6,2.95);
\draw[-,green] (4.05,0.8)--(5.05,3);
\filldraw[red]  (4,3)    circle (2pt);
\filldraw [green]  (0,3)    circle (2pt)
[green]  (2,0.8)    circle (2pt)
[green]  (1,3)    circle (2pt)
[green]  (2,2)    circle (2pt)
[green]  (2,3)    circle (2pt)
[green]  (3,2)    circle (2pt)
[green]  (4,2)    circle (2pt)
[green]  (4,0.8)    circle (2pt)
[green]  (5,3)    circle (2pt)
[green]  (6,3)    circle (2pt);
\draw[red,fill=red] (0.05,3.05) arc (45:225:.07cm);
\draw[red,fill=red] (1.05,3.05) arc (45:225:.07cm);
\draw[red,fill=red] (2.05,0.85) arc (45:225:.07cm);
\draw[red,fill=red] (2.05,2.05) arc (45:225:.07cm);
\draw[red,fill=red] (3.05,2.05) arc (45:225:.07cm);
\draw[red,fill=red] (4.05,0.85) arc (45:225:.07cm);
\draw[red,fill=red] (4.05,2.05) arc (45:225:.07cm);
\draw[red,fill=red] (5.05,3.05) arc (45:225:.07cm);
\draw[red,fill=red] (6.05,3.05) arc (45:225:.07cm);
\draw (2,3.5) node {};
\end{tikzpicture}

\noindent Thus $n_{max}\le 10$ and $\ell_{max}\le 8$. Moreover,
we have verified all cases up to $n=10$ and will write down the lengthy computations in a systematic way in a forthcoming article.
This implies that $n_{max}=10$ and by the above result we already know that $5\le \ell_{max} \le 8$.

The article is organized as follows. In the first section we prove Theorem~\ref{resultado caso 1}, which illustrates our method in a simple case.
The theorem says that if the intersections vertices of two longest path
$P$ and $Q$ have the same sequential order in $P$ and in $Q$, then $V(P)\cap V(Q)$ is a separator, thus $P$ and $Q$ cannot
be a counterexample to the Hippchen conjecture.

In the second section we introduce our main new tools:
bi-traceable (BT) graphs, the standard representations $BT(P,Q)$ of these graphs and what we call exterior swap units (ESU)
(see Definition~\ref{def BT}). We also introduce in Definitions~\ref{def NDC}, \ref{def NC} and \ref{def LNC} three different conditions on how you
can connect, or rather on how you cannot connect,
two components of $(P\cup Q)\setminus (V(P)\cap V(Q))$. They can be not directly connectable (NDC), non connectable (NC) or
locally non connectable (LNC).
We also prove that if $V(P)\ne V(Q)$ and if all ESU's in $BT(P,Q)$ are NC, then $V(P)\cap V(Q)$ is a separator.
In the third section we prove that for every graph $BT(P,Q)$ with $\ell\le 4$ all ESU's are NC, which reproves the
Hippchen conjecture for $k=5$. In section 4 we analyze the case $\# V(P)\cap V(Q))=5$ and find that in all but three exceptional cases
all ESU's are NC. In section 5 we prove that in these three cases $V(P)\cap V(Q)$ is a separator, proving the Hippchen conjecture for
$k=6$. In the last section we prove that in none of the three cases there can be a third longest path $R$ with
$P\cap Q\cap R=\emptyset$, which proves that if the intersection of three longest paths $P,Q,R$ is empty, then $\#(V(P)\cap V(Q))\ge 6$.

If one wants to continue our analysis for the case where the intersection consists of 6 points, it will be necessary to make a computer assisted 
analysis of the 720 permutations and the approximately 100 configurations resulting from them, and find the exceptional configurations.
The methods developed in the present paper will be useful in order to treat exceptional cases that will arise for $\ell\ge 6$.
The complete tables of possible configurations can be useful for other problems involving intersections of longest paths.

\section{Preliminaries and an illustrative example}

Along the present paper $G$ is a simple graph, and $P,Q$ are longest paths. By definition, a path cannot pass twice through the same vertex, and the
length of a path is the number of its edges. We will construct a large class of graphs $P\cup Q$, such that in each exterior swap unit
(see Definition~\ref{def ESU})
$P\setminus (V(P)\cap V(Q))$ cannot be connected with $Q\setminus (V(P)\cap V(Q))$ outside $V(P)\cap V(Q)$
(See Proposition~\ref{prop no path exists} and Remark~\ref{large class}).
As usual, removing the vertices also removes the incident edges. In these graphs,
$V(P)\cap V(Q)$ is a separator, if we additionally assume that $V(P)\ne V(Q)$.

We define an operation $'' + ''$ on two paths $P$, $Q$, which is defined only
  if $P$ and $Q$ share exactly one endpoint, and $P+Q$ is simply the union of both paths. If we write $P+Q+R$, then $P$ and $Q$ share one endpoint,
  and the other endpoint of $Q$ is also an endpoint of $R$, and $P+Q+R$ is the union of the three paths.

\begin{notation}\label{notation permutation}
  Along this work we assume that $\# V(P)\cap V(Q)=\ell$ for some $\ell\ge 1$.
  If we write $P=P_0+P_1+\dots+P_\ell$ such that $P_i\cap P_{i-1}=\{a_i\}$, and $\{ a_1,\dots,a_{\ell}\}=V(P)\cap V(Q)$,
  and we write $Q=Q_0+Q_1+\dots+Q_\ell$ such that $Q_i\cap Q_{i-1}=\{b_i\}$, and $\{ b_1,\dots,b_{\ell}\}=V(P)\cap V(Q)$,
  then the two paths $P,Q$ determine a permutation $\sigma$ on $\{1,2,\dots,\ell\}$ given by $b_j=a_{\sigma(j)}$.
\end{notation}

Write $Q_i':=Q_i\setminus (V(P)\cap V(Q))$ and $P_i':=P_i\setminus (V(P)\cap V(Q))$. These are the
components of $P\cup Q \setminus (V(P)\cap V(Q))$.
Note that $P_0'=\emptyset$ if and only if $P_0=\{a_1\}$, i.e., $L(P_0)=0$. The same holds for $Q_0'$, $P_\ell'$ and $Q_\ell'$. On the other hand
if $i\notin \{0,\ell\}$, then $P_i'=\emptyset$ if and only if $L(P_i)=1$ and the same holds for $Q_i'$.
Let
$$
X,Y \in \{P_i',Q_i': P_i'\ne \emptyset, Q_i'\ne \emptyset\}.
$$
A direct connection between $X$ and $Y$ is a path $R$ from $X$ to $Y$ internally disjoint from $V(P)\cup V(Q)$. We write
$X\sim Y$ or $X\sim_R Y$ (Note that this is not an equivalence relation). Clearly
such a path $R$ cannot be part of a graph where $P$ and $Q$ are longest
paths, if there is a path in $P\cup Q\cup R$, which is longer than $P$ (or $Q$).

The case in which $\sigma=Id$, is the prototype of the large class of graphs mentioned above.
Our principal ideas and methods are already present in this case, so we will give a detailed proof and describe an example with $\ell=7$.

\begin{theorem} \label{resultado caso 1}
  Let $G$ be a simple graph, and let $P,Q$ be longest paths in $G$.
  Assume that $a_i=b_i$ for all $i$ (which means that $\sigma=Id$) and that $V(P)\ne V(Q)$. Then the complement of
  $V(P)\cap V(Q)$ cannot be connected.
\end{theorem}

\begin{proof}
  Assume by contradiction that it is connected. Since $V(P)\ne V(Q)$, there exists $i_0$ such that $P_{i_0}'\ne \emptyset$, and it follows that
  $Q_{i_0}'\ne \emptyset$, since $L(P_{i_0})=L(Q_{i_0})$. Consequently there exist
  $$
  X_1,\dots,X_r\in \{P_i',Q_i': P_i'\ne \emptyset, Q_i'\ne \emptyset\},\quad \text{such that $X_i\sim X_{i+1}$ and $X_1=P_{i_0}'$ and $X_r=Q_{i_0}'$}.
  $$ Let
  $\widehat X_i=\begin{cases}
              P_j, & \mbox{if } X_i=Q_j \\
              Q_j, & \mbox{if } X_i=P_j
            \end{cases}$.
  Set
  $$
  j_0=\min\{j>1, \widehat X_j\in\{X_1,\dots,X_{j-1}\}\}.
  $$
  We can assume that $X_j\subset P$ for $j<j_0$. In fact, if $X_j=Q_i'$, we redefine $P$ as $P_{<i}+Q_i+P_{>i}$ and $Q$ as $Q_{<i}+P_i+Q_{>i}$.

  There exists $i_1$ such that $X_{j_0}=Q_{i_1}'$ and so there is a path $R$ from $P_{i_1}'$ to $Q_{i_1}'$
  internally disjoint from $Q$ and from $P_{i_1}$. The endpoints of $R$ split the subpaths
  $Q_{i_1}$ and $P_{i_1}$ into two subpaths each, which we name  $P_{i_1,1}$, $P_{i_1,2}$,$Q_{i_1,1}$,$Q_{i_1,2}$.
  The two paths
  $$
  Q_{<i_1}+Q_{i_1,1}+R+P_{i_1,2}+Q_{>i_1}\quad\text{and}\quad Q_{<i_1}+P_{i_1,1}+R+Q_{i_1,2}+Q_{>i_1}
  $$

  
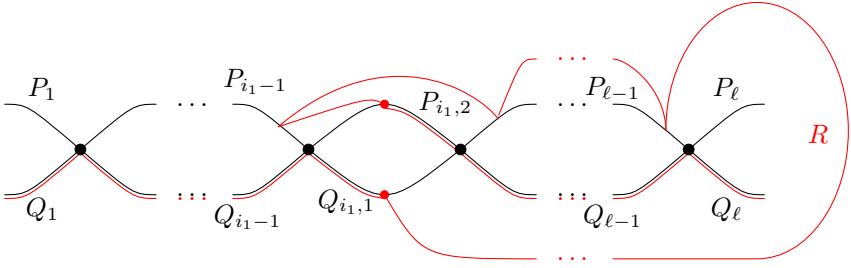
\begin{tikzpicture}[scale=1]
    \draw(5,-0.3) node {};
    \draw[-] (0,2.6)..controls (0.3,2.6)..(1,2);
    \draw[-] (2,2.6)..controls (1.7,2.6)..(1,2);
    \draw[-] (3,2.6)..controls (3.3,2.6)..(4,2);
    \draw[-] (4,2)..controls (5,2.8)..(6,2);
    \draw[-,red] (5,2.55)..controls (5.3,2.5)..(6,1.95);

    \draw[-] (7,2.6)..controls (6.7,2.6)..(6,2);
    \draw[-] (9,2)..controls (8.3,2.6)..(8,2.6);
    \draw[-] (9,2)..controls (9.7,2.6)..(10,2.6);
    \draw[-] (0,1.4)..controls (0.3,1.4)..(1,2);
    \draw[-] (2,1.4)..controls (1.7,1.4)..(1,2);
    \draw[-] (3,1.4)..controls (3.3,1.4)..(4,2);
    \draw[-] (4,2)..controls (5,1.2)..(6,2);
    \draw[-] (7,1.4)..controls (6.7,1.4)..(6,2);
    \draw[-] (9,2)..controls (9.7,1.4)..(10,1.4);
    \draw[-] (9,2)..controls (8.3,1.4)..(8,1.4);
    \draw[-,red] (0,1.35)..controls (0.3,1.35)..(1,1.95);
    \draw[-,red] (2,1.35)..controls (1.7,1.35)..(1,1.95);
    \draw[-,red] (3,1.35)..controls (3.3,1.35)..(4,1.95);
    \draw[-,red] (4,1.95)..controls (4.7,1.4)..(5,1.35);
    \draw[-,red] (7,1.35)..controls (6.7,1.35)..(6,1.95);
    \draw[-,red] (9,1.95)..controls (9.7,1.35)..(10,1.35);
    \draw[-,red] (9,1.95)..controls (8.3,1.35)..(8,1.35);
    \filldraw [black]  (1,2)    circle (2pt)
    [black]  (4,2)    circle (2pt)
    [black]  (6,2)    circle (2pt)
    [black]  (9,2)    circle (2pt);
    \filldraw [red]  (5,1.4)    circle (1.5pt)
    [red]  (5,2.6)    circle (1.5pt);
    \draw (0.5,2.8) node {$ P_1$};
    \draw (3.3,2.9) node {$P_{i_1-1}$};
    \draw (5.8,2.6) node {$P_{i_1,2}$};
    \draw (8,2.8) node {$P_{\ell-1}$};
    \draw (9.5,2.8) node {$ P_\ell$};
    \draw (0.5,1.2) node {$ Q_1$};
    \draw (3.2,1.1) node {$Q_{i_1-1}$};
    \draw (4.5,1.3) node {$Q_{i_1,1}$};
    \draw (8,1.1) node {$Q_{\ell-1}$};
    \draw (9.5,1.2) node {$ Q_\ell$};
    \draw[red] (10.7,2.2) node {$R$};
    \draw (2.5,1.4) node {$\dots$};
    \draw[red] (2.5,1.35) node {$\dots$};
    \draw (2.5,2.6) node {$\dots$};
    \draw (7.5,1.4) node {$\dots$};
    \draw[red] (7.5,1.35) node {$\dots$};
    \draw[red] (7.5,3.2) node {$\dots$};
    \draw (7.5,2.6) node {$\dots$};
    \draw[red] (7.5,0.55) node {$\dots$};
    \draw[-,red] (5,2.6)..controls (4.8,2.7)..(3.6,2.3);
    \draw[-,red] (6.5,2.4)..controls (6.2,2.8)and (5,3.5)..(3.6,2.3);
    \draw[-,red] (6.5,2.45)..controls (6.8,3.2)..(7,3.2);
    \draw[-,red] (8.7,2.25)..controls (8.7,2.8)and (8.2,3.2)..(8,3.2);
    \draw[red] (8.7,2.25) arc(180:(-90):1.2cm and 1.7cm);
    \draw[red] (9.9,0.55)--(8,0.55);
    \draw[red] (5,1.4)..controls (5.5,0.55)..(7,0.55);
  \end{tikzpicture}

  
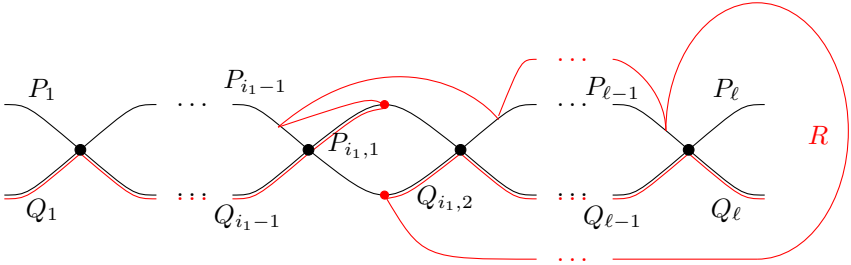
\begin{tikzpicture}[scale=1]
    \draw[-] (0,2.6)..controls (0.3,2.6)..(1,2);
    \draw[-] (2,2.6)..controls (1.7,2.6)..(1,2);
    \draw[-] (3,2.6)..controls (3.3,2.6)..(4,2);
    \draw[-] (4,2)..controls (5,2.8)..(6,2);
    \draw[-,red] (4,1.95)..controls (4.7,2.5)..(5,2.55);

    \draw[-] (7,2.6)..controls (6.7,2.6)..(6,2);
    \draw[-] (9,2)..controls (8.3,2.6)..(8,2.6);
    \draw[-] (9,2)..controls (9.7,2.6)..(10,2.6);
    \draw[-] (0,1.4)..controls (0.3,1.4)..(1,2);
    \draw[-] (2,1.4)..controls (1.7,1.4)..(1,2);
    \draw[-] (3,1.4)..controls (3.3,1.4)..(4,2);
    \draw[-] (4,2)..controls (5,1.2)..(6,2);
    \draw[-] (7,1.4)..controls (6.7,1.4)..(6,2);
    \draw[-] (9,2)..controls (9.7,1.4)..(10,1.4);
    \draw[-] (9,2)..controls (8.3,1.4)..(8,1.4);
    \draw[-,red] (0,1.35)..controls (0.3,1.35)..(1,1.95);
    \draw[-,red] (2,1.35)..controls (1.7,1.35)..(1,1.95);
    \draw[-,red] (3,1.35)..controls (3.3,1.35)..(4,1.95);
    \draw[-,red] (5,1.35)..controls (5.3,1.4)..(6,1.95);
    \draw[-,red] (7,1.35)..controls (6.7,1.35)..(6,1.95);
    \draw[-,red] (9,1.95)..controls (9.7,1.35)..(10,1.35);
    \draw[-,red] (9,1.95)..controls (8.3,1.35)..(8,1.35);
    \filldraw [black]  (1,2)    circle (2pt)
    [black]  (4,2)    circle (2pt)
    [black]  (6,2)    circle (2pt)
    [black]  (9,2)    circle (2pt);
    \filldraw [red]  (5,1.4)    circle (1.5pt)
    [red]  (5,2.6)    circle (1.5pt);
    \draw (0.5,2.8) node {$ P_1$};
    \draw (3.3,2.9) node {$P_{i_1-1}$};
    \draw (4.6,2.05) node {$P_{i_1,1}$};
    \draw (5.8,1.35) node {$Q_{i_1,2}$};
    \draw (8,2.8) node {$P_{\ell-1}$};
    \draw (9.5,2.8) node {$ P_{\ell}$};
    \draw (0.5,1.2) node {$ Q_1$};
    \draw (3.2,1.1) node {$Q_{i_1-1}$};
    \draw (8,1.1) node {$Q_{\ell-1}$};
    \draw (9.5,1.2) node {$ Q_{\ell}$};
    \draw[red] (10.7,2.2) node {$R$};
    \draw (2.5,1.4) node {$\dots$};
    \draw[red] (2.5,1.35) node {$\dots$};
    \draw (2.5,2.6) node {$\dots$};
    \draw (7.5,1.4) node {$\dots$};
    \draw[red] (7.5,1.35) node {$\dots$};
    \draw[red] (7.5,3.2) node {$\dots$};
    \draw (7.5,2.6) node {$\dots$};
    \draw[red] (7.5,0.55) node {$\dots$};
    \draw[-,red] (5,2.6)..controls (4.8,2.7)..(3.6,2.3);
    \draw[-,red] (6.5,2.4)..controls (6.2,2.8)and (5,3.5)..(3.6,2.3);
    \draw[-,red] (6.5,2.45)..controls (6.8,3.2)..(7,3.2);
    \draw[-,red] (8.7,2.25)..controls (8.7,2.8)and (8.2,3.2)..(8,3.2);
    \draw[red] (8.7,2.25) arc(180:(-90):1.2cm and 1.7cm);
    \draw[red] (9.9,0.55)--(8,0.55);
    \draw[red] (5,1.4)..controls (5.5,0.55)..(7,0.55);
  \end{tikzpicture}

  \noindent  have lengths that sum $2L(Q)+2L(R)$, a contradiction that concludes the proof.
\end{proof}

\begin{example}
  We will illustrate the proof of the theorem in an example with $\ell=7$. The green path is $P$, the red path is $Q$ and the blue paths are
   in $G\setminus (P\cup Q)$, with endpoints in $P\cup Q\setminus (V(P)\cap V(Q))$.

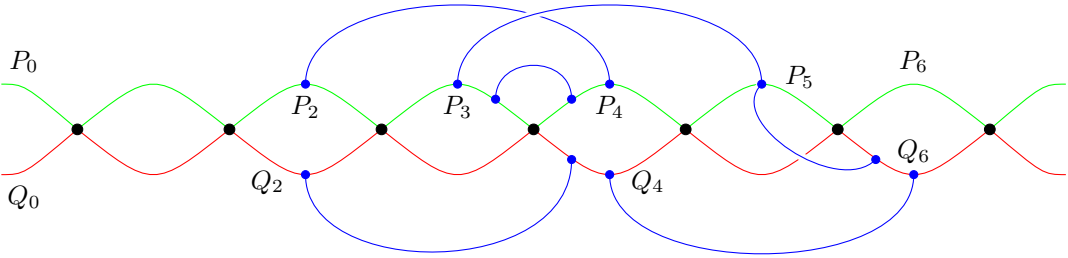
  \begin{tikzpicture}[scale=1]
    \draw[-,green] (1,2.6)..controls (1.3,2.6)..(2,2);
    \draw[-,green] (2,2)..controls (3,2.8)..(4,2);
    \draw[-,green] (4,2)..controls (5,2.8)..(6,2);
    \draw[-,green] (6,2)..controls (7,2.8)..(8,2);
    \draw[-,green] (8,2)..controls (9,2.8)..(10,2);
    \draw[-,green] (10,2)..controls (11,2.8)..(12,2);
    \draw[-,green] (12,2)..controls (13,2.8)..(14,2);
    \draw[-,green] (14,2)..controls (14.7,2.6)..(15,2.6);
    \draw[-,red] (1,1.4)..controls (1.3,1.4)..(2,2);
    \draw[-,red] (2,2)..controls (3,1.2)..(4,2);
    \draw[-,red] (4,2)..controls (5,1.2)..(6,2);
    \draw[-,red] (6,2)..controls (7,1.2)..(8,2);
    \draw[-,red] (8,2)..controls (9,1.2)..(10,2);
    \draw[-,red] (10,2)..controls (11,1.2)..(12,2);
    \draw[-,red] (12,2)..controls (13,1.2)..(14,2);
    \draw[-,red] (14,2)..controls (14.7,1.4)..(15,1.4);
    \draw[-,blue] (5,2.6)..controls (5,4)and(9,4)..(9,2.6);
    \draw[-,white, line width=3pt] (7,2.6)..controls (7,4)and(11,4)..(11,2.7);
    \draw[-,blue] (7,2.6)..controls (7,4)and(11,4)..(11,2.6);
    \draw[-,blue] (7.5,2.4)..controls (7.5,3)and(8.5,3)..(8.5,2.4);
    \draw[-,blue] (5,1.4)..controls (5,0)and(8.5,0)..(8.5,1.6);
    \draw[-,blue] (9,1.4)..controls (9,0)and(13,0)..(13,1.4);
    \draw[-,white,line width=3pt] (11,2.6)..controls (10.5,2.1)and(12,1.1)..(12.5,1.6);
    \draw[-,blue] (11,2.6)..controls (10.5,2.1)and(12,1.1)..(12.5,1.6);
    \filldraw [blue]  (5,2.6)    circle (1.5pt)
    [blue]  (5,1.4)    circle (1.5pt)
    [blue]  (7,2.6)    circle (1.5pt)
    [blue]  (7.5,2.4)    circle (1.5pt)
    [blue]  (8.5,2.4)    circle (1.5pt)
    [blue]  (9,2.6)    circle (1.5pt)
    [blue]  (8.5,1.6)    circle (1.5pt)
    [blue]  (9,1.4)    circle (1.5pt)
    [blue]  (11,2.6)    circle (1.5pt)
    [blue]  (12.5,1.6)    circle (1.5pt)
    [blue]  (13,1.4)    circle (1.5pt);
    \filldraw [black]  (2,2)    circle (2pt)
    [black]  (4,2)    circle (2pt)
    [black]  (6,2)    circle (2pt)
    [black]  (8,2)    circle (2pt)
    [black]  (10,2)    circle (2pt)
    [black]  (12,2)    circle (2pt)
    [black]  (14,2)    circle (2pt);
    \draw(1.3,2.9)node{$P_0$};
    \draw(1.3,1.1)node{$Q_0$};
    \draw(5,2.3)node{$P_2$};
    \draw(7,2.3)node{$P_3$};
    \draw(9,2.3)node{$P_4$};
    \draw(11.5,2.7)node{$P_5$};
    \draw(13,2.9)node{$P_6$};
    \draw(13,1.7)node{$Q_6$};
    \draw(4.5,1.3)node{$Q_2$};
    \draw(9.5,1.3)node{$Q_4$};
  \end{tikzpicture}

\noindent  In this example $i_0=2$, since the additional blue paths connect $P_2$ with $Q_2$. Moreover
  $$
    X_1=P_2',\quad X_2=P_4',\quad X_3=P_3',\quad X_4=P_5',\quad X_5=Q_6',\quad X_6=Q_4',\quad X_7=Q_2',
  $$
  and so
  $$
    \widehat{X}_1=Q_2',\quad \widehat{X}_2=Q_4',\quad \widehat{X}_3=Q_3',\quad \widehat{X}_4=Q_5',\quad \widehat{X}_5=P_6',\quad
    \widehat{X}_6=P_4',\quad \widehat{X}_7=P_2'.
  $$
  We have $j_0=6$, since
  $$
    \widehat{X}_2\notin \{X_1\},\quad \widehat{X}_3\notin \{X_1,X_2\},\quad \widehat{X}_4\notin \{X_1,X_2,X_3\},\quad
    \widehat{X}_5\notin \{X_1,X_2,X_3,X_4\},
  $$
  but
  $$
    \widehat X_6=P_4'=X_2\in \{X_1,X_2,X_3,X_4,X_5\}.
  $$
  Note that we also have $\widehat X_7=P_2'=X_1\in \{X_1,X_2,X_3,X_4,X_5,X_6\}$.
   Since $X_5=Q_6'$, we redefine $P$ and $Q$, swapping $Q_6$ with $P_6$.

    
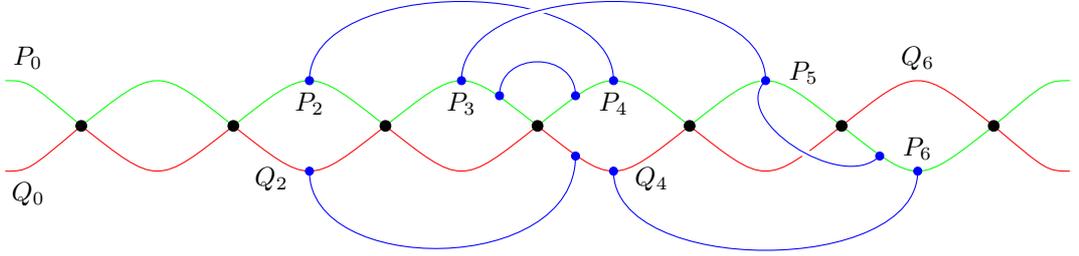
\begin{tikzpicture}[scale=1]
    \draw[-,green] (1,2.6)..controls (1.3,2.6)..(2,2);
    \draw[-,green] (2,2)..controls (3,2.8)..(4,2);
    \draw[-,green] (4,2)..controls (5,2.8)..(6,2);
    \draw[-,green] (6,2)..controls (7,2.8)..(8,2);
    \draw[-,green] (8,2)..controls (9,2.8)..(10,2);
    \draw[-,green] (10,2)..controls (11,2.8)..(12,2);
    \draw[-,red] (12,2)..controls (13,2.8)..(14,2);
    \draw[-,green] (14,2)..controls (14.7,2.6)..(15,2.6);
    \draw[-,red] (1,1.4)..controls (1.3,1.4)..(2,2);
    \draw[-,red] (2,2)..controls (3,1.2)..(4,2);
    \draw[-,red] (4,2)..controls (5,1.2)..(6,2);
    \draw[-,red] (6,2)..controls (7,1.2)..(8,2);
    \draw[-,red] (8,2)..controls (9,1.2)..(10,2);
    \draw[-,red] (10,2)..controls (11,1.2)..(12,2);
    \draw[-,green] (12,2)..controls (13,1.2)..(14,2);
    \draw[-,red] (14,2)..controls (14.7,1.4)..(15,1.4);
    \draw[-,blue] (5,2.6)..controls (5,4)and(9,4)..(9,2.6);
    \draw[-,white, line width=3pt] (7,2.6)..controls (7,4)and(11,4)..(11,2.7);
    \draw[-,blue] (7,2.6)..controls (7,4)and(11,4)..(11,2.6);
    \draw[-,blue] (7.5,2.4)..controls (7.5,3)and(8.5,3)..(8.5,2.4);
    \draw[-,blue] (5,1.4)..controls (5,0)and(8.5,0)..(8.5,1.6);
    \draw[-,blue] (9,1.4)..controls (9,0)and(13,0)..(13,1.4);
    \draw[-,white,line width=3pt] (11,2.6)..controls (10.5,2.1)and(12,1.1)..(12.5,1.6);
    \draw[-,blue] (11,2.6)..controls (10.5,2.1)and(12,1.1)..(12.5,1.6);
    \filldraw [blue]  (5,2.6)    circle (1.5pt)
    [blue]  (5,1.4)    circle (1.5pt)
    [blue]  (7,2.6)    circle (1.5pt)
    [blue]  (7.5,2.4)    circle (1.5pt)
    [blue]  (8.5,2.4)    circle (1.5pt)
    [blue]  (9,2.6)    circle (1.5pt)
    [blue]  (8.5,1.6)    circle (1.5pt)
    [blue]  (9,1.4)    circle (1.5pt)
    [blue]  (11,2.6)    circle (1.5pt)
    [blue]  (12.5,1.6)    circle (1.5pt)
    [blue]  (13,1.4)    circle (1.5pt);
    \filldraw [black]  (2,2)    circle (2pt)
    [black]  (4,2)    circle (2pt)
    [black]  (6,2)    circle (2pt)
    [black]  (8,2)    circle (2pt)
    [black]  (10,2)    circle (2pt)
    [black]  (12,2)    circle (2pt)
    [black]  (14,2)    circle (2pt);
    \draw(1.3,2.9)node{$P_0$};
    \draw(1.3,1.1)node{$Q_0$};
    \draw(5,2.3)node{$P_2$};
    \draw(7,2.3)node{$P_3$};
    \draw(9,2.3)node{$P_4$};
    \draw(11.5,2.7)node{$P_5$};
    \draw(13,2.9)node{$Q_6$};
    \draw(13,1.7)node{$P_6$};
    \draw(4.5,1.3)node{$Q_2$};
    \draw(9.5,1.3)node{$Q_4$};
  \end{tikzpicture}

  Then $i_1=4$, since $X_{j_0}=X_6=Q_4'$, and we arrive at the following situation, where $P_{i_1}'=P_4'$ is connected with $Q_{i_1}'=Q_4'$
  via a path $R$ that is internally disjoint from $Q$ and from $P_4\cup Q_4$.

  
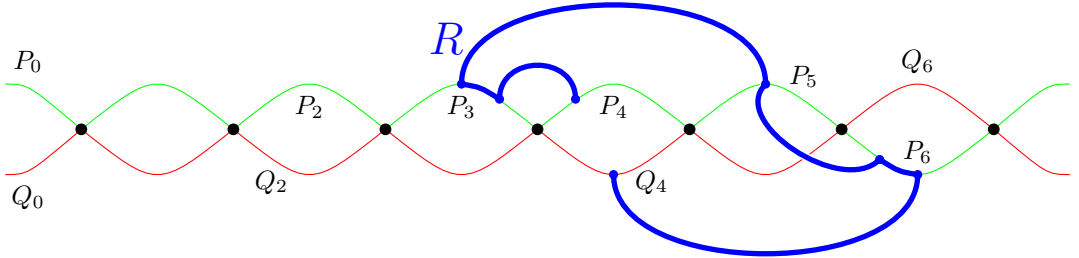
\begin{tikzpicture}[scale=1]
    \draw[-,green] (1,2.6)..controls (1.3,2.6)..(2,2);
    \draw[-,green] (2,2)..controls (3,2.8)..(4,2);
    \draw[-,green] (4,2)..controls (5,2.8)..(6,2);
    \draw[-,green] (6,2)..controls (7,2.8)..(8,2);
    \draw[-,green] (8,2)..controls (9,2.8)..(10,2);
    \draw[-,green] (10,2)..controls (11,2.8)..(12,2);
    \draw[-,red] (12,2)..controls (13,2.8)..(14,2);
    \draw[-,green] (14,2)..controls (14.7,2.6)..(15,2.6);
    \draw[-,red] (1,1.4)..controls (1.3,1.4)..(2,2);
    \draw[-,red] (2,2)..controls (3,1.2)..(4,2);
    \draw[-,red] (4,2)..controls (5,1.2)..(6,2);
    \draw[-,red] (6,2)..controls (7,1.2)..(8,2);
    \draw[-,red] (8,2)..controls (9,1.2)..(10,2);
    \draw[-,red] (10,2)..controls (11,1.2)..(12,2);
    \draw[-,green] (12,2)..controls (13,1.2)..(14,2);
    \draw[-,red] (14,2)..controls (14.7,1.4)..(15,1.4);
    \draw[-,white, line width=3pt] (7,2.6)..controls (7,4)and(11,4)..(11,2.7);
    \draw[-,blue,line width=2pt] (7,2.6)..controls (7,4)and(11,4)..(11,2.6);
    \draw[-,blue,line width=2pt] (7.5,2.4)..controls (7.5,3)and(8.5,3)..(8.5,2.4);
    \draw[-,blue,line width=2pt] (9,1.4)..controls (9,0)and(13,0)..(13,1.4);
    \draw[-,white,line width=3pt] (11,2.6)..controls (10.5,2.1)and(12,1.1)..(12.5,1.6);
    \draw[-,blue,line width=2pt] (11,2.6)..controls (10.5,2.1)and(12,1.1)..(12.5,1.6);
    \draw[-,blue,line width=2pt] (7,2.6)..controls (7.25,2.56)..(7.5,2.4);
    \draw[-,blue,line width=2pt] (12.5,1.6)..controls (12.75,1.44)..(13,1.4);
    \filldraw     [blue]  (7,2.6)    circle (1.5pt)
    [blue]  (7.5,2.4)    circle (1.5pt)
    [blue]  (8.5,2.4)    circle (1.5pt)
    [blue]  (9,1.4)    circle (1.5pt)
    [blue]  (11,2.6)    circle (1.5pt)
    [blue]  (12.5,1.6)    circle (1.5pt)
    [blue]  (13,1.4)    circle (1.5pt);
    \filldraw [black]  (2,2)    circle (2pt)
    [black]  (4,2)    circle (2pt)
    [black]  (6,2)    circle (2pt)
    [black]  (8,2)    circle (2pt)
    [black]  (10,2)    circle (2pt)
    [black]  (12,2)    circle (2pt)
    [black]  (14,2)    circle (2pt);
    \draw(1.3,2.9)node{$P_0$};
    \draw(1.3,1.1)node{$Q_0$};
    \draw(5,2.3)node{$P_2$};
    \draw(7,2.3)node{$P_3$};
    \draw(9,2.3)node{$P_4$};
    \draw(11.5,2.7)node{$P_5$};
    \draw(13,2.9)node{$Q_6$};
    \draw(13,1.7)node{$P_6$};
    \draw(4.5,1.3)node{$Q_2$};
    \draw(9.5,1.3)node{$Q_4$};
    \draw[blue](6.8,3.2)node{{\huge\bf{$R$}}};
  \end{tikzpicture}

  As in the proof of the theorem, now we can construct the two paths,
  whose length sum $2L(Q)+2L(R)$.
\end{example}

We want to generalize the theorem in order to apply to other permutations, and not only for $\sigma=Id$.
Note that in the proof of the theorem we use the following properties of the pair $\{P_i,Q_i\}$:
\begin{itemize}
  \item $L(P_i)=L(Q_i)$,
  \item They share both endpoints, when $1<i<\ell$, and one endpoint when $i\in \{0,\ell\}$.
  \item Thus they can be interchanged, i.e.,  $P_{<i}+Q_i+P_{>i}$ and $Q_{<i}+P_i+Q_{>i}$ are two longest paths having the same intersection and 
  the same  union as $P$ and $Q$.
  \item An internally disjoint path $R$ between $P_i'$ and $Q_i'$ when $0<i<\ell$, enables to construct the two paths
  $P_{i,1}+R+Q_{i,2}$ and $Q_{i,1}+R+P_{i,2}$ such that both paths connect the endpoints of $P_i$.
  Moreover, the lengths of the new paths sum $2L(P_i)+2R>L(Q_i)+L(P_i)$.
\end{itemize}

\section{Bi-traceable graphs}

We want to consider the graph which is the union of the two longest paths $P$ and $Q$. In this section we assume the same notations as in
Notation~\ref{notation permutation}.
\begin{definition}\label{def BT}
  \begin{enumerate}
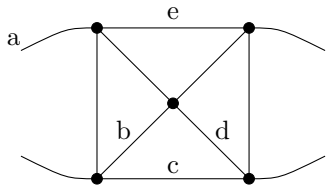
    \item A \textbf{bi-traceable graph (BT-graph)} is a simple graph that has two longest paths such that the union is the whole graph.
    \item The \textbf{representation of a bi-traceable graph associated to $P$ and $Q$} is a two colored graph $BT(P,Q)$,
      obtained from the BT-graph $P\cup Q$ by the following process.
      First we assign to all the edges of $P$ one color. Then the edges of $Q$ are colored with the other color. If the paths
      share an edge, then this edge is duplicated, and each copy is colored with one of the colors. The vertices are not colored.

\noindent  
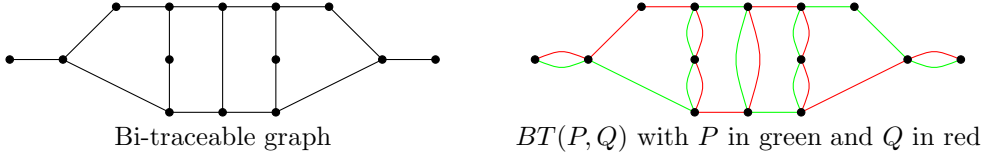
\begin{tikzpicture}[scale=0.7]
\draw[-] (0,1)--(1,1);
\draw[-] (1,1)--(2,2);
\draw[-] (2,2)--(3,2);
\draw[-] (3,0)--(3,2);
\draw[-] (4,0)--(4,2);
\draw[-] (5,0)--(5,2);
\draw[-] (5,2)--(6,2);
\draw[-] (6,2)--(7,1);
\draw[-] (7,1)--(8,1);
\draw[-] (3,2)--(5,2);
\draw[-] (3,0)--(5,0);
\draw[-] (1,1)--(3,0);
\draw[-] (5,0)--(7,1);
\filldraw [black]  (0,1)    circle (2pt)
[black]  (1,1)    circle (2pt)
[black]  (2,2)    circle (2pt)
[black]  (3,0)    circle (2pt)
[black]  (3,1)    circle (2pt)
[black]  (3,2)    circle (2pt)
[black]  (4,0)    circle (2pt)
[black]  (4,2)    circle (2pt)
[black]  (5,0)    circle (2pt)
[black]  (5,1)    circle (2pt)
[black]  (5,2)    circle (2pt)
[black]  (6,2)    circle (2pt)
[black]  (7,1)    circle (2pt)
[black]  (8,1)    circle (2pt);
\draw (4,-0.5) node {Bi-traceable graph};
\draw (3.4,2.8) node {};
\draw (9,-0.8) node {};
\end{tikzpicture}
  \begin{tikzpicture}[scale=0.7]
\draw[-,green] (0,1)..controls (0.5,0.8)..(1,1);
\draw[-,red] (0,1)..controls (0.5,1.2)..(1,1);
\draw[-,red] (1,1)--(2,2);
\draw[-,red] (2,2)--(3,2);
\draw[-,green] (3,0)..controls (2.8,0.5)..(3,1);
\draw[-,green] (3,1)..controls (2.8,1.5)..(3,2);
\draw[-,red] (3,0)..controls (3.2,0.5)..(3,1);
\draw[-,red] (3,1)..controls (3.2,1.5)..(3,2);
\draw[-,green] (4,0)..controls (3.7,1)..(4,2);
\draw[-,red] (4,0)..controls (4.3,1)..(4,2);
\draw[-,green] (5,0)..controls (4.8,0.5)..(5,1);
\draw[-,green] (5,1)..controls (4.8,1.5)..(5,2);
\draw[-,red] (5,0)..controls (5.2,0.5)..(5,1);
\draw[-,red] (5,1)..controls (5.2,1.5)..(5,2);
\draw[-,green] (7,1)..controls (7.5,0.8)..(8,1);
\draw[-,red] (7,1)..controls (7.5,1.2)..(8,1);
\draw[-,green] (5,2)--(6,2);
\draw[-,green] (6,2)--(7,1);
\draw[-,green] (3,2)--(4,2);
\draw[-,red] (4,2)--(5,2);
\draw[-,red] (3,0)--(4,0);
\draw[-,green] (4,0)--(5,0);
\draw[-,green] (1,1)--(3,0);
\draw[-,red] (5,0)--(7,1);
\filldraw [black] (1,1)    circle (2pt)
[black]  (2,2)    circle (2pt)
[black]  (0,1)    circle (2pt)
[black]  (8,1)    circle (2pt)
[black]  (3,0)    circle (2pt)
[black]  (3,1)    circle (2pt)
[black]  (3,2)    circle (2pt)
[black]  (4,0)    circle (2pt)
[black]  (4,2)    circle (2pt)
[black]  (5,0)    circle (2pt)
[black]  (5,1)    circle (2pt)
[black]  (5,2)    circle (2pt)
[black]  (6,2)    circle (2pt)
[black]  (7,1)    circle (2pt);
\draw (4,-0.5) node {$BT(P,Q)$ with $P$ in green and $Q$ in red};
\draw (3.4,-0.8) node {};
\end{tikzpicture}

     \noindent  Note that the resulting graph is no longer simple, but $P$ and $Q$ are still longest paths in it. We also have $BT(P,Q)=BT(Q,P)$.
    \item The subpaths $P_i$, $Q_i$ are called the \textbf{completed components} of $BT(P,Q)$, or simply the completed components, if $P$ and $Q$
        are clear
        from the context. Note that each of the extremal completed components $P_0$, $Q_0$, $P_{\ell}$ and $Q_{\ell}$ may be only a point.
  \end{enumerate}
\end{definition}

Note that some components $P_i'$ or $Q_i'$ might be empty.

\begin{definition} \label{def NDC}
  Two components $X,Y$ are said to be \textbf{not directly connectable (NDC)}, if one of them is empty, or if $X\sim_R Y$ implies that
  there exist two paths $\widehat{P}$ and $\widehat{Q}$ in $BT(P,Q)\cup R$ such that
  $$
  L(\widehat{P})+L(\widehat{Q})=2L(R)+ L(P)+L(Q)=2 L(P)+2 L(R).
  $$
\end{definition}

The following result is one of the main ingredients in the proof of the Hippchen conjecture for $k=4$ in~\cite{G}, and the idea was already present
in~\cite{H}*{Lemma 2.2.3}.

\begin{lemma}[\cite{G}*{Lemma 4.1}] \label{lema 4.1}
  If two completed components of different colors are adjacent (which means that they
  have a common vertex), then they are NDC.
\end{lemma}

\begin{proof}
  See the diagram, or see also~\cite{G}*{Lemma 4.1}.

    
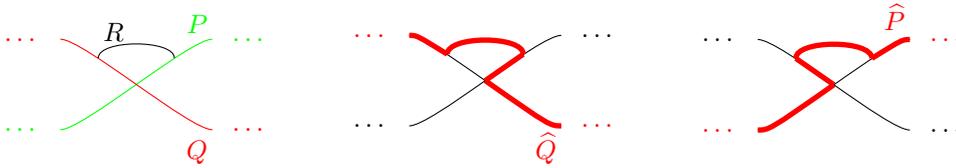
\begin{tikzpicture}[scale=1]
\draw (6.5,3.2) node{};
\draw[-,green] (3,1.4)..controls (3.3,1.4)and (4.7,2.6)..(5,2.6);
\draw[-,red] (3,2.6)..controls (3.3,2.6)and (4.7,1.4)..(5,1.4);
\draw[-] (3.5,2.35)..controls (3.5,2.6)and (4.5,2.6)..(4.5,2.35);
\draw[red] (4.8,1.1) node {$Q$};
\draw[green] (4.8,2.8) node {$ P$};
\draw (3.7,2.7) node {$R$};
\draw[green] (2.5,1.4) node {$\dots$};
\draw[red] (2.5,2.6) node {$\dots$};
\draw[red] (5.5,1.4) node {$\dots$};
\draw[green] (5.5,2.6) node {$\dots$};
\end{tikzpicture}
    \begin{tikzpicture}[scale=1]
\draw (6.5,3.2) node{};
\draw[-] (3,1.4)..controls (3.3,1.4)and (4.7,2.6)..(5,2.6);
\draw[-] (3,2.6)..controls (3.3,2.6)and (4.7,1.4)..(5,1.4);
\draw[-,line width=2pt,red] (3.5,2.35)..controls (3.5,2.6)and (4.5,2.6)..(4.5,2.35);
\draw[-,line width=2pt,red] (4.5,2.35)--(4,2);
\draw[-,line width=2pt,red] (3.5,2.35)..controls  (3.1,2.6)..(3,2.6);
\draw[-,line width=2pt,red] (4,2)..controls (4.85,1.4) ..(5,1.4);
\draw[red] (4.8,1.1) node {$\widehat Q$};
\draw (2.5,1.4) node {$\dots$};
\draw[red] (2.5,2.6) node {$\dots$};
\draw[red] (5.5,1.4) node {$\dots$};
\draw (5.5,2.6) node {$\dots$};
\end{tikzpicture}
    \begin{tikzpicture}[scale=1]
\draw[-] (3,1.4)..controls (3.3,1.4)and (4.7,2.6)..(5,2.6);
\draw[-] (3,2.6)..controls (3.3,2.6)and (4.7,1.4)..(5,1.4);
\draw[-,line width=2pt,red] (3.5,2.35)..controls (3.5,2.6)and (4.5,2.6)..(4.5,2.35);
\draw[-,line width=2pt,red] (3.5,2.35)--(4,2);
\draw[-,line width=2pt,red] (4.5,2.35)..controls  (4.9,2.6)..(5,2.6);
\draw[-,line width=2pt,red] (4,2)..controls (3.15,1.4) ..(3,1.4);
\draw[red] (4.8,2.9) node {$\widehat{P}$};
\draw[red] (2.5,1.4) node {$\dots$};
\draw (2.5,2.6) node {$\dots$};
\draw (5.5,1.4) node {$\dots$};
\draw[red] (5.5,2.6) node {$\dots$};
\draw[red] (5.5,0.95) node {};

\end{tikzpicture}
\end{proof}

\begin{lemma}\label{conexion entre extremos}
  Two extremal components of different colors are NDC. The extremal components
   are $P_0',P_\ell',Q_0',Q_\ell'$.
\end{lemma}

\begin{proof}
  Assume for example $P_0'\sim_R Q_\ell'$.
  Let $P_{0,1}+P_{0,2}=P_0$  such that $P_{0,1}\cap P_{0,2}$ is one endpoint of $R$ and such that $P_{0,1}$ and $P$ share one endpoint. Similarly
   define $Q_{\ell,1}$ and $Q_{\ell,2}$, with $Q_{\ell,2}$ having a common endpoint with $Q$.
  Then
  $$
    \widehat{P}=Q_{\ell,2}+R+P_{0,2}+P_1+\dots+P_\ell \quad\text{and}\quad \widehat{Q}= P_{0,1}+R+Q_{\ell,1}+Q_{\ell-1}+\dots+Q_1+Q_0
  $$
  
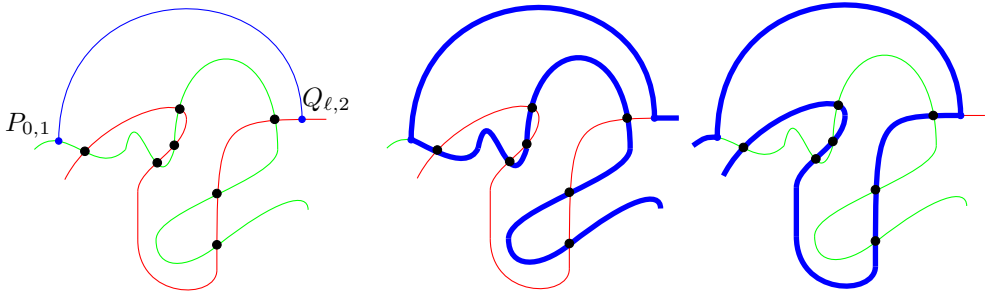
\begin{tikzpicture}[scale=0.8]
     \draw[-,green] (0,2)..controls (0.5,2.5)and (1,1.5)..(1.5,2);
     \draw[-,green] (1.5,2)..controls (1.7,3)and (2,1)..(2.3,2);
     \draw[-,green] (2.3,2)..controls (2.3,4)and (4,4)..(4,2);
     \draw[-,green] (4,2)..controls (4,1.5)and (2,1)..(2,0.5);
     \draw[-,green] (2,0.5)..controls (2,0)and (2.5,0)..(3,0.4);
     \draw[-,green] (3,0.4)..controls (4,1.2)and (4.5,1.3)..(4.5,1);
     \draw[-,red] (0.5,1.5)..controls (1,2.5)and (2.5,3)..(2.5,2.5);
     \draw[-,red] (2.5,2.5)..controls (2.5,2)and (1.7,1.7)..(1.7,1.3);
     \draw[-,red] (1.7,1.3)--(1.7,0.5);
     \draw[-,red] (1.7,0.5)..controls (1.7,-0.5)and (3,-0.5)..(3,0);
     \draw[-,red] (3,0)..controls (3,2.5)..(4.8,2.5);
     \draw[-,blue] (0.4,2.14)..controls (0.4,5)and(4.4,5)..(4.4,2.5);
     \filldraw [blue]  (0.4,2.14)    circle (1.5pt)
    [blue]  (4.4,2.5)    circle (1.5pt);
\filldraw [black]  (0.83,1.97)    circle (2pt)
    [black]  (2.39,2.67)    circle (2pt)
    [black]  (2.3,2.07)    circle (2pt)
    [black]  (2.02,1.78)    circle (2pt)
    [black]  (3,0.42)    circle (2pt)
    [black]  (3,1.27)    circle (2pt)
    [black]  (3.95,2.5)    circle (2pt);
    \draw(-0.1,2.4)node{$P_{0,1}$};
    \draw(4.8,2.8)node{$Q_{\ell,2}$};
    \draw(2.5,-1)node{$P$ in green, $Q$ in red, $R$ in blue};
  \end{tikzpicture}
  \begin{tikzpicture}[scale=0.8]
     \draw[-,green] (0,2)..controls (0.5,2.5)and (1,1.5)..(1.5,2);
     \draw[-,blue,line width=2pt] (0.4,2.14)..controls (0.5,2.14)and (1.3,1.55)..(1.5,2);
     \draw[-,blue,line width=2pt] (1.5,2)..controls (1.7,3)and (2,1)..(2.3,2);
     \draw[-,blue,line width=2pt] (2.3,2)..controls (2.3,4)and (4,4)..(4,2);
     \draw[-,blue,line width=2pt] (4,2)..controls (4,1.5)and (2,1)..(2,0.5);
     \draw[-,blue,line width=2pt] (2,0.5)..controls (2,0)and (2.5,0)..(3,0.4);
     \draw[-,blue,line width=2pt] (3,0.4)..controls (4,1.2)and (4.5,1.3)..(4.5,1);
     \draw[-,red] (0.5,1.5)..controls (1,2.5)and (2.5,3)..(2.5,2.5);
     \draw[-,red] (2.5,2.5)..controls (2.5,2)and (1.7,1.7)..(1.7,1.3);
     \draw[-,red] (1.7,1.3)--(1.7,0.5);
     \draw[-,red] (1.7,0.5)..controls (1.7,-0.5)and (3,-0.5)..(3,0);
     \draw[-,red] (3,0)..controls (3,2.5)..(4.8,2.5);
     \draw[-,blue,line width=2pt] (4.4,2.5)--(4.8,2.5);
     \draw[-,blue,line width=2pt] (0.4,2.14)..controls (0.4,5)and(4.4,5)..(4.4,2.5);
     \filldraw [blue]  (0.4,2.14)    circle (1.5pt)
    [blue]  (4.4,2.5)    circle (1.5pt);
\filldraw [black]  (0.83,1.97)    circle (2pt)
    [black]  (2.39,2.67)    circle (2pt)
    [black]  (2.3,2.07)    circle (2pt)
    [black]  (2.02,1.78)    circle (2pt)
    [black]  (3,0.42)    circle (2pt)
    [black]  (3,1.27)    circle (2pt)
    [black]  (3.95,2.5)    circle (2pt);
    \draw(2.5,-1)node{$\widehat P$ in blue};
  \end{tikzpicture}
    \begin{tikzpicture}[scale=0.8]
     \draw[-,green] (0,2)..controls (0.5,2.5)and (1,1.5)..(1.5,2);
     \draw[-,blue,line width=2pt] (0,2)..controls (0.2,2.14)..(0.4,2.14);
     \draw[-,green] (1.5,2)..controls (1.7,3)and (2,1)..(2.3,2);
     \draw[-,green] (2.3,2)..controls (2.3,4)and (4,4)..(4,2);
     \draw[-,green] (4,2)..controls (4,1.5)and (2,1)..(2,0.5);
     \draw[-,green] (2,0.5)..controls (2,0)and (2.5,0)..(3,0.4);
     \draw[-,green] (3,0.4)..controls (4,1.2)and (4.5,1.3)..(4.5,1);
     \draw[-,blue,line width=2pt] (0.5,1.5)..controls (1,2.5)and (2.5,3)..(2.5,2.5);
     \draw[-,blue,line width=2pt] (2.5,2.5)..controls (2.5,2)and (1.7,1.7)..(1.7,1.3);
     \draw[-,blue,line width=2pt] (1.7,1.3)--(1.7,0.5);
     \draw[-,blue,line width=2pt] (1.7,0.5)..controls (1.7,-0.5)and (3,-0.5)..(3,0);
     \draw[-,blue,line width=2pt] (3,0)..controls (3,2.5)..(4.4,2.5);
     \draw[-,red] (4.4,2.5)--(4.8,2.5);
     \draw[-,blue,line width=2pt] (0.4,2.14)..controls (0.4,5)and(4.4,5)..(4.4,2.5);
     \filldraw [blue]  (0.4,2.14)    circle (1.5pt)
    [blue]  (4.4,2.5)    circle (1.5pt);
\filldraw [black]  (0.83,1.97)    circle (2pt)
    [black]  (2.39,2.67)    circle (2pt)
    [black]  (2.3,2.07)    circle (2pt)
    [black]  (2.02,1.78)    circle (2pt)
    [black]  (3,0.42)    circle (2pt)
    [black]  (3,1.27)    circle (2pt)
    [black]  (3.95,2.5)    circle (2pt);
    \draw(2.5,-1)node{The path $\widehat Q$ in blue};
  \end{tikzpicture}

\noindent  are two paths whose lengths sum $2L(P)+2L(R)$, as desired.
\end{proof}

Let $G$ be a graph, $P$ and $Q$ be longest paths. Assume that $BT(P,Q)$ has at least two non empty components of different colors.
If all pairs of different colors in $BT(P,Q)$ are NDC, then $V(P)\cap V(Q)$ is a separator of $G$.
Verifying this condition for all pairs of different colors in small graphs is a
manageable task, but in order to obtain results generalizing Theorem~\ref{resultado caso 1}, we need a notion of ``connectable'' with a broader scope.
For this we formalize the process of swapping colors of some completed components in the proof of Theorem~\ref{resultado caso 1}, and we analyze the
block structure of $BT(P,Q)$.

\begin{definition}
   A \textbf{swap unit in $BT(P,Q)$} is a set of completed components
   of two colors, such that if we swap the colors in the given set of completed components,
   then we obtain a new representation $BT(\widetilde P,\widetilde Q)$ of the original BT-graph.
\end{definition}
\begin{definition}
  An \textbf{internal building block (IBB)} of $BT(P,Q)$ is a 2-connected subgraph which is the union of two subpaths $\widetilde P$ of $P$ and
  $\widetilde Q$ of $Q$,
  such that $\widetilde P$ and $\widetilde Q$ have the same endpoints, and don't intersect other subpaths of $P$ and $Q$
  other than in these endpoints, which are called the endpoints of the IBB.

  An \textbf{elementary IBB} is the union of two completed components that share both endpoints.
\end{definition}

For example, if the $P$ and $Q$ share one edge, then the duplicated edge is an elementary IBB.

\noindent  \begin{tikzpicture}[scale=1]
\draw[-,green] (0,1)--(1,2);
\draw[-,green] (1,2)--(2,2);
\draw[-,green] (2,2)--(3,1);
\draw[-,red] (0,1)--(1,0);
\draw[-,red] (1,0)--(2,0);
\draw[-,red] (2,0)--(3,1);
\filldraw [black]  (1,0)    circle (2pt)
[black]  (2,0)    circle (2pt)
[black]  (2,2)    circle (2pt)
[black]  (0,1)    circle (2pt)
[black]  (1,2)    circle (2pt)
[black]  (3,1)    circle (2pt);
\draw (1.5,-0.4) node {Elementary IBB};
\draw (3.5,2.4) node {};
\end{tikzpicture}
  \begin{tikzpicture}[scale=1]
\draw[-,green] (1,0)..controls (0.8,0.5)..(1,1);
\draw[-,red] (1,0)..controls (1.2,0.5)..(1,1);
\draw[-,green] (0,0.5)--(1,1);
\draw[-,red] (1,1)--(2,0.5);
\draw[-,red] (0,0.5)--(1,0);
\draw[-,green] (1,0)--(2,0.5);
\filldraw [black]  (1,0)    circle (2pt)
[black]  (1,1)    circle (2pt)
[black]  (0,0.5)    circle (2pt)
[black]  (2,0.5)    circle (2pt);
\draw (1,-0.4) node {IBB with embedded};
\draw (1,-1) node {elementary IBB};
\draw (3,0.4) node {};
\end{tikzpicture}
  \begin{tikzpicture}[scale=1]
\draw[-,red] (1,1)--(2,2);
\draw[-,red] (2,2)--(3,2);
\draw[-,green] (3,0)..controls (2.8,0.5)..(3,1);
\draw[-,green] (3,1)..controls (2.8,1.5)..(3,2);
\draw[-,red] (3,0)..controls (3.2,0.5)..(3,1);
\draw[-,red] (3,1)..controls (3.2,1.5)..(3,2);
\draw[-,green] (4,0)..controls (3.7,1)..(4,2);
\draw[-,red] (4,0)..controls (4.3,1)..(4,2);
\draw[-,green] (5,0)..controls (4.8,0.5)..(5,1);
\draw[-,green] (5,1)..controls (4.8,1.5)..(5,2);
\draw[-,red] (5,0)..controls (5.2,0.5)..(5,1);
\draw[-,red] (5,1)..controls (5.2,1.5)..(5,2);
\draw[-,green] (5,2)--(6,2);
\draw[-,green] (6,2)--(7,1);
\draw[-,green] (3,2)--(4,2);
\draw[-,red] (4,2)--(5,2);
\draw[-,red] (3,0)--(4,0);
\draw[-,green] (4,0)--(5,0);
\draw[-,green] (1,1)--(3,0);
\draw[-,red] (5,0)--(7,1);
\filldraw [black] (1,1)    circle (2pt)
[black]  (2,2)    circle (2pt)
[black]  (3,0)    circle (2pt)
[black]  (3,1)    circle (2pt)
[black]  (3,2)    circle (2pt)
[black]  (4,0)    circle (2pt)
[black]  (4,2)    circle (2pt)
[black]  (5,0)    circle (2pt)
[black]  (5,1)    circle (2pt)
[black]  (5,2)    circle (2pt)
[black]  (6,2)    circle (2pt)
[black]  (7,1)    circle (2pt);
\draw (4,-0.4) node {IBB with 5 embedded IBB's};
\end{tikzpicture}

\noindent Note that every IBB is a swap unit. In fact, write $P=P_1+P_2+P_3$ and $Q=Q_1+Q_2+Q_3$,
where $P_2$ and $Q_2$ are the subpaths spanning the building block. Swapping the colors generates two longest paths
$$
\widetilde P= P_1+Q_2+P_3\quad \text{and}\quad \widetilde Q=Q_1+P_2+Q_3,
$$
such that the union is still $BT(P,Q)$ and so we obtain a new representation $BT(\widetilde P, \widetilde Q)$.
Hence, if $B$ is an IBB, the subpath of $P$ and the
  subpath of $Q$ have the same length, which we call the
  length of the block $L(B)$.

\begin{remark}
  An IBB with endpoints $a$ and $b$, can be defined independently of $BT(P,Q)$ as the representation of a bitraceable graph with fixed
  endpoints $a$ and $b$, where such a graph is generated by two paths from $a$ to $b$ that have the same length, and such that there is no longer path
  from $a$ to $b$.

  If the intersection of two IBB's is one common endpoint, then the concatenation of the IBB's is their union, and similarly we can
  concatenate three or more IBB's.
\end{remark}

\begin{definition} \label{def EBB BB}
  An \textbf{extremal building unit} of $BT(P,Q)$ is a subgraph which is the union of two extremal subpaths $\widetilde P$ of $P$ and
  $\widetilde Q$ of $Q$, such that $\widetilde P$ and $\widetilde Q$ have exactly one endpoint in common,
  and this common endpoint is neither an endpoint of $P$ nor of $Q$.
  We also require that $\widetilde P$ and $\widetilde Q$ don't intersect other subpaths of $P$ and $Q$.

  An \textbf{extremal building block (EBB)} of $BT(P,Q)$, is a minimal extremal building unit. This means that it is an extremal
  building unit, which is not the concatenation of one or more
  IBB's with another extremal building unit.

  An \textbf{elementary EBB} is an EBB, which is the union of two extremal completed
   components that share one endpoint.

  A \textbf{building block (BB)} is an IBB, an EBB or all of $BT(P,Q)$, if $BT(P,Q)$ is not the concatenation of IBB's and/or EBB's,
  and $P$ and $Q$ have no common endpoints.
\end{definition}

\noindent \begin{tikzpicture}[scale=0.8]
\draw[-,green] (2,2)..controls (3,2.5)..(4,2);
\draw[-,red] (0,3)..controls (1.4,2)..(2,2);
\draw[-,red] (2,2)..controls (3,1.5)..(4,2);
\draw[-,green] (2,4)..controls (1.4,4)..(0,3);
\draw[-,red] (2,4)..controls (3,3.5)..(4,4);
\draw[-,green] (5,2.3)..controls (4.5,2)..(4,2);
\draw[-,green] (2,4)..controls (3,4.5)..(4,4);
\draw[-,red] (2,4)..controls (1.5,4.6)..(1,4.4);
\draw[-,red] (4,2)--(4,4);
\draw[-,green] (2,2)..controls (2,2.6) and (4,3.4)..(4,4);
\filldraw [black]  (2,2)    circle (2pt)
[black]  (4,2)    circle (2pt)
[black]  (2,4)    circle (2pt)
[black]  (0,3)    circle (2pt)
[black]  (4,4)    circle (2pt)
[black]  (5,2.3)    circle (2pt)
[black]  (1,4.4)    circle (2pt);
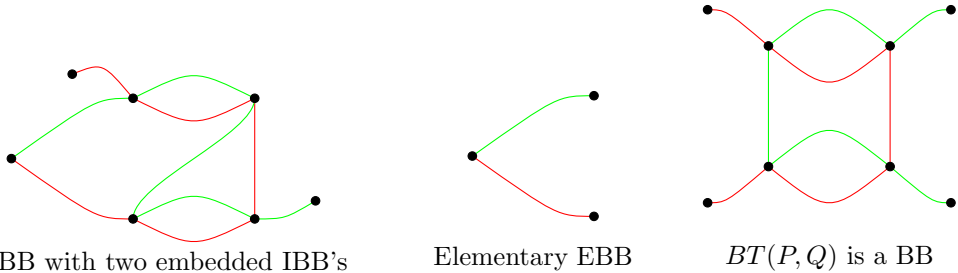
\draw(2.5,1.3)node{EBB with two embedded IBB's};
\draw(6.5,1.3)node{};
\end{tikzpicture}
\begin{tikzpicture}[scale=0.8]
\draw[-,red] (0,3)..controls (1.4,2)..(2,2);
\draw[-,green] (2,4)..controls (1.4,4)..(0,3);
\filldraw [black]  (2,2)    circle (2pt)
[black]  (2,4)    circle (2pt)
[black]  (0,3)    circle (2pt);
\draw(1,1.3)node{Elementary EBB};
\draw(3.5,2)node{};
\end{tikzpicture}
\begin{tikzpicture}[scale=0.8]
\draw[-,green] (2,2)..controls (3,2.8)..(4,2);
\draw[-,red] (1,1.4)..controls (1.3,1.4)..(2,2);
\draw[-,red] (2,2)..controls (3,1.2)..(4,2);
\draw[-,green] (4,2)..controls (4.7,1.4)..(5,1.4);
\draw[-,red] (2,4)..controls (3,3.2)..(4,4);
\draw[-,red] (1,4.6)..controls (1.3,4.6)..(2,4);
\draw[-,green] (2,4)..controls (3,4.8)..(4,4);
\draw[-,green] (4,4)..controls (4.7,4.6)..(5,4.6);
\draw[-,red] (4,2)--(4,4);
\draw[-,green] (2,2)--(2,4);
\filldraw [black]  (2,2)    circle (2pt)
[black]  (4,2)    circle (2pt)
[black]  (2,4)    circle (2pt)
[black]  (5,4.6)    circle (2pt)
[black]  (5,1.4)    circle (2pt)
[black]  (1,4.6)    circle (2pt)
[black]  (1,1.4)    circle (2pt)
[black]  (4,4)    circle (2pt);
\draw(3,0.5)node{$BT(P,Q)$ is a BB};
\end{tikzpicture}

Note that an IBB can be contained in another BB, and that BB's are also swap units.
Note also that in the case $\sigma=Id$, for each $0<i<\ell$ the union of $P_i$ and $Q_i$ is an IBB, and
 the union of $P_0$ and $Q_0$ and the union of $P_\ell$ and $Q_\ell$ are the EBB's.

\begin{definition} \label{def ESU}
  Given a building block, the \textbf{exterior swap unit (ESU)} associated with the building block, is the union of all
  completed components in the building block, that are not contained in any embedded IBBs.
\end{definition}

Note that the ESU of an elementary IBB is the whole IBB and the same holds for an elementary EBB.

The following representation of a BT-graph is the concatenation of three IBB's, two of them are elementary IBB's and the middle one
contains five elementary IBBs. Hence $BT(P,Q)$ has
seven ESU's of two completed components each and one ESU consisting of 8 completed components (it has ten edges).

\noindent  
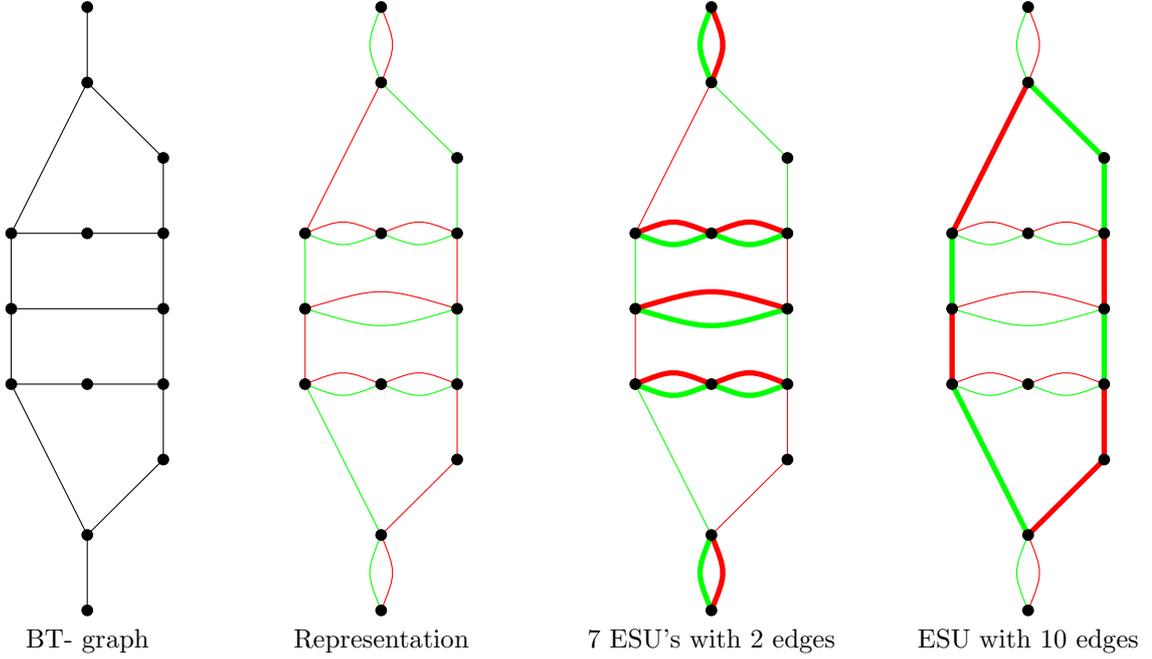
\begin{tikzpicture}[scale=1]
\draw[-] (1,0)--(1,1);
\draw[-] (1,1)--(2,2);
\draw[-] (2,2)--(2,3);
\draw[-] (0,3)--(2,3);
\draw[-] (0,4)--(2,4);
\draw[-] (0,5)--(2,5);
\draw[-] (2,5)--(2,6);
\draw[-] (2,6)--(1,7);
\draw[-] (1,7)--(1,8);
\draw[-] (2,3)--(2,5);
\draw[-] (0,3)--(0,5);
\draw[-] (1,1)--(0,3);
\draw[-] (0,5)--(1,7);
\filldraw [black]  (1,0)    circle (2pt)
[black]  (1,1)    circle (2pt)
[black]  (2,2)    circle (2pt)
[black]  (0,3)    circle (2pt)
[black]  (1,3)    circle (2pt)
[black]  (2,3)    circle (2pt)
[black]  (0,4)    circle (2pt)
[black]  (2,4)    circle (2pt)
[black]  (0,5)    circle (2pt)
[black]  (1,5)    circle (2pt)
[black]  (2,5)    circle (2pt)
[black]  (2,6)    circle (2pt)
[black]  (1,7)    circle (2pt)
[black]  (1,8)    circle (2pt);
\draw (1,-0.4) node {BT- graph};
\draw (3.4,0.4) node {};
\end{tikzpicture}
  \begin{tikzpicture}[scale=1]
\draw[-,green] (1,0)..controls (0.8,0.5)..(1,1);
\draw[-,red] (1,0)..controls (1.2,0.5)..(1,1);
\draw[-,red] (1,1)--(2,2);
\draw[-,red] (2,2)--(2,3);
\draw[-,green] (0,3)..controls (0.5,2.8)..(1,3);
\draw[-,green] (1,3)..controls (1.5,2.8)..(2,3);
\draw[-,red] (0,3)..controls (0.5,3.2)..(1,3);
\draw[-,red] (1,3)..controls (1.5,3.2)..(2,3);
\draw[-,green] (0,4)..controls (1,3.7)..(2,4);
\draw[-,red] (0,4)..controls (1,4.3)..(2,4);
\draw[-,green] (0,5)..controls (0.5,4.8)..(1,5);
\draw[-,green] (1,5)..controls (1.5,4.8)..(2,5);
\draw[-,red] (0,5)..controls (0.5,5.2)..(1,5);
\draw[-,red] (1,5)..controls (1.5,5.2)..(2,5);
\draw[-,green] (2,5)--(2,6);
\draw[-,green] (2,6)--(1,7);
\draw[-,green] (1,7)..controls (0.8,7.5)..(1,8);
\draw[-,red] (1,7)..controls (1.2,7.5)..(1,8);
\draw[-,green] (2,3)--(2,4);
\draw[-,red] (2,4)--(2,5);
\draw[-,red] (0,3)--(0,4);
\draw[-,green] (0,4)--(0,5);
\draw[-,green] (1,1)--(0,3);
\draw[-,red] (0,5)--(1,7);
\filldraw [black]  (1,0)    circle (2pt)
[black]  (1,1)    circle (2pt)
[black]  (2,2)    circle (2pt)
[black]  (0,3)    circle (2pt)
[black]  (1,3)    circle (2pt)
[black]  (2,3)    circle (2pt)
[black]  (0,4)    circle (2pt)
[black]  (2,4)    circle (2pt)
[black]  (0,5)    circle (2pt)
[black]  (1,5)    circle (2pt)
[black]  (2,5)    circle (2pt)
[black]  (2,6)    circle (2pt)
[black]  (1,7)    circle (2pt)
[black]  (1,8)    circle (2pt);
\draw (1,-0.4) node {Representation};
\draw (3.4,0.4) node {};
\end{tikzpicture}
\begin{tikzpicture}[scale=1]
\draw[-,line width=2pt,green] (1,0)..controls (0.8,0.5)..(1,1);
\draw[-,line width=2pt,red] (1,0)..controls (1.2,0.5)..(1,1);
\draw[-,red] (1,1)--(2,2);
\draw[-,red] (2,2)--(2,3);
\draw[-,line width=2pt,green] (0,3)..controls (0.5,2.8)..(1,3);
\draw[-,line width=2pt,green] (1,3)..controls (1.5,2.8)..(2,3);
\draw[-,line width=2pt,red] (0,3)..controls (0.5,3.2)..(1,3);
\draw[-,line width=2pt,red] (1,3)..controls (1.5,3.2)..(2,3);
\draw[-,line width=2pt,green] (0,4)..controls (1,3.7)..(2,4);
\draw[-,line width=2pt,red] (0,4)..controls (1,4.3)..(2,4);
\draw[-,line width=2pt,green] (0,5)..controls (0.5,4.8)..(1,5);
\draw[-,line width=2pt,green] (1,5)..controls (1.5,4.8)..(2,5);
\draw[-,line width=2pt,red] (0,5)..controls (0.5,5.2)..(1,5);
\draw[-,line width=2pt,red] (1,5)..controls (1.5,5.2)..(2,5);
\draw[-,green] (2,5)--(2,6);
\draw[-,green] (2,6)--(1,7);
\draw[-,line width=2pt,green] (1,7)..controls (0.8,7.5)..(1,8);
\draw[-,line width=2pt,red] (1,7)..controls (1.2,7.5)..(1,8);
\draw[-,green] (2,3)--(2,4);
\draw[-,red] (2,4)--(2,5);
\draw[-,red] (0,3)--(0,4);
\draw[-,green] (0,4)--(0,5);
\draw[-,green] (1,1)--(0,3);
\draw[-,red] (0,5)--(1,7);
\filldraw [black]  (1,0)    circle (2pt)
[black]  (1,1)    circle (2pt)
[black]  (2,2)    circle (2pt)
[black]  (0,3)    circle (2pt)
[black]  (1,3)    circle (2pt)
[black]  (2,3)    circle (2pt)
[black]  (0,4)    circle (2pt)
[black]  (2,4)    circle (2pt)
[black]  (0,5)    circle (2pt)
[black]  (1,5)    circle (2pt)
[black]  (2,5)    circle (2pt)
[black]  (2,6)    circle (2pt)
[black]  (1,7)    circle (2pt)
[black]  (1,8)    circle (2pt);
\draw (1,-0.4) node {7 ESU's with 2 edges};
\draw (3.4,8.2) node {};
\end{tikzpicture}
\begin{tikzpicture}[scale=1]
\draw[-,green] (1,0)..controls (0.8,0.5)..(1,1);
\draw[-,red] (1,0)..controls (1.2,0.5)..(1,1);
\draw[-,line width=2pt,red] (1,1)--(2,2);
\draw[-,line width=2pt,red] (2,2)--(2,3);
\draw[-,green] (0,3)..controls (0.5,2.8)..(1,3);
\draw[-,green] (1,3)..controls (1.5,2.8)..(2,3);
\draw[-,red] (0,3)..controls (0.5,3.2)..(1,3);
\draw[-,red] (1,3)..controls (1.5,3.2)..(2,3);
\draw[-,green] (0,4)..controls (1,3.7)..(2,4);
\draw[-,red] (0,4)..controls (1,4.3)..(2,4);
\draw[-,green] (0,5)..controls (0.5,4.8)..(1,5);
\draw[-,green] (1,5)..controls (1.5,4.8)..(2,5);
\draw[-,red] (0,5)..controls (0.5,5.2)..(1,5);
\draw[-,red] (1,5)..controls (1.5,5.2)..(2,5);
\draw[-,line width=2pt,green] (2,5)--(2,6);
\draw[-,line width=2pt,green] (2,6)--(1,7);
\draw[-,green] (1,7)..controls (0.8,7.5)..(1,8);
\draw[-,red] (1,7)..controls (1.2,7.5)..(1,8);
\draw[-,line width=2pt,green] (2,3)--(2,4);
\draw[-,line width=2pt,red] (2,4)--(2,5);
\draw[-,line width=2pt,red] (0,3)--(0,4);
\draw[-,line width=2pt,green] (0,4)--(0,5);
\draw[-,line width=2pt,green] (1,1)--(0,3);
\draw[-,line width=2pt,red] (0,5)--(1,7);
\filldraw [black]  (1,0)    circle (2pt)
[black]  (1,1)    circle (2pt)
[black]  (2,2)    circle (2pt)
[black]  (0,3)    circle (2pt)
[black]  (1,3)    circle (2pt)
[black]  (2,3)    circle (2pt)
[black]  (0,4)    circle (2pt)
[black]  (2,4)    circle (2pt)
[black]  (0,5)    circle (2pt)
[black]  (1,5)    circle (2pt)
[black]  (2,5)    circle (2pt)
[black]  (2,6)    circle (2pt)
[black]  (1,7)    circle (2pt)
[black]  (1,8)    circle (2pt);
\draw (1,-0.4) node {ESU with 10 edges};
\end{tikzpicture}

Note that
swapping the color of the edges in a swap unit
doesn't change the intersection vertices, since each intersection vertex has two incident edges of each color.
After the swapping it must still have two incident edges of each color, since otherwise one of the paths would visit this vertex twice.

\begin{lemma} \label{coincidencia de longitud y numero de componentes}
  In a swap unit, the sum of the lengths of the completed components of one color equals the sum of the lengths of
  the completed components of the other color. Moreover, the number of components of each color coincide.
\end{lemma}

\begin{proof}
  The sum of the lengths coincide, since otherwise one of the new paths resulting from the swap would be longer. The same holds for the number of
  components, since swapping the color of the edges doesn't change the intersection vertices, and so the number of components of each longest
   path remains constant.
\end{proof}

\begin{remark}
  Note that  the union of the ESU's in $BT(P,Q)$ is $BT(P,Q)$ and that
no completed component is in two ESU's at the same time.
\end{remark}

\begin{definition}\label{def NC}
  Let $X$ and $Y$ be components of different colors in an ESU. The pair $X,Y$
  is \textbf{non connectable (NC)}, if a path
  $R$ that connects $X$ with $Y$, satisfying $R\cap V(P)\cap (Q)=\emptyset$ and such that
   $R$ is internally disjoint from the given ESU and in each of the other ESU's
   touches at most one color,
  allows to construct two paths $\widehat P$ and $\widehat Q$ in $BT(P,Q)\cup R$ such that
  $$
  L(\widehat P)+L(\widehat Q)= 2L(R)+ L(P)+L(Q)=2 L(P)+2 L(R).
  $$
  The ESU is called NC, if all the pairs of component of the ESU of different colors are NC.
\end{definition}

\begin{proposition} \label{elementary block is NC}
  The components of an elementary IBB are NC.
\end{proposition}

\begin{proof}
  Let $R$ be a path that connects one component of one color in the
  ESU with one component of another color in the ESU, such that $R$ is internally disjoint from the given ESU and in each of the other ESU's touches
  at most one color.

        
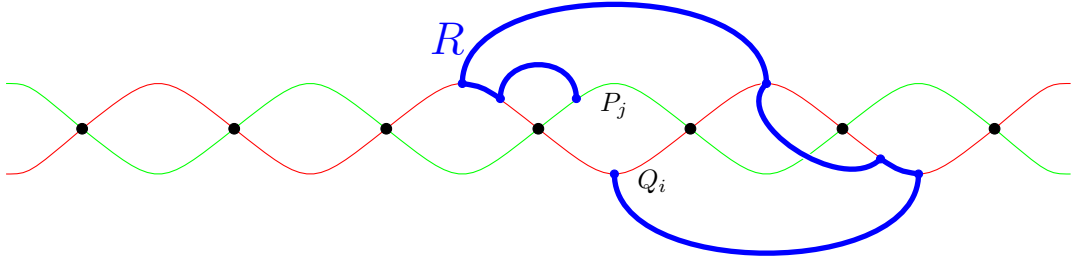
\begin{tikzpicture}[scale=1]
    \draw[-,green] (1,2.6)..controls (1.3,2.6)..(2,2);
    \draw[-,red] (2,2)..controls (3,2.8)..(4,2);
    \draw[-,green] (4,2)..controls (5,2.8)..(6,2);
    \draw[-,red] (6,2)..controls (7,2.8)..(8,2);
    \draw[-,green] (8,2)..controls (9,2.8)..(10,2);
    \draw[-,red] (10,2)..controls (11,2.8)..(12,2);
    \draw[-,green] (12,2)..controls (13,2.8)..(14,2);
    \draw[-,red] (14,2)..controls (14.7,2.6)..(15,2.6);
    \draw[-,red] (1,1.4)..controls (1.3,1.4)..(2,2);
    \draw[-,green] (2,2)..controls (3,1.2)..(4,2);
    \draw[-,red] (4,2)..controls (5,1.2)..(6,2);
    \draw[-,green] (6,2)..controls (7,1.2)..(8,2);
    \draw[-,red] (8,2)..controls (9,1.2)..(10,2);
    \draw[-,green] (10,2)..controls (11,1.2)..(12,2);
    \draw[-,red] (12,2)..controls (13,1.2)..(14,2);
    \draw[-,green] (14,2)..controls (14.7,1.4)..(15,1.4);
    \draw[-,white, line width=3pt] (7,2.6)..controls (7,4)and(11,4)..(11,2.7);
    \draw[-,blue,line width=2pt] (7,2.6)..controls (7,4)and(11,4)..(11,2.6);
    \draw[-,blue,line width=2pt] (7.5,2.4)..controls (7.5,3)and(8.5,3)..(8.5,2.4);
    \draw[-,blue,line width=2pt] (9,1.4)..controls (9,0)and(13,0)..(13,1.4);
    \draw[-,white,line width=3pt] (11,2.6)..controls (10.5,2.1)and(12,1.1)..(12.5,1.6);
    \draw[-,blue,line width=2pt] (11,2.6)..controls (10.5,2.1)and(12,1.1)..(12.5,1.6);
    \draw[-,blue,line width=2pt] (7,2.6)..controls (7.25,2.56)..(7.5,2.4);
    \draw[-,blue,line width=2pt] (12.5,1.6)..controls (12.75,1.44)..(13,1.4);
    \filldraw     [blue]  (7,2.6)    circle (1.5pt)
    [blue]  (7.5,2.4)    circle (1.5pt)
    [blue]  (8.5,2.4)    circle (1.5pt)
    [blue]  (9,1.4)    circle (1.5pt)
    [blue]  (11,2.6)    circle (1.5pt)
    [blue]  (12.5,1.6)    circle (1.5pt)
    [blue]  (13,1.4)    circle (1.5pt);
    \filldraw [black]  (2,2)    circle (2pt)
    [black]  (4,2)    circle (2pt)
    [black]  (6,2)    circle (2pt)
    [black]  (8,2)    circle (2pt)
    [black]  (10,2)    circle (2pt)
    [black]  (12,2)    circle (2pt)
    [black]  (14,2)    circle (2pt);
    \draw(9,2.3)node{$P_j$};
    \draw(9.5,1.3)node{$Q_i$};
    \draw[blue](6.8,3.2)node{{\huge\bf{$R$}}};
    \draw(8,0)node{$P$ in green, $Q$ in red, $R$ in blue};
  \end{tikzpicture}

  Then $R$ is internally disjoint from $\widetilde Q$, which is the
  path obtained from $Q$ by swapping the colors in the ESU's where $R$ touches only the color of $Q$.

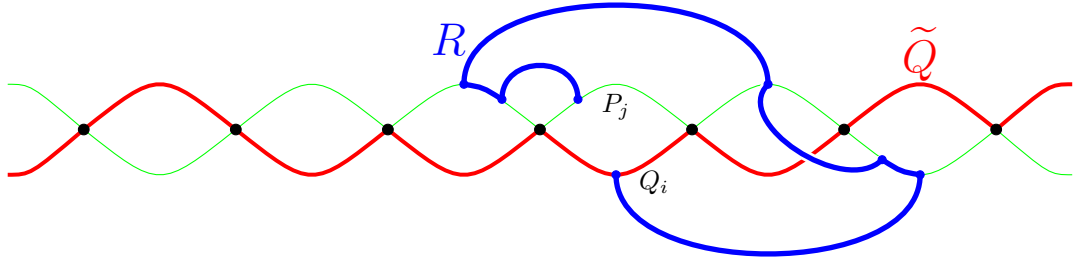
        \begin{tikzpicture}[scale=1]
    \draw[-,green] (1,2.6)..controls (1.3,2.6)..(2,2);
    \draw[-,red,line width=1.5pt] (2,2)..controls (3,2.8)..(4,2);
    \draw[-,green] (4,2)..controls (5,2.8)..(6,2);
    \draw[-,green] (6,2)..controls (7,2.8)..(8,2);
    \draw[-,green] (8,2)..controls (9,2.8)..(10,2);
    \draw[-,green] (10,2)..controls (11,2.8)..(12,2);
    \draw[-,red,line width=1.5pt] (12,2)..controls (13,2.8)..(14,2);
    \draw[-,red,line width=1.5pt] (14,2)..controls (14.7,2.6)..(15,2.6);
    \draw[-,red,line width=1.5pt] (1,1.4)..controls (1.3,1.4)..(2,2);
    \draw[-,green] (2,2)..controls (3,1.2)..(4,2);
    \draw[-,red,line width=1.5pt] (4,2)..controls (5,1.2)..(6,2);
    \draw[-,red,line width=1.5pt] (6,2)..controls (7,1.2)..(8,2);
    \draw[-,red,line width=1.5pt] (8,2)..controls (9,1.2)..(10,2);
    \draw[-,red,line width=1.5pt] (10,2)..controls (11,1.2)..(12,2);
    \draw[-,green] (12,2)..controls (13,1.2)..(14,2);
    \draw[-,green] (14,2)..controls (14.7,1.4)..(15,1.4);
    \draw[-,white, line width=3pt] (7,2.6)..controls (7,4)and(11,4)..(11,2.7);
    \draw[-,blue,line width=2pt] (7,2.6)..controls (7,4)and(11,4)..(11,2.6);
    \draw[-,blue,line width=2pt] (7.5,2.4)..controls (7.5,3)and(8.5,3)..(8.5,2.4);
    \draw[-,blue,line width=2pt] (9,1.4)..controls (9,0)and(13,0)..(13,1.4);
    \draw[-,white,line width=4pt] (11,2.6)..controls (10.5,2.1)and(12,1.1)..(12.5,1.6);
    \draw[-,blue,line width=2pt] (11,2.6)..controls (10.5,2.1)and(12,1.1)..(12.5,1.6);
    \draw[-,blue,line width=2pt] (7,2.6)..controls (7.25,2.56)..(7.5,2.4);
    \draw[-,blue,line width=2pt] (12.5,1.6)..controls (12.75,1.44)..(13,1.4);
    \filldraw     [blue]  (7,2.6)    circle (1.5pt)
    [blue]  (7.5,2.4)    circle (1.5pt)
    [blue]  (8.5,2.4)    circle (1.5pt)
    [blue]  (9,1.4)    circle (1.5pt)
    [blue]  (11,2.6)    circle (1.5pt)
    [blue]  (12.5,1.6)    circle (1.5pt)
    [blue]  (13,1.4)    circle (1.5pt);
    \filldraw [black]  (2,2)    circle (2pt)
    [black]  (4,2)    circle (2pt)
    [black]  (6,2)    circle (2pt)
    [black]  (8,2)    circle (2pt)
    [black]  (10,2)    circle (2pt)
    [black]  (12,2)    circle (2pt)
    [black]  (14,2)    circle (2pt);
    \draw(9,2.3)node{$P_j$};
    \draw(9.5,1.3)node{$Q_i$};
    \draw[blue](6.8,3.2)node{{\huge\bf{$R$}}};
    \draw[red](13,3)node{{\huge\bf{$\widetilde Q$}}};
    \draw(8,0)node{$R$ is internally disjoint from the new path $\widetilde Q$ in red};
  \end{tikzpicture}

   Write $\widetilde Q$ as $\widetilde Q_1+ Q_i+\widetilde Q_2$, where $Q_i$ is the
  subpath of $Q$ in the given ESU, and assume that $P_j$ is the subpath of $P$ in the given ESU. The endpoints of $R$ split the subpaths
  $Q_i$ and $P_j$ into two subpaths each, which we name  $P_{j,1}$, $P_{j,2}$,$Q_{i,1}$,$Q_{i,2}$.
  The two paths
  $$
  \widehat Q=\widetilde Q_{1}+Q_{i,1}+R+P_{j,2}+\widetilde Q_{2}\quad\text{and}\quad \widehat P=\widetilde Q_{1}+P_{j,1}+R+Q_{i,2}+\widetilde Q_{2}
  $$

        
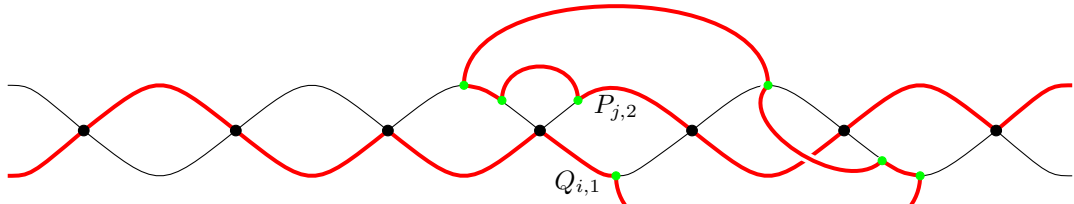
\begin{tikzpicture}[scale=1]
    \draw[-] (1,2.6)..controls (1.3,2.6)..(2,2);
    \draw[-,red,line width=1.5pt] (2,2)..controls (3,2.8)..(4,2);
    \draw[-] (4,2)..controls (5,2.8)..(6,2);
    \draw[-] (6,2)..controls (7,2.8)..(8,2);
    \draw[-] (8,2)--(8.5,2.4);
    \draw[-,red,line width=1.5pt] (8.5,2.4)..controls (9,2.8)and(9.5,2.4)..(10,2);
    \draw[-] (10,2)..controls (11,2.8)..(12,2);
    \draw[-,red,line width=1.5pt] (12,2)..controls (13,2.8)..(14,2);
    \draw[-,red,line width=1.5pt] (14,2)..controls (14.7,2.6)..(15,2.6);
    \draw[-,red,line width=1.5pt] (1,1.4)..controls (1.3,1.4)..(2,2);
    \draw[-] (2,2)..controls (3,1.2)..(4,2);
    \draw[-,red,line width=1.5pt] (4,2)..controls (5,1.2)..(6,2);
    \draw[-,red,line width=1.5pt] (6,2)..controls (7,1.2)..(8,2);
    \draw[-,red,line width=1.5pt] (8,2)..controls (8.8,1.4)..(9,1.4);
    \draw[-] (9,1.4)..controls (9.2,1.4)..(10,2);
    \draw[-,red,line width=1.5pt] (10,2)..controls (11,1.2)..(12,2);
    \draw[-] (12,2)..controls (13,1.2)..(14,2);
    \draw[-] (14,2)..controls (14.7,1.4)..(15,1.4);
    \draw[-,white, line width=3pt] (7,2.6)..controls (7,4)and(11,4)..(11,2.7);
    \draw[-,red,line width=1.5pt] (7,2.6)..controls (7,4)and(11,4)..(11,2.6);
    \draw[-,red,line width=1.5pt] (7.5,2.4)..controls (7.5,3)and(8.5,3)..(8.5,2.4);
    \draw[-,red,line width=1.5pt] (9,1.4)..controls (9,0)and(13,0)..(13,1.4);
    \draw[-,white,line width=4pt] (11,2.6)..controls (10.5,2.1)and(12,1.1)..(12.5,1.6);
    \draw[-,red,line width=1.5pt] (11,2.6)..controls (10.5,2.1)and(12,1.1)..(12.5,1.6);
    \draw[-,red,line width=1.5pt] (7,2.6)..controls (7.25,2.56)..(7.5,2.4);
    \draw[-,red,line width=1.5pt] (12.5,1.6)..controls (12.75,1.44)..(13,1.4);
    \filldraw     [green]  (7,2.6)    circle (1.5pt)
    [green]  (7.5,2.4)    circle (1.5pt)
    [green]  (8.5,2.4)    circle (1.5pt)
    [green]  (9,1.4)    circle (1.5pt)
    [green]  (11,2.6)    circle (1.5pt)
    [green]  (12.5,1.6)    circle (1.5pt)
    [green]  (13,1.4)    circle (1.5pt);
    \filldraw [black]  (2,2)    circle (2pt)
    [black]  (4,2)    circle (2pt)
    [black]  (6,2)    circle (2pt)
    [black]  (8,2)    circle (2pt)
    [black]  (10,2)    circle (2pt)
    [black]  (12,2)    circle (2pt)
    [black]  (14,2)    circle (2pt);
    \draw(9,2.3)node{$P_{j,2}$};
    \draw(8.5,1.3)node{$Q_{i,1}$};
    \draw(8,0)node{$\widehat Q=\widetilde Q_{1}+Q_{i,1}+R+P_{j,2}+\widetilde Q_{2}$ in red};
  \end{tikzpicture}

        
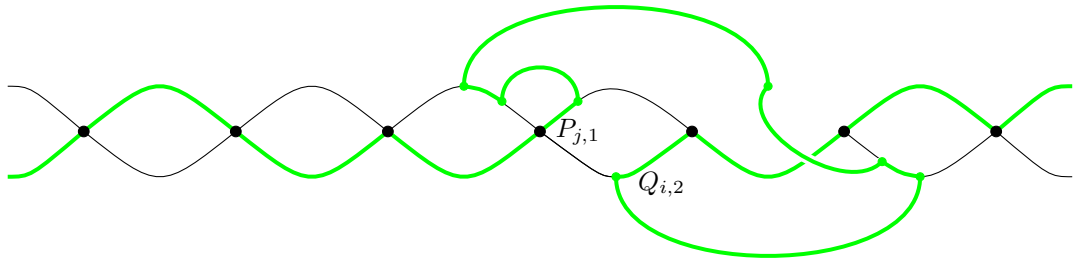
\begin{tikzpicture}[scale=1]
    \draw[-] (1,2.6)..controls (1.3,2.6)..(2,2);
    \draw[-,green,line width=1.5pt] (2,2)..controls (3,2.8)..(4,2);
    \draw[-] (4,2)..controls (5,2.8)..(6,2);
    \draw[-] (6,2)..controls (7,2.8)..(8,2);
    \draw[-] (8,2)..controls (8.8,1.4)..(9,1.4);
    \draw[-,green,line width=1.5pt] (8,2)--(8.5,2.4);
    \draw[-] (8.5,2.4)..controls (9,2.8)and(9.5,2.4)..(10,2);
    \draw[-,green,line width=1.5pt] (12,2)..controls (13,2.8)..(14,2);
    \draw[-,green,line width=1.5pt] (14,2)..controls (14.7,2.6)..(15,2.6);
    \draw[-,green,line width=1.5pt] (1,1.4)..controls (1.3,1.4)..(2,2);
    \draw[-] (2,2)..controls (3,1.2)..(4,2);
    \draw[-,green,line width=1.5pt] (4,2)..controls (5,1.2)..(6,2);
    \draw[-,green,line width=1.5pt] (6,2)..controls (7,1.2)..(8,2);
    \draw[-] (8,2)..controls (8.8,1.4)..(9,1.4);
    \draw[-,green,line width=1.5pt] (9,1.4)..controls (9.2,1.4)..(10,2);
    \draw[-,green,line width=1.5pt] (10,2)..controls (11,1.2)..(12,2);
    \draw[-] (12,2)..controls (13,1.2)..(14,2);
    \draw[-] (14,2)..controls (14.7,1.4)..(15,1.4);
    \draw[-,white, line width=3pt] (7,2.6)..controls (7,4)and(11,4)..(11,2.7);
    \draw[-,green,line width=1.5pt] (7,2.6)..controls (7,4)and(11,4)..(11,2.6);
    \draw[-,green,line width=1.5pt] (7.5,2.4)..controls (7.5,3)and(8.5,3)..(8.5,2.4);
    \draw[-,green,line width=1.5pt] (9,1.4)..controls (9,0)and(13,0)..(13,1.4);
    \draw[-,white,line width=4pt] (11,2.6)..controls (10.5,2.1)and(12,1.1)..(12.5,1.6);
    \draw[-,green,line width=1.5pt] (11,2.6)..controls (10.5,2.1)and(12,1.1)..(12.5,1.6);
    \draw[-,green,line width=1.5pt] (7,2.6)..controls (7.25,2.56)..(7.5,2.4);
    \draw[-,green,line width=1.5pt] (12.5,1.6)..controls (12.75,1.44)..(13,1.4);
    \filldraw     [green]  (7,2.6)    circle (1.5pt)
    [green]  (7.5,2.4)    circle (1.5pt)
    [green]  (8.5,2.4)    circle (1.5pt)
    [green]  (9,1.4)    circle (1.5pt)
    [green]  (11,2.6)    circle (1.5pt)
    [green]  (12.5,1.6)    circle (1.5pt)
    [green]  (13,1.4)    circle (1.5pt);
    \filldraw [black]  (2,2)    circle (2pt)
    [black]  (4,2)    circle (2pt)
    [black]  (6,2)    circle (2pt)
    [black]  (8,2)    circle (2pt)
    [black]  (10,2)    circle (2pt)
    [black]  (12,2)    circle (2pt)
    [black]  (14,2)    circle (2pt);
    \draw(8.5,2)node{$P_{j,1}$};
    \draw(9.6,1.3)node{$Q_{i,2}$};
    \draw(8,0)node{$\widehat P=\widetilde Q_{1}+P_{j,1}+R+Q_{i,2}+\widetilde Q_{2}$ in green};
  \end{tikzpicture}

  have lengths that sum $2L(\widetilde Q)+2L(R)=2 L(P)+2 L(R)$, as desired.
\end{proof}

\begin{proposition} \label{prop no path exists}
  Let $G$ be a graph and $P$,$Q$ longest paths. If all the ESU's in $BT(P,Q)$ are NC, then there can be no path in
  $G\setminus (V(P)\cap V(Q))$ from a component of one color in one ESU to a component of the other color in the same ESU.
\end{proposition}

\begin{proof}
  Assume by contradiction that such a path from $P_{p}'$ to $Q_{q}'$ exists. Then there exist
  $X_1,\dots,X_r\in \{P_i',Q_i': P_i'\ne \emptyset, Q_i'\ne \emptyset\}$, such that $X_i\sim X_{i+1}$ and $X_1=P_{p}'$ and $X_r=Q_{q}'$.
  Let $i_0<j_0$ be such that $X_{j_0}$ and $X_{i_0}$ are in the same ESU and have different colors, and such that
  $$
  j_0-i_0=\min\{j-i: \ i<j, X_j\text{ and }X_i \text{ are in the same ESU and have different colors}\}.
  $$
  Then the subpath $R'$ of $R$ which connects successively $X_{i_0},X_{i_0+1},\dots, X_{j_0}$, is internally disjoint from the
   ESU of $X_{j_0}$ and $X_{i_0}$
  and in each of the other ESU's touches at most one color.
  Since the ESU of $X_{j_0}$ and $X_{i_0}$ is NC, by Definition~\ref{def NC} there exists a path in $BT(P,Q)\cup R$
  that is longer than $P$. This path determines a path of the same length in $G$, a contradiction that
  concludes the proof.
\end{proof}

\begin{remark} \label{large class}
  Proposition~\ref{prop no path exists} characterizes the large class of graphs $BT(P,Q)$,
  in which all the ESU's are NC. When $\#(V(P)\cap V(Q))$ is small, in most cases $BT(P,Q)$ is in this class, in particular we
  will show that this class includes all $BT(P,Q)$ such that $\# (V(P)\cap V(Q))\le 4$.
\end{remark}

\begin{proposition} \label{isu nc implica articulation set}
  Let $G$ be a graph and $P$,$Q$ longest paths. If $V(P)\ne V(Q)$ and all the ESU's in $BT(P,Q)$ are NC, then
  $V(P)\cap V(Q)$ is a separator of $G$.
\end{proposition}

\begin{proof}
   Since $V(P)\ne V(Q)$, there exists $i_0$ such that $P_{i_0}'\ne \emptyset$. Every $P_i'$ is part of an ESU in
   $BT(P,Q)$, and by
   Lemma~\ref{coincidencia de longitud y numero de componentes} there is a $Q_{j_0}'\ne \emptyset$ in that ESU. By
   Proposition~\ref{prop no path exists} these components cannot be connected in $G\setminus (V(P)\cap V(Q))$, so
   $V(P)\cap V(Q)$ is a separator of $G$, as desired.
\end{proof}

\begin{definition} \label{def LNC}
\begin{itemize}
  \item[a)] Let $B$ be an BB in $BT(P,Q)$ and let $X,Y$ be components of different colors of the corresponding ESU.
  The pair $X,Y$ is called \textbf{locally non connectable (LNC)},
  if for any $x\in X$ and $y\in Y$,
  there exist two pairs of disjoint paths in $B$: one pair of disjoint paths, $X_1$ from $x$ to one endpoint of $B$, and $Y_1$ from $y$ to
  another endpoint of $B$; and another
  pair of disjoint paths, $X_2$ from $x$ to one endpoint of $B$, and $Y_2$ from $y$ to another endpoint of $B$, such that
  $$
  X_1\cup Y_1\cup X_2\cup Y_2=B,
  $$
and such that the intersection of each of $X_1$, $X_2$, $Y_1$, $Y_2$ with an embedded ESU is either empty, or is equal to the
intersection of the embedded ESU with $P$ or with $Q$.

\begin{tikzpicture}[scale=0.5]
\draw[-] (0,2)..controls (4,0)..(5,0);
\draw[-] (0,2)..controls (2,4)..(3,4);
\draw[-,line width=2pt] (10,2)..controls (6,0)..(5,0);
\draw[-] (10,2)..controls (8,4)..(7,4);
\draw[-,line width=2pt] (5,0)..controls (3,3)..(3,4);
\draw[-] (5,0)..controls (7,3)..(7,4);
\draw[-,line width=2pt] (3,4)..controls(5,2.8)..(7,4);
\draw[-] (3,4)..controls(5,5.2)..(7,4);
\draw (5,-1) node {\normalfont $X_1$ and $Y_1$ thickened};
\draw[-,line width=2pt] (0,2)--(1.5,3.5);
\draw (0.4,3.2) node {$X_1$};
\draw (1.4,3.8) node {$x$};
\draw[-,line width=2pt] (7,4)..controls (7.8,4)..(8.5,3.5);
\draw (7.4,1.2) node {$Y_1$};
\draw (8.8,3.7) node {$y$};
\filldraw [black] (3,4) circle (2pt)
[black] (7,4) circle (2pt)
[black] (0,2) circle (2pt)
[black] (10,2) circle (2pt)
[black] (5,0) circle (2pt);
\draw (14,6) node{};
\filldraw [black] (1.5,3.5) circle (2pt)
[black] (8.5,3.5) circle (2pt);
\draw (-0.3,2) node {$a$};
\draw (10.3,2) node {$b$};
\end{tikzpicture}
\begin{tikzpicture}[scale=0.5]
\draw[-,line width=2pt,cyan] (0,2)..controls (4,0)..(5,0);
\draw[-] (0,2)..controls (2,4)..(3,4);
\draw[-] (10,2)..controls (6,0)..(5,0);
\draw[-] (10,2)..controls (8,4)..(7,4);
\draw[-] (5,0)..controls (3,3)..(3,4);
\draw[-,line width=2pt,cyan] (5,0)..controls (7,3)..(7,4);
\draw[-] (3,4)..controls(5,2.8)..(7,4);
\draw[-,line width=2pt,cyan] (3,4)..controls(5,5.2)..(7,4);
\draw (-0.3,2) node {$a$};
\draw (10.3,2) node {$b$};
\draw (5,-1) node {\normalfont $X_2$ and $Y_2$ in blue};
\draw[-,line width=2pt,cyan] (10,2)--(8.5,3.5);
\draw (9.6,3.2) node {$Y_2$};
\draw (1.4,3.8) node {$x$};
\draw[-,line width=2pt,cyan] (3,4)..controls (2.2,4)..(1.5,3.5);
\draw (2.6,1.2) node {$X_2$};
\draw (8.8,3.7) node {$y$};
\filldraw [black] (3,4) circle (2pt)
[black] (7,4) circle (2pt)
[black] (0,2) circle (2pt)
[black] (10,2) circle (2pt)
[black] (5,0) circle (2pt);
\filldraw [black] (1.5,3.5) circle (2pt)
[black] (8.5,3.5) circle (2pt);
\end{tikzpicture}

  \item[b)]  The ESU of a building block $B$ is called LNC, if all the pairs of component of the ESU of different colors are LNC.
  The building block $B$ is called LNC, if its ESU and the ESU's of all embedded IBB's are LNC.
\end{itemize}
 \end{definition}

\begin{remark} \label{path through embedded esu}
Note that if one of the paths of the pair $X_1$, $Y_1$ has an edge contained in one of the embedded IBB's, then it goes
from one endpoint of the IBB to the other, since by definition an IBB can touch the rest of $BT(P,Q)$ only at its endpoints.
Consequently, the other path in the pair cannot touch this embedded IBB. The same holds for the pair $X_2,Y_2$. Since the union is the whole block,
each embedded IBB has to be travelled twice, and so in each of the two pairs one of the paths has to go through the given embedded IBB.
\end{remark}

\begin{proposition}\label{ESU en LNC son NC}
  If a pair of completed components of different colors in an ESU is LNC, then it is NC. Consequently, if an ESU is LNC, then it is NC.
\end{proposition}

\begin{proof}
  Take a pair of components of different colors $X=P_i'$ and $Y=Q_j'$ of a given ESU. Assume that there exists a path $R$ that
  connects $X$ and $Y$, such that $R$ is internally disjoint from the given ESU and in each of the other ESU's
  touches at most one color. Then $R$ is internally disjoint from $\widetilde Q$, which is the
  path obtained from $Q$ by swapping the colors in the ESU's where $R$ touches only the color of $Q$.
  Let $x$, $y$ be the endpoints of $R$ in $P_i$ and $Q_j$ respectively, and let
  $X_1,X_2,Y_1,Y_2$ be as in Definition~\ref{def LNC}(a).

  Let $\widetilde X_1$ be the path obtained from $X_1$ by swapping the ESU's of the embedded IBB's, where $R$ touches $X_1$. This means that
  if $R$ touches $X_1$ in the ESU of an embedded block in a certain color, then, since all the completed components
  of $X_1$ in this ESU have this one color, we can
  replace these completed components of $X_1$ with the completed components of the ESU of the other color,
  which are not touched by $R$, and we have still a path.
  Similarly we define $\widetilde X_2$, $\widetilde Y_1$ and $\widetilde Y_2$. By Remark~\ref{path through embedded esu}, these are paths
  that have the same endpoints and the same length as the original ones, $\widetilde X_1$ and $\widetilde Y_1$ are disjoint, and
  $\widetilde X_2$ and $\widetilde Y_2$ are disjoint. Hence
  $$
  \widehat P=\widetilde X_1+R+\widetilde Y_1\quad\text{and}\quad
  \widehat Q=\widetilde Y_2+R+\widetilde X_2,
  $$
  are two paths such that
  $$
  L(\widehat P)+L(\widehat Q)= 2L(B)+2L(R),
  $$
  and such that the set of the endpoints of the paths coincides with the set of the endpoints of $B$.
  If the building block $B$ is all of $BT(P,Q)$, then this finishes the proof, since then $L(P)=L(B)$. If $B$ is an EBB or an IBB, then
  we extend $\widehat P$ and $\widehat Q$ using $\widetilde Q$. For this we write $\widetilde Q=\widetilde Q_1+\widetilde Q_2+\widetilde Q_3$,
  where $\widetilde Q_2$ is the intersection of $\widetilde Q$ with the given block. Then we set
  $$
  \widehat{\widehat P\,}:= \widetilde Q_1+ \widehat P+\widetilde Q_3\quad \text{and} \quad
    \widehat{\widehat Q\,}:= \widetilde Q_1+ \widehat Q+\widetilde Q_3,
  $$
  in order to obtain two paths
  $\widehat{\widehat P\,}$ and $\widehat{\widehat Q\,}$ such that
  $$
  L(\widehat{\widehat P\,})+L(\widehat{\widehat Q\,})= 2L(P)+2L(R),
  $$
  which concludes the proof.
\end{proof}

\begin{theorem}\label{concatenation of LNC blocks}
  Let $G$ be a graph and let $P$, $Q$ be two longest paths. If $V(P)\ne V(Q)$ and all the ESU's in $BT(P,Q)$ are LNC, then
  $V(P)\cap V(Q)$ is a separator of $G$.
\end{theorem}

\begin{proof}
  By Propositions~\ref{isu nc implica articulation set} and~\ref{ESU en LNC son NC}.
\end{proof}

Our next goal is to construct recursively new LNC building blocks out of some given LNC building blocks. For this we
first generalize Lemmas~\ref{lema 4.1} and~\ref{conexion entre extremos}.

\begin{proposition} \label{prop adyacentes}
  If two completed components $X,Y$ of different colors in an ESU are adjacent, then they are LNC.
\end{proposition}

\begin{proof}
  Let one of the subpaths spanning the block be $\widetilde P_1+X+\widetilde P_2$ and the other $\widetilde Q_1+Y+\widetilde Q_2$, where one
  common vertex is
  $$
  u=X\cap \widetilde P_2=Y \cap \widetilde Q_1.
  $$
  Let $x\in X$ and $y\in Y$ be internal vertices, and write $X=\widetilde X_1+\widetilde X_2$, with $x=\widetilde X_1\cap \widetilde X_2$
  and similarly $Y=\widetilde Y_1+\widetilde Y_2$. Then
  $$
  X_1=\widetilde P_1+\widetilde X_1,\quad X_2=\widetilde Q_1+\widetilde X_2, \quad Y_1=\widetilde P_2+\widetilde Y_1
  \quad\text{and}\quad Y_2=\widetilde Q_2+\widetilde Y_2,
  $$

\noindent
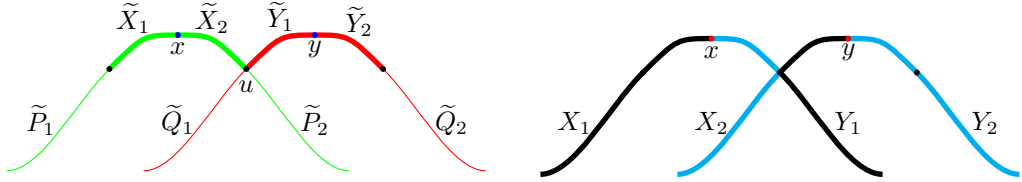
\begin{tikzpicture}[scale=0.45]
\draw[-,green] (0,0)..controls (1,0) and (2,2)..(3,3);
\draw[-,green,line width=2pt] (3,3)..controls (4,4)..(5,4);
\draw[-,green,line width=2pt] (5,4)..controls (6,4)..(7,3);
\draw[-,green] (7,3)..controls (8,2) and (9,0)..(10,0);
\draw[-,red] (4,0)..controls (5,0) and (6,2)..(7,3);
\draw[-,red,line width=2pt] (7,3)..controls (8,4)..(9,4);
\draw[-,red,line width=2pt] (9,4)..controls (10,4)..(11,3);
\draw[-,red] (11,3)..controls (12,2) and (13,0)..(14,0);
\filldraw [black] (3,3) circle (2pt)
[black] (7,3) circle (2pt)
[black] (11,3) circle (2pt);
\filldraw [blue] (5,4) circle (2pt)
[blue] (9,4) circle (2pt);
\draw (14,5) node{};
\draw (3.7,4.5) node {$\widetilde X_1$};
\draw (6,4.5) node {$\widetilde X_2$};
\draw (10.3,4.4) node {$\widetilde Y_2$};
\draw (8,4.55) node {$\widetilde Y_1$};
\draw (1,1.5) node {$\widetilde P_1$};
\draw (13,1.5) node {$\widetilde Q_2$};
\draw (9,1.5) node {$\widetilde P_2$};
\draw (5,1.5) node {$\widetilde Q_1$};
\draw (5,3.6) node {$x$};
\draw (9,3.6) node {$y$};
\draw (7,2.5) node {$u$};
\draw (7,-1) node {\normalfont $X=\widetilde X_1+\widetilde X_2$ and $Y=\widetilde Y_1+\widetilde Y_2$ are adjacent};
\draw(15,2)node{};
\end{tikzpicture}
\begin{tikzpicture}[scale=0.45]
\draw[-,line width=2pt] (0,0)..controls (1,0) and (2,2)..(3,3);
\draw[-,line width=2pt] (3,3)..controls (4,4)..(5,4);
\draw[-,cyan,line width=2pt] (5,4)..controls (6,4)..(7,3);
\draw[-,line width=2pt] (7,3)..controls (8,2) and (9,0)..(10,0);
\draw[-,cyan,line width=2pt] (4,0)..controls (5,0) and (6,2)..(7,3);
\draw[-,line width=2pt] (7,3)..controls (8,4)..(9,4);
\draw[-,cyan,line width=2pt] (9,4)..controls (10,4)..(11,3);
\draw[-,cyan,line width=2pt] (11,3)..controls (12,2) and (13,0)..(14,0);
\filldraw [black] (3,3) circle (2pt)
[black] (7,3) circle (2pt)
[black] (11,3) circle (2pt);
\filldraw [red] (5,4) circle (2pt)
[red] (9,4) circle (2pt);
\draw (14,5) node{};
\draw (1,1.5) node {$X_1$};
\draw (13,1.5) node {$Y_2$};
\draw (9,1.5) node {$Y_1$};
\draw (5,1.5) node {$X_2$};
\draw (5,3.6) node {$x$};
\draw (9,3.6) node {$y$};
\draw (7,-1) node {\normalfont $X_1$ and $Y_1$ in black, $X_2$ and $Y_2$ in blue};
\end{tikzpicture}

\noindent  satisfy the conditions of Definition~\ref{def LNC}. Note that if the common vertex is an end point of the building block,
then $\widetilde P_2$   and $\widetilde Q_1$ have length zero.
\end{proof}

\begin{proposition}\label{PROP conexion entre extremos}
  If two extremal components of $P$ and $Q$ are disjoint and in the same BB, then they are LNC.
\end{proposition}

\begin{proof}
  The proof is similar to the proof of Lemma~\ref{conexion entre extremos}.
  Let
  $$
  \widetilde P+\widetilde X_1+\widetilde X_2\quad\text{and}\quad \widetilde Q+\widetilde Y_2+\widetilde Y_1
  $$

  
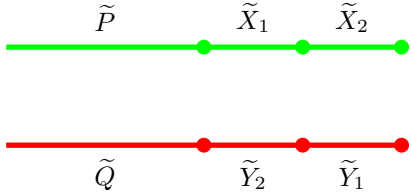
\begin{tikzpicture}[scale=1.3]
\draw[-,line width=2pt,green] (-1,1)--(3,1);
\draw[-,line width=2pt,red] (-1,0)--(3,0);
\filldraw [green] (1,1) circle (2pt)
[green] (2,1) circle (2pt)
[green] (3,1) circle (2pt);
\filldraw [red] (1,0) circle (2pt)
[red] (2,0) circle (2pt)
[red] (3,0) circle (2pt);
\draw (1.5,1.3) node {$\widetilde X_1$};
\draw (2.5,1.3) node {$\widetilde X_2$};
\draw (1.5,-0.3) node {$\widetilde Y_2$};
\draw (2.5,-0.3) node {$\widetilde Y_1$};
\draw (0,1.3) node {$\widetilde P$};
\draw (0,-0.3) node {$\widetilde Q$};
\end{tikzpicture}

\noindent  be two paths spanning the BB, such that
  $X=\widetilde X_1+\widetilde X_2$ and $Y=\widetilde Y_1+\widetilde Y_2$ are the disjoint extremal components. Then
  $$
  X_1=\widetilde P+\widetilde X_1,\quad X_2=\widetilde X_2, \quad Y_1=\widetilde Y_1
  \quad\text{and}\quad Y_2=\widetilde Q+\widetilde Y_2
  $$
  satisfy the conditions of Definition~\ref{def LNC}.
\end{proof}

\begin{definition}
  A \textbf{locally non connectable internal building unit (LNC IBU)} is a LNC IBB, or the concatenation of LNC IBB's.
\end{definition}

\begin{proposition} \label{lista}
  The following four constructions yield LNC blocks, starting from some given LNC IBU's.
  \begin{enumerate}
    \item Let $B'$ be a LNC IBU with endpoints $d$ and $e$, as represented in the figure. We can embed $B'$
    as in the figure and obtain a LNC IBB  $B$ with endpoints $a$ and $b$.  The four new subpaths
    $W$ from $a$ to $e$, $X$ from $d$ to $b$, $Y$ from $a$ to $d$ and $Z$ from $e$ to $b$  with the given colors
    are the ESU of the new building block. Moreover,
    $$
    L(W)+L(X)=L(Y)+L(Z).
    $$

\begin{tikzpicture}[scale=0.5]
\draw[-,green] (2,1)--(3,0);
\draw[-,red] (4,1)--(3,0);
\draw[-,red] (2,3)--(3,4);
\draw[-,green] (4,3)--(3,4);
\filldraw [green] (2,1)--(3,1)--(3,2)--(2,2);
\filldraw [red] (3,1)--(4,1)--(4,2)--(3,2);
\filldraw [red] (2,2)--(3,2)--(3,3)--(2,3);
\filldraw [green] (3,2)--(4,2)--(4,3)--(3,3);
\draw (3,4.5) node {$e$};
\draw (3,-0.4) node {$d$};
\draw (7,5) node {};
\draw (3,-1.2) node {\normalfont The LNC IBU $B'$};
\filldraw [black] (3,4) circle (2pt)
[black] (3,0) circle (2pt);
\end{tikzpicture}
    \begin{tikzpicture}[scale=0.5]
\draw[-,green] (0,2)..controls (2,0)..(3,0);
\draw[-,red] (0,2)..controls (2,4)..(3,4);
\draw[-,red] (6,2)..controls (4,0)..(3,0);
\draw[-,green] (6,2)..controls (4,4)..(3,4);
\draw[-,green] (2,1)--(3,0);
\draw[-,red] (4,1)--(3,0);
\draw[-,red] (2,3)--(3,4);
\draw[-,green] (4,3)--(3,4);
\filldraw [green] (2,1)--(3,1)--(3,2)--(2,2);
\filldraw [red] (3,1)--(4,1)--(4,2)--(3,2);
\filldraw [red] (2,2)--(3,2)--(3,3)--(2,3);
\filldraw [green] (3,2)--(4,2)--(4,3)--(3,3);
\draw (-0.3,2) node {$a$};
\draw (6.3,2) node {$b$};
\draw (3,4.5) node {$e$};
\draw (3,-0.4) node {$d$};
\draw (1,3.8) node {$W$};
\draw (1,0.4) node {$Y$};
\draw (5,3.6) node {$Z$};
\draw (5,0.3) node {$X$};
\draw (3,-1.2) node {\normalfont The new IBB $B$ };
\filldraw [black] (0,2) circle (2pt)
[black] (6,2) circle (2pt);
\filldraw [black] (3,4) circle (2pt)
[black] (3,0) circle (2pt);
\end{tikzpicture}

    \item Let $B'$ be a LNC IBU with endpoints $d$ and $e$, as represented in the figure. We can embed $B'$
    as in the figure and obtain a LNC EBB $B$ with common endpoint $c$ and with two other endpoints $a$ and $b$.
    The four new subpaths with the given colors are the ESU of the new building block.

    \begin{tikzpicture}[scale=0.7]
\draw[-,green] (2,1)--(3,0);
\draw[-,red] (4,1)--(3,0);
\draw[-,red] (2,3)--(3,4);
\draw[-,green] (4,3)--(3,4);
\filldraw [green] (2,1)--(3,1)--(3,2)--(2,2);
\filldraw [red] (3,1)--(4,1)--(4,2)--(3,2);
\filldraw [red] (2,2)--(3,2)--(3,3)--(2,3);
\filldraw [green] (3,2)--(4,2)--(4,3)--(3,3);
\draw (3,4.5) node {$e$};
\draw (3,-0.4) node {$d$};
\draw (7,5) node {};
\draw (3,-1.2) node {\normalfont The LNC IBU $B'$};
\filldraw [black] (3,4) circle (2pt)
[black] (3,0) circle (2pt);
\end{tikzpicture}
    \begin{tikzpicture}[scale=0.7]
\draw[-,green] (0,2)..controls (2,0)..(3,0);
\draw[-,red] (0,2)..controls (2,4)..(3,4);
\draw[-,red] (5,0)..controls (4,-0.5)..(3,0);
\draw[-,green] (5,4)..controls (4,4.5)..(3,4);
\draw[-,green] (2,1)--(3,0);
\draw[-,red] (4,1)--(3,0);
\draw[-,red] (2,3)--(3,4);
\draw[-,green] (4,3)--(3,4);
\filldraw [green] (2,1)--(3,1)--(3,2)--(2,2);
\filldraw [red] (3,1)--(4,1)--(4,2)--(3,2);
\filldraw [red] (2,2)--(3,2)--(3,3)--(2,3);
\filldraw [green] (3,2)--(4,2)--(4,3)--(3,3);
\draw (-0.3,2) node {$c$};
\draw (5.3,4) node {$b$};
\draw (5.3,0) node {$a$};
\draw (3,4.5) node {$e$};
\draw (3,-0.4) node {$d$};
\draw (3,-1.2) node {\normalfont The new EBB $B$ };
\filldraw [black] (3,4) circle (2pt)
[black] (3,0) circle (2pt);
\filldraw [black] (0,2) circle (2pt)
[black] (5,4) circle (2pt)
[black] (5,0) circle (2pt);
\end{tikzpicture}

    \item Let $B'$ be an a LNC IBU with endpoints $d$ and $e$, as represented in the figure. We can embed $B'$
    as in the figure and obtain a LNC IBB $B$ with endpoints $a$ and $b$.
    The six new subpaths with the given colors are the ESU of the new building block.

    \begin{tikzpicture}[scale=0.5]
\draw[-,green] (3,4)--(4,3);
\draw[-,red] (6,3)--(7,4);
\draw[-,red] (3,4)--(4,5);
\draw[-,green] (6,5)--(7,4);
\filldraw [green] (4,3)--(5,3)--(5,4)--(4,4);
\filldraw [red] (5,3)--(6,3)--(6,4)--(5,4);
\filldraw [red] (4,4)--(5,4)--(5,5)--(4,5);
\filldraw [green] (5,4)--(6,4)--(6,5)--(5,5);
\draw (3,4.5) node {$d$};
\draw (7,4.5) node {$e$};
\draw (5,1) node {\normalfont The LNC IBU $B'$ };
\draw (9,8) node {};
\filldraw [black] (3,4) circle (2pt)
[black] (7,4) circle (2pt);
\end{tikzpicture}
    \begin{tikzpicture}[scale=0.5]
\draw[-,green] (0,2)..controls (4,0)..(5,0);
\draw[-,red] (0,2)..controls (2,4)..(3,4);
\draw[-,red] (10,2)..controls (6,0)..(5,0);
\draw[-,green] (10,2)..controls (8,4)..(7,4);
\draw[-,green] (5,0)..controls (3,3)..(3,4);
\draw[-,red] (5,0)..controls (7,3)..(7,4);
\draw[-,green] (3,4)--(4,3);
\draw[-,red] (6,3)--(7,4);
\draw[-,red] (3,4)--(4,5);
\draw[-,green] (6,5)--(7,4);
\filldraw [green] (4,3)--(5,3)--(5,4)--(4,4);
\filldraw [red] (5,3)--(6,3)--(6,4)--(5,4);
\filldraw [red] (4,4)--(5,4)--(5,5)--(4,5);
\filldraw [green] (5,4)--(6,4)--(6,5)--(5,5);
\draw (-0.3,2) node {$a$};
\draw (10.3,2) node {$b$};
\draw (3,4.5) node {$d$};
\draw (7,4.5) node {$e$};
\draw (5,-1) node {\normalfont The new IBB $B$ };
\filldraw [black] (3,4) circle (2pt)
[black] (7,4) circle (2pt)
[black] (0,2) circle (2pt)
[black] (10,2) circle (2pt)
[black] (5,0) circle (2pt);
\end{tikzpicture}

    \item Let $B'$ be an a LNC IBU with endpoints $d$ and $e$, as represented in the figure. We can embed $B'$
    as in the figure and obtain a LNC EBB $B$ with common endpoint $c$ and with two other endpoints $a$ and $b$.
    The six new subpaths with the given colors are the ESU of the new building block.

    
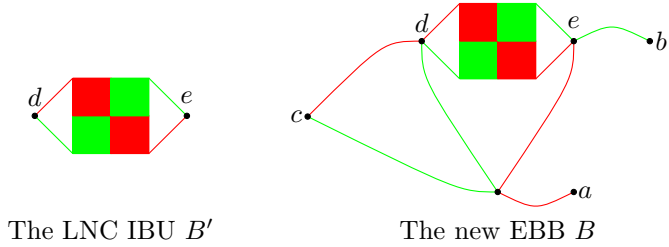
\begin{tikzpicture}[scale=0.5]
\draw[-,green] (3,4)--(4,3);
\draw[-,red] (6,3)--(7,4);
\draw[-,red] (3,4)--(4,5);
\draw[-,green] (6,5)--(7,4);
\filldraw [green] (4,3)--(5,3)--(5,4)--(4,4);
\filldraw [red] (5,3)--(6,3)--(6,4)--(5,4);
\filldraw [red] (4,4)--(5,4)--(5,5)--(4,5);
\filldraw [green] (5,4)--(6,4)--(6,5)--(5,5);
\draw (3,4.5) node {$d$};
\draw (7,4.5) node {$e$};
\draw (5,1) node {\normalfont The LNC IBU $B'$ };
\draw (9,8) node {};
\filldraw [black] (3,4) circle (2pt)
[black] (7,4) circle (2pt);
\end{tikzpicture}
    \begin{tikzpicture}[scale=0.5]
\draw[-,green] (0,2)..controls (4,0)..(5,0);
\draw[-,red] (0,2)..controls (2,4)..(3,4);
\draw[-,red] (7,0)..controls (6,-0.5)..(5,0);
\draw[-,green] (9,4)..controls (8,4.5)..(7,4);
\draw[-,green] (5,0)..controls (3,3)..(3,4);
\draw[-,red] (5,0)..controls (7,3)..(7,4);
\draw[-,green] (3,4)--(4,3);
\draw[-,red] (6,3)--(7,4);
\draw[-,red] (3,4)--(4,5);
\draw[-,green] (6,5)--(7,4);
\filldraw [green] (4,3)--(5,3)--(5,4)--(4,4);
\filldraw [red] (5,3)--(6,3)--(6,4)--(5,4);
\filldraw [red] (4,4)--(5,4)--(5,5)--(4,5);
\filldraw [green] (5,4)--(6,4)--(6,5)--(5,5);
\draw (-0.3,2) node {$c$};
\draw (7.3,0) node {$a$};
\draw (9.3,4) node {$b$};
\draw (3,4.5) node {$d$};
\draw (7,4.5) node {$e$};
\draw (5,-1) node {\normalfont The new EBB $B$ };
\filldraw [black] (3,4) circle (2pt)
[black] (7,4) circle (2pt)
[black] (0,2) circle (2pt)
[black] (9,4) circle (2pt)
[black] (7,0) circle (2pt)
[black] (5,0) circle (2pt);
\end{tikzpicture}

  \end{enumerate}
\end{proposition}

\begin{proof}
  (1) By assumption all the ESU's in $B'$ are LNC. All the pairs
  of different colors in the new ESU are adjacent, so Proposition~\ref{prop adyacentes} concludes the proof in this case.

  \noindent (2) By assumption, all the ESU's in $B'$ are LNC. On the other hand all the pairs of different colors of
  the new components are either adjacent or both extremal, so Propositions~\ref{prop adyacentes} and~\ref{PROP conexion entre extremos}
  conclude the proof in this case.

  \noindent (3) By assumption, all the ESU's in $B'$ are LNC. On the other hand all the pairs of
  the new components of different color except one, are adjacent, so Propositions~\ref{prop adyacentes} proves that they are LNC.
  The remaining pair is proven to be LNC by the paths $X_1,X_2,Y_1,Y_2$ in the following diagrams, that satisfy the conditions of
  Definition~\ref{def LNC}.

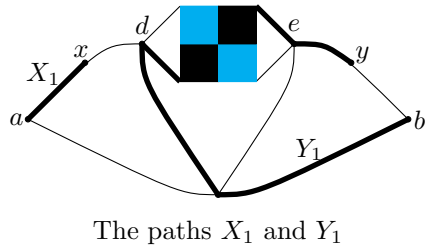
    \begin{tikzpicture}[scale=0.5]
\draw[-] (0,2)..controls (4,0)..(5,0);
\draw[-] (0,2)..controls (2,4)..(3,4);
\draw[-,line width=2pt] (10,2)..controls (6,0)..(5,0);
\draw[-] (10,2)..controls (8,4)..(7,4);
\draw[-,line width=2pt] (5,0)..controls (3,3)..(3,4);
\draw[-] (5,0)..controls (7,3)..(7,4);
\draw[-,line width=2pt] (3,4)--(4,3);
\draw[-] (6,3)--(7,4);
\draw[-] (3,4)--(4,5);
\draw[-,line width=2pt] (6,5)--(7,4);
\filldraw [black] (4,3)--(5,3)--(5,4)--(4,4);
\filldraw [cyan] (5,3)--(6,3)--(6,4)--(5,4);
\filldraw [cyan] (4,4)--(5,4)--(5,5)--(4,5);
\filldraw [black] (5,4)--(6,4)--(6,5)--(5,5);
\draw (-0.3,2) node {$a$};
\draw (10.3,2) node {$b$};
\draw (3,4.5) node {$d$};
\draw (7,4.5) node {$e$};
\draw (5,-1) node {\normalfont The paths $X_1$ and $Y_1$};
\draw[-,line width=2pt] (0,2)--(1.5,3.5);
\draw (0.4,3.2) node {$X_1$};
\draw (1.4,3.8) node {$x$};
\draw[-,line width=2pt] (7,4)..controls (7.8,4)..(8.5,3.5);
\draw (7.4,1.2) node {$Y_1$};
\draw (8.8,3.7) node {$y$};
\filldraw [black] (3,4) circle (2pt)
[black] (7,4) circle (2pt)
[black] (0,2) circle (2pt)
[black] (10,2) circle (2pt)
[black] (5,0) circle (2pt);
\draw (14,6) node{};
\filldraw [black] (1.5,3.5) circle (2pt)
[black] (8.5,3.5) circle (2pt);
\end{tikzpicture}
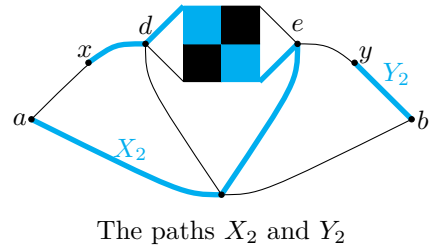
\begin{tikzpicture}[scale=0.5]
\draw[-,line width=2pt,cyan] (0,2)..controls (4,0)..(5,0);
\draw[-] (0,2)..controls (2,4)..(3,4);
\draw[-] (10,2)..controls (6,0)..(5,0);
\draw[-] (10,2)..controls (8,4)..(7,4);
\draw[-] (5,0)..controls (3,3)..(3,4);
\draw[-,line width=2pt,cyan] (5,0)..controls (7,3)..(7,4);
\draw[-] (3,4)--(4,3);
\draw[-,line width=2pt,cyan] (6,3)--(7,4);
\draw[-,line width=2pt,cyan] (3,4)--(4,5);
\draw[-] (6,5)--(7,4);
\filldraw [black] (4,3)--(5,3)--(5,4)--(4,4);
\filldraw [cyan] (5,3)--(6,3)--(6,4)--(5,4);
\filldraw [cyan] (4,4)--(5,4)--(5,5)--(4,5);
\filldraw [black] (5,4)--(6,4)--(6,5)--(5,5);
\draw (-0.3,2) node {$a$};
\draw (10.3,2) node {$b$};
\draw (3,4.5) node {$d$};
\draw (7,4.5) node {$e$};
\draw (5,-1) node {\normalfont The paths $X_2$ and $Y_2$ };
\draw[-,line width=2pt,cyan] (10,2)--(8.5,3.5);
\draw[cyan] (9.6,3.2) node {$Y_2$};
\draw (1.4,3.8) node {$x$};
\draw[-,line width=2pt,cyan] (3,4)..controls (2.2,4)..(1.5,3.5);
\draw[cyan] (2.6,1.2) node {$X_2$};
\draw (8.8,3.7) node {$y$};
\filldraw [black] (3,4) circle (2pt)
[black] (7,4) circle (2pt)
[black] (0,2) circle (2pt)
[black] (10,2) circle (2pt)
[black] (5,0) circle (2pt);
\filldraw [black] (1.5,3.5) circle (2pt)
[black] (8.5,3.5) circle (2pt);
\end{tikzpicture}

  \noindent (4) By assumption, the ESU's in $B'$ are LNC. On the other hand all the pairs of
  the new components of different color except one, are adjacent or both extremal, so
  Propositions~\ref{prop adyacentes} and~\ref{PROP conexion entre extremos}
   prove that they are LNC.
  The remaining pair is proven to be LNC by the paths in the following diagrams.

    
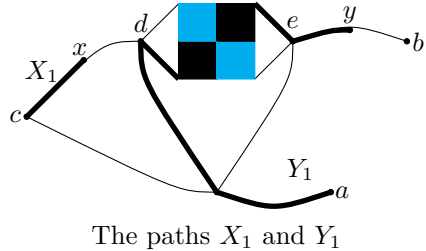
\begin{tikzpicture}[scale=0.5]
\draw[-] (0,2)..controls (4,0)..(5,0);
\draw[-] (0,2)..controls (2,4)..(3,4);
\draw[-,line width=2pt] (8,0)..controls (6.5,-0.5)..(5,0);
\draw[-] (10,4)..controls (8.5,4.5)..(7,4);
\draw[-,line width=2pt] (5,0)..controls (3,3)..(3,4);
\draw[-] (5,0)..controls (7,3)..(7,4);
\draw[-,line width=2pt] (3,4)--(4,3);
\draw[-] (6,3)--(7,4);
\draw[-] (3,4)--(4,5);
\draw[-,line width=2pt] (6,5)--(7,4);
\filldraw [black] (4,3)--(5,3)--(5,4)--(4,4);
\filldraw [cyan] (5,3)--(6,3)--(6,4)--(5,4);
\filldraw [cyan] (4,4)--(5,4)--(5,5)--(4,5);
\filldraw [black] (5,4)--(6,4)--(6,5)--(5,5);
\draw (-0.3,2) node {$c$};
\draw (10.3,4) node {$b$};
\draw (8.3,0) node {$a$};
\draw (3,4.5) node {$d$};
\draw (7,4.5) node {$e$};
\draw (5,-1.2) node {\normalfont The paths $X_1$ and $Y_1$};
\draw[-,line width=2pt] (0,2)--(1.5,3.5);
\draw (0.4,3.2) node {$X_1$};
\draw (1.4,3.8) node {$x$};
\draw[-,line width=2pt] (7,4)..controls (8,4.3)..(8.5,4.3);
\draw (7.2,0.6) node {$Y_1$};
\draw (8.5,4.7) node {$y$};
\filldraw [black] (3,4) circle (2pt)
[black] (7,4) circle (2pt)
[black] (0,2) circle (2pt)
[black] (8,0) circle (2pt)
[black] (10,4) circle (2pt)
[black] (5,0) circle (2pt);
\draw (14,6) node{};
\filldraw [black] (1.5,3.5) circle (2pt)
[black] (8.5,4.3) circle (2pt);
\end{tikzpicture}
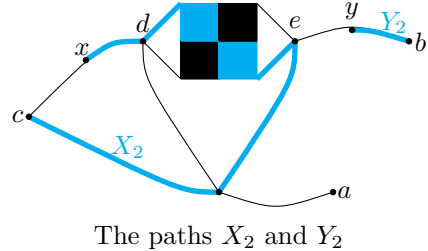
\begin{tikzpicture}[scale=0.5]
\draw[-,line width=2pt,cyan] (0,2)..controls (4,0)..(5,0);
\draw[-] (0,2)..controls (2,4)..(3,4);
\draw[-] (8,0)..controls (6.5,-0.5)..(5,0);
\draw[-] (10,4)..controls (8.5,4.5)..(7,4);
\draw[-] (5,0)..controls (3,3)..(3,4);
\draw[-,line width=2pt,cyan] (5,0)..controls (7,3)..(7,4);
\draw[-] (3,4)--(4,3);
\draw[-,line width=2pt,cyan] (6,3)--(7,4);
\draw[-,line width=2pt,cyan] (3,4)--(4,5);
\draw[-] (6,5)--(7,4);
\filldraw [black] (4,3)--(5,3)--(5,4)--(4,4);
\filldraw [cyan] (5,3)--(6,3)--(6,4)--(5,4);
\filldraw [cyan] (4,4)--(5,4)--(5,5)--(4,5);
\filldraw [black] (5,4)--(6,4)--(6,5)--(5,5);
\draw (-0.3,2) node {$c$};
\draw (10.3,4) node {$b$};
\draw (8.3,0) node {$a$};
\draw (3,4.5) node {$d$};
\draw (7,4.5) node {$e$};
\draw (5,-1.2) node {\normalfont The paths $X_2$ and $Y_2$ };
\draw[-,line width=2pt,cyan] (10,4)..controls (9,4.3)..(8.5,4.3);
\draw[cyan] (9.6,4.5) node {$Y_2$};
\draw (1.4,3.8) node {$x$};
\draw[-,line width=2pt,cyan] (3,4)..controls (2.2,4)..(1.5,3.5);
\draw[cyan] (2.6,1.2) node {$X_2$};
\draw (8.5,4.8) node {$y$};
\filldraw [black] (3,4) circle (2pt)
[black] (7,4) circle (2pt)
[black] (0,2) circle (2pt)
[black] (10,4) circle (2pt)
[black] (8,0) circle (2pt)
[black] (5,0) circle (2pt);
\filldraw [black] (1.5,3.5) circle (2pt)
[black] (8.5,4.3) circle (2pt);
\end{tikzpicture}

\end{proof}

\section{Bi-traceable graphs with $\#(V(P)\cap V(Q))\le 4$}

In this section we will prove in Theorem~\ref{k menor igual a 4} that if $G$ is a simple graph, $P,Q$ are longest paths and $\# V(P)\cap V(Q) \le 4$,
then $BT(P,Q)$ is either a concatenation of LNC blocks,
or it is totally disconnected (TD), according to the following definition.

\begin{definition} \label{TD}
  $BT(P,Q)$  is \textbf{totally disconnected (TD)}, if all pairs of components are NDC.
\end{definition}

\begin{remark}\label{en TD todos son NC}
  If $BT(P,Q)$ is TD and $V(P)\ne V(Q)$, then
  $V(P)\cap V(Q)$ is a separator.
\end{remark}

Clearly, if $\ell=\# (V(P)\cap V(Q))=1$, then $BT(P,Q)$ is TD. When $\ell=2$, then the representation of the resulting BT-graph is a
concatenation of LNC
blocks.

When $\ell=3$, then there are two
different representation types for $BT(P,Q)$. Either it is the concatenation of two elementary IBB's and eventually up to two elementary EBB's,
or it is the concatenation of an elementary EBB and an EBB as in
Proposition~\ref{lista}(2). For this we write $(i,j,k)$ for the permutation $\sigma$ with $\sigma(1)=i$, $\sigma(2)=j$ and $\sigma(3)=k$, and
similarly for $\ell>3$.
Consider the equivalence relation generated by $\sigma\sim \sigma^{-1}$ and $\sigma\sim \sigma^\bot$, where $\sigma^{\bot}(j)=\sigma(\ell-j)$.

We will draw generic representations for each equivalence class of permutations for $\ell=3,4,5$. This means that in each completed component
 the number of vertices is left open. Note that two equivalent permutations $(i_1,j_1,k_1)\sim (i_2,j_2,k_2)$
have the same generic representation.

\noindent For the identity permutation $(1,2,3)$ we have
$$
(1,2,3)\sim (3,2,1)
\text{ and the (generic) representation of $BT(P,Q)$ is \ \ }
{\vcenter{\hbox{\begin{tikzpicture}[scale=0.3]
\draw[-,green] (1,2.6)..controls (1.3,2.6)..(2,2);
\draw[-,green] (2,2)..controls (3,2.8)..(4,2);
\draw[-,green] (4,2)..controls (5,2.8)..(6,2);
\draw[-,green] (6,2)..controls (6.7,2.6)..(7,2.6);
\draw[-,red] (1,1.4)..controls (1.3,1.4)..(2,2);
\draw[-,red] (2,2)..controls (3,1.2)..(4,2);
\draw[-,red] (4,2)..controls (5,1.2)..(6,2);
\draw[-,red] (6,2)..controls (6.7,1.4)..(7,1.4);
\filldraw [black]  (2,2)    circle (2pt)
[black]  (4,2)    circle (2pt)
[black]  (6,2)    circle (2pt);
\end{tikzpicture}}}}.
$$
For the remaining permutations we have
$$
(1,3,2)\sim (2,1,3)\sim (2,3,1)\sim (3,1,2)
\text{ and the representation of the BT-graph is \ \ }
{\vcenter{\hbox{\begin{tikzpicture}[scale=0.3]
\draw[-,green] (3,2.6)..controls (3.3,2.6)..(4,2);
\draw[-,green] (4,2)..controls (5.6,3)..(6,3);
\draw[-,green] (6,3)..controls (6.8,2)..(6,1);
\draw[-,red] (6,3)..controls (5.2,2)..(6,1);
\draw[-,red] (3,1.4)..controls (3.3,1.4)..(4,2);
\draw[-,red] (4,2)..controls (5.6,1)..(6,1);
\draw[-,red] (6,3)..controls (6.5,3)..(7,2.7);
\draw[-,green] (6,1)..controls (6.5,1)..(7,1.3);
\filldraw [black]  (4,2)    circle (2pt)
[black]  (6,3)    circle (2pt)
[black]  (6,1)    circle (2pt);
\end{tikzpicture}}}}.
$$

\subsection*{The case $\ell=4$}
Note that if $\ell=4$, then $\sigma^{\bot}=(4,3,2,1)\circ \sigma$, where $(4,3,2,1)$ corresponds to $(14)(23)$ in the standard form, so
our equivalence relation is the same as in~\cite{CCP}. We
 have seven classes, which coincide with the cases in~\cite{CCP}, and are listed in the following table.

 \smallskip
$$
\begin{tabular}{|c|c|c|c|}
  \hline
Case&  Permutations & Representation & Conn.\\ \hline
 1.&  $(1,2,3,4),(4,3,2,1)$&
   \begin{tikzpicture}[scale=0.3]
\draw[-,green] (1,2.6)..controls (1.3,2.6)..(2,2);
\draw[-,green] (2,2)..controls (3,2.8)..(4,2);
\draw[-,green] (4,2)..controls (5,2.8)..(6,2);
\draw[-,green] (6,2)..controls (7,2.8)..(8,2);
\draw[-,green] (8,2)..controls (8.7,2.6)..(9,2.6);
\draw[-,red] (1,1.4)..controls (1.3,1.4)..(2,2);
\draw[-,red] (2,2)..controls (3,1.2)..(4,2);
\draw[-,red] (4,2)..controls (5,1.2)..(6,2);
\draw[-,red] (6,2)..controls (7,1.2)..(8,2);
\draw[-,red] (8,2)..controls (8.7,1.4)..(9,1.4);
\filldraw [black]  (2,2)    circle (2pt)
[black]  (4,2)    circle (2pt)
[black]  (6,2)    circle (2pt)
[black]  (8,2)    circle (2pt);
\draw (2,3)node{};
\end{tikzpicture}
& LNC \\  \hline
2.& $(1,2,4,3),(3,4,2,1),(4,3,1,2),(2,1,3,4)$&
\begin{tikzpicture}[scale=0.3]
\draw[-,green] (1,2.6)..controls (1.3,2.6)..(2,2);
\draw[-,green] (2,2)..controls (3,2.8)..(4,2);
\draw[-,green] (4,2)..controls (5.6,3)..(6,3);
\draw[-,green] (6,3)..controls (6.8,2)..(6,1);
\draw[-,red] (6,3)..controls (5.2,2)..(6,1);
\draw[-,red] (1,1.4)..controls (1.3,1.4)..(2,2);
\draw[-,red] (2,2)..controls (3,1.2)..(4,2);
\draw[-,red] (4,2)..controls (5.6,1)..(6,1);
\draw[-,red] (6,3)..controls (6.5,3)..(7,2.7);
\draw[-,green] (6,1)..controls (6.5,1)..(7,1.3);
\filldraw [black]  (2,2)    circle (2pt)
[black]  (4,2)    circle (2pt)
[black]  (6,3)    circle (2pt)
[black]  (6,1)    circle (2pt);
\end{tikzpicture}
& LNC \\  \hline
3. &$(1,3,2,4),(4,2,3,1)$&
\begin{tikzpicture}[scale=0.3]
\draw[-,green] (1,2.6)..controls (1.3,2.6)..(2,2);
\draw[-,green] (2,2)..controls (3.6,3)..(4,3);
\draw[-,green] (4,3)..controls (3.2,2)..(4,1);
\draw[-,red] (1,1.4)..controls (1.3,1.4)..(2,2);
\draw[-,red] (2,2)..controls (3.6,1)..(4,1);
\draw[-,red] (7,2.6)..controls (6.7,2.6)..(6,2);
\draw[-,red] (6,2)..controls (4.4,3)..(4,3);
\draw[-,red] (4,3)..controls (4.8,2)..(4,1);
\draw[-,green] (7,1.4)..controls (6.7,1.4)..(6,2);
\draw[-,green] (6,2)..controls (4.4,1)..(4,1);
\filldraw [black]  (2,2)    circle (2pt)
[black]  (4,3)    circle (2pt)
[black]  (6,2)    circle (2pt)
[black]  (4,1)    circle (2pt);
\end{tikzpicture}
& LNC \\  \hline
4. & $\begin{array}{c}
  (1,3,4,2), (1,4,2,3), (2,3,1,4), (2,4,3,1),\\
   (3,1,2,4), (3,2,4,1), (4,1,3,2), (4,2,1,3)
\end{array}$&
\begin{tikzpicture}[scale=0.3]
\draw[-,green] (1,2.6)..controls (1.3,2.6)..(2,2);
\draw[-,green] (2,2)..controls (3.6,3)..(4,3);
\draw[-,green] (4,3) -- (4,1);
\draw[-,red] (4,1)..controls (5,2)..(6,2);
\draw[-,red] (1,1.4)..controls (1.3,1.4)..(2,2);
\draw[-,red] (2,2)..controls (3.6,1)..(4,1);
\draw[-,green] (7,2.6)..controls (6.7,2)..(6,2);
\draw[-,red] (6,2)..controls (4.4,3)..(4,3);
\draw[-,red] (4,3)..controls (4.8,3.3)..(5.6,3);
\draw[-,green] (6,2)..controls (5.4,1)..(4,1);
\filldraw [black]  (2,2)    circle (2pt)
[black]  (4,3)    circle (2pt)
[black]  (6,2)    circle (2pt)
[black]  (4,1)    circle (2pt);
\draw (2,3.5) node{};
\end{tikzpicture}
& LNC \\  \hline
5. &$(1,4,3,2), (2,3,4,1), (4,1,2,3), (3,2,1,4)$&
\begin{tikzpicture}[scale=0.3]
\draw[-,red] (2,2)..controls (2.8,3)..(2,4);
\draw[-,red] (2,4)..controls (2.8,5)..(2,6);
\draw[-,red] (2,2)..controls (1.2,2)..(0,4);
\draw[-,green] (2,6)..controls (1.2,6)..(0,4);
\draw[-,green] (-1,4.6)..controls (-0.6,4.6)..(0,4);
\draw[-,red] (-1,3.4)..controls (-0.6,3.4)..(0,4);
\draw[-,green] (3,2.5)..controls (2.7,2)..(2,2);
\draw[-,green] (2,2)..controls (1.2,3)..(2,4);
\draw[-,green] (2,4)..controls (1.2,5)..(2,6);
\draw[-,red] (2,6)..controls (2.7,6)..(3,5.5);
\filldraw [black]  (2,2)    circle (2pt)
[black]  (2,4)    circle (2pt)
[black]  (2,6)    circle (2pt)
[black]  (0,4)    circle (2pt);
\end{tikzpicture}
& LNC \\  \hline
6. & $(2,1,4,3), (3,4,1,2)$&
\begin{tikzpicture}[scale=0.3]
\draw[-,red] (2,2)..controls (3,2.8)..(4,2);
\draw[-,green] (1,1.4)..controls (1.3,1.4)..(2,2);
\draw[-,green] (2,2)..controls (3,1.2)..(4,2);
\draw[-,red] (4,2)..controls (4.7,1.4)..(5,1.4);
\draw[-,green] (2,4)..controls (3,3.2)..(4,4);
\draw[-,green] (1,4.6)..controls (1.3,4.6)..(2,4);
\draw[-,red] (2,4)..controls (3,4.8)..(4,4);
\draw[-,red] (4,4)..controls (4.7,4.6)..(5,4.6);
\draw[-,green] (4,2)--(4,4);
\draw[-,red] (2,2)--(2,4);
\filldraw [black]  (2,2)    circle (2pt)
[black]  (4,2)    circle (2pt)
[black]  (2,4)    circle (2pt)
[black]  (4,4)    circle (2pt);
\end{tikzpicture}
&LNC
\\  \hline
7. & $(2,4,1,3), (3,1,4,2)$
&
\begin{tikzpicture}[scale=0.3]
\draw[-,green] (2,2) --(4,4);
\draw[-,red] (2,6) --(4,4);
\draw[-,red] (6,4) --(4,4);
\draw[-,green] (3,4) --(4,4);
\draw[-,red] (6,4) --(7,4);
\draw[-,green] (2,2)..controls (4.5,2)..(6,4);
\draw[-,green] (2,6)..controls (4.5,6)..(6,4);
\draw[-,red] (2,2)..controls (1,4)..(2,6);
\draw[-,red] (2,2) --(1,2);
\draw[-,green] (2,6) --(1,6);
\filldraw [black]  (2,2)    circle (2pt)
[black]  (2,6)    circle (2pt)
[black]  (4,4)    circle (2pt)
[black]  (6,4)    circle (2pt);
\draw(2,6.3)node{};
\end{tikzpicture}
& TD \\  \hline
\end{tabular}
$$

\begin{theorem}\label{Teorema k menor igual a 4}
  If $G$ is a simple graph, $P,Q$ are longest paths and $\# V(P)\cap V(Q) \le 4$, then $BT(P,Q)$
  is either a concatenation of LNC blocks, or it is TD.
\end{theorem}

\begin{proof}
  For $\ell=1$ and $\ell=2$, we know by the discussion after Definition~\ref{TD} that $BT(P,Q)$ is TD.
   For $\ell=3$ and for the first five subcases of the case $\ell=4$, $BT(P,Q)$
  is the concatenation of LNC blocks. In fact, the first case for $\ell=3$ and the first case for $\ell=4$ is a concatenation of
  elementary building blocks which are LNC (for example by Proposition~\ref{prop adyacentes}). The second case for $\ell=3$ and the cases 2. and 5.
  for $\ell=4$ have additional blocks corresponding to Proposition~\ref{lista}(2). The cases 3. and 4. for $\ell=4$ have additional blocks
  corresponding to
  items~(1) and~(4) of Proposition~\ref{lista}. So it remains to prove that for $\ell=4$ in Case 6. $BT(P,Q)$ is LNC and in Case~7. it is TD.
  In the sixth case there is one BB, with two embedded elementary IBB's and an ESU with six completed components. By
  Proposition~\ref{prop adyacentes}, the two elementary ESU's  are LNC.

\noindent \begin{tikzpicture}[scale=1]
\draw[-,red] (2,2)..controls (3,2.8)..(4,2);
\draw[-,line width=2pt,green] (1,1.4)..controls (1.3,1.4)..(2,2);
\draw[-,green] (2,2)..controls (3,1.2)..(4,2);
\draw[-,line width=2pt,red] (4,2)..controls (4.7,1.4)..(5,1.4);
\draw[-,green] (2,4)..controls (3,3.2)..(4,4);
\draw[-,line width=2pt,green] (1,4.6)..controls (1.3,4.6)..(2,4);
\draw[-,red] (2,4)..controls (3,4.8)..(4,4);
\draw[-,red,line width=2pt] (4,4)..controls (4.7,4.6)..(5,4.6);
\draw[-,line width=2pt,green] (4,2)--(4,4);
\draw[-,red,line width=2pt] (2,2)--(2,4);
\filldraw [black]  (2,2)    circle (2pt)
[black]  (4,2)    circle (2pt)
[black]  (2,4)    circle (2pt)
[black]  (4,4)    circle (2pt);
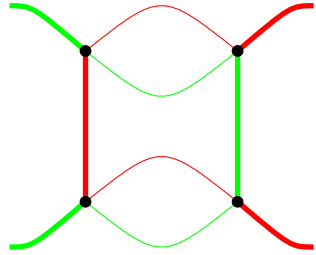
\draw (3,0.8) node {ESU with six edges};
\draw(7,2)node{};
\draw(0,2)node{};
\end{tikzpicture}
\begin{tikzpicture}[scale=1]
\draw[-,line width=2pt] (2,2)..controls (3,2.5)..(4,2);
\draw[-,line width=2pt] (1,1.4)..controls (1.3,1.4)..(2,2);
\draw[-,line width=2pt,cyan] (2,2)..controls (3,1.5)..(4,2);
\draw[-,line width=2pt,cyan] (4,2)..controls (4.7,1.4)..(5,1.4);
\draw[-,line width=2pt,cyan] (2,4)..controls (3,3.5)..(4,4);
\draw[-,line width=2pt,cyan] (1,4.6)..controls (1.3,4.6)..(2,4);
\draw[-,line width=2pt] (2,4)..controls (3,4.5)..(4,4);
\draw[-,line width=2pt] (4,4)..controls (4.7,4.6)..(5,4.6);
\draw[-,line width=2pt,cyan] (4,3)--(4,4);
\draw[-,line width=2pt] (4,3)--(4,2);
\draw[-,line width=2pt,cyan] (2,2)--(2,3);
\draw[-,line width=2pt] (2,4)--(2,3);
\filldraw [black]  (2,2)    circle (2pt)
[black]  (4,2)    circle (2pt)
[black]  (2,4)    circle (2pt)
[black]  (4,4)    circle (2pt);
\filldraw [red]  (4,3)    circle (1.5pt)
[red]  (2,3)    circle (1.5pt);
\draw (1.7,3) node {$x$};
\draw (4.3,3) node {$y$};
\draw (4.1,4.5) node {$X_1$};
\draw (0.8,1.5) node {$Y_1$};
\draw[blue] (0.8,4.6) node {$Y_2$};
\draw[blue] (4.1,1.5) node {$X_2$};
\draw (3,0.8) node {$X_1$ and $Y_1$ in black, $X_2$ and $Y_2$ in blue};
\draw (7,5) node {};
\end{tikzpicture}

In the ESU with six completed components all the pairs of
  the components of different color except one, are adjacent or both extremal, so
  by Propositions~\ref{prop adyacentes} and~\ref{PROP conexion entre extremos} they are LNC.
  The remaining pair is proven to be LNC by the paths in the second diagram.

  Finally we prove that $BT(P,Q)$ in Case 7 is TD. We claim that adjacent components are NDC, and that any pair of extremal components is also NDC.
  In fact, since the graph is
  symmetric, we can assume that these pairs are of different colors and so Lemmas~\ref{lema 4.1} and~\ref{conexion entre extremos} prove the claim.
  Again by symmetry, we are left with two cases:
  \begin{enumerate}
    \item[a)] Either one component is extremal and the other is not adjacent and internal, or
    \item[b)] both components are internal and not adjacent.
  \end{enumerate}
  The following diagrams show that in both cases the pairs are NDC.

 \begin{tikzpicture}[scale=0.7]
\draw[-] (2,2) --(4,4);
\draw[-] (2,2) --(1,2);
\draw[-] (6,2) --(4,4);
\draw[-] (6,2) --(7,2);
\draw[-] (4,7) --(4,3);
\draw[-] (2,2)..controls (2,4.5)..(4,6);
\draw[-] (6,2)..controls (6,4.5)..(4,6);
\draw[-] (2,2)..controls (4,1)..(6,2);
\draw[-,red] (4.04,3.4)--(4.04,3);
\draw[-,red] (4.04,3.4)..controls (4.8,2.5)..(4.8,1.48);
\draw[-,red] (6,2.05)--(4.8,1.48);
\draw[-,red] (5.95,2)..controls (5.95,4.5)..(3.95,6);
\draw[-,red] (4.04,4)--(4.04,6);
\draw[-,red] (2,1.95) --(4,3.95);
\draw[-,red] (2,1.95) --(1,1.95);
\draw[-,green] (3.95,7) --(3.95,6);
\draw[-,green] (1.95,2)..controls (1.95,4.5)..(3.95,6);
\draw[-,green] (2,2.05)..controls (3.85,1.2)..(4.8,1.48);
\draw[-,green] (4.09,3.4)..controls (4.85,2.5)..(4.85,1.48);
\draw[-,green] (4.05,3.4) --(4.05,4);
\draw[-,green] (6,1.95) --(4,3.95);
\draw[-,green] (6,1.95) --(7,1.95);
\filldraw [black]  (2,2)    circle (2pt)
[black]  (6,2)    circle (2pt)
[black]  (4,4)    circle (2pt)
[black]  (4,6)    circle (2pt);
\filldraw [red]  (4.04,3.4)    circle (1pt)
[red]  (4.8,1.48)    circle (1pt);
\draw(4,0.8) node {Case a), with $\widehat P$ in green, $\widehat Q$ in red};
\draw (9,8) node{};
\end{tikzpicture}
 \begin{tikzpicture}[scale=0.7]
\draw[-] (2,2) --(4,4);
\draw[-] (2,2) --(1,2);
\draw[-] (6,2) --(4,4);
\draw[-] (6,2) --(7,2);
\draw[-] (4,7) --(4,3);
\draw[-] (2,2)..controls (2,4.5)..(4,6);
\draw[-] (6,2)..controls (6,4.5)..(4,6);
\draw[-] (2,2)..controls (4,1)..(6,2);
\draw[-,red] (6,2.05)--(4.8,1.48);
\draw[-,red] (5.95,2)..controls (5.95,4.5)..(3.95,6);
\draw[-,green] (2,1.95) --(4,3.95);
\draw[-,red] (2,1.95) --(1,1.95);
\draw[-,green] (6,1.95) --(4,3.95);
\draw[-,red] (1.95,2)..controls (1.95,4.5)..(3.95,6);
\draw[-,green] (2,2.05)..controls (3.85,1.2)..(4.8,1.48);
\draw[-,line width=3pt,white] (4.065,4.8)..controls (4.825,3.5)..(4.825,1.48);
\draw[-,red] (4.09,4.8)..controls (4.85,3.5)..(4.85,1.48);
\draw[-,green] (4.04,4.8)..controls (4.8,3.5)..(4.8,1.48);
\draw[-,green] (4.05,4.8) --(4.05,7);
\draw[-,red] (4.04,4.8)--(4.04,3);
\draw[-,green] (6,1.95) --(7,1.95);
\filldraw [black]  (2,2)    circle (2pt)
[black]  (6,2)    circle (2pt)
[black]  (4,4)    circle (2pt)
[black]  (4,6)    circle (2pt);
\filldraw [red]  (4.04,4.8)    circle (1pt)
[red]  (4.8,1.48)    circle (1pt);
\draw(4,0.8) node {Case b), with $\widehat P$ in green, $\widehat Q$ in red};
\end{tikzpicture}

\noindent  This concludes Case 7 and thus finishes the proof.
\end{proof}

\begin{corollary} \label{k menor igual a 4}
  If $G$ is a simple graph, $P,Q$ are longest paths, $V(P)\ne V(Q)$ and $\# V(P)\cap V(Q) \le 4$, then $V(P)\cap V(Q)$ is a separator.
\end{corollary}

\begin{proof}
  By Theorem~\ref{Teorema k menor igual a 4} and Proposition~\ref{ESU en LNC son NC}
  in the first six cases all the ESU's in $BT(P,Q)$ are
  NC. Hence, Proposition~\ref{isu nc implica articulation set} implies that
  $V(P)\cap V(Q)$ is a separator, as desired. Case 7 follows from Remark~\ref{en TD todos son NC}.
\end{proof}

\begin{corollary}\label{Corollary Hippchen 5}
  Assume that $P$ and $Q$ are two longest paths in a $5$-connected simple graph $G$. Then $\# (V(P)\cap V(Q))\ge 5$.
\end{corollary}

\begin{proof}
  We know that $\# V(P)\ge 5$, so we can assume that $V(P)\ne V(Q)$.
  Since a $5$-connected graph is also $4$-connected, by~\cite{G}*{Theorem 4.2} we know that $\# (V(P)\cap V(Q))\ge 4$.
  Assume by contradiction that $\# (V(P)\cap V(Q))<5$, i.e., that $\# (V(P)\cap V(Q))= 4$. Since $G$ is $5$-connected,
  the complement of $V(P)\cap V(Q)$ is connected, which contradicts the fact that by Corollary~\ref{k menor igual a 4}
  $V(P)\cap V(Q)$ is a separator. This contradiction concludes the proof.
\end{proof}

\begin{proposition}\label{tres caminos NC}
Let $P,Q,R$ be three longest paths in a graph $G$.
  If all the ESU's in $BT(P,Q)$ are NC, or $BT(P,Q)$ is TD, then the intersection of the three longest paths $V(P)\cap V(Q)\cap V(R)$ is not empty.
\end{proposition}

\begin{proof}
  On one hand, $R$
  must touch in at least one ESU the components of different colors, since otherwise we could swap the colors conveniently and obtain a longest
  path $\widetilde Q$ such that $\widetilde Q\cap R=\emptyset$, which is impossible. On the other hand, it is impossible that $R$ touches two
  components of different colors in an ESU. In fact, if all ESU's are NC, then this follows from Proposition~\ref{prop no path exists},
  and if $BT(P,Q)$ is TD, then $R$ can touch only one component of $BT(P,Q)$.
\end{proof}

\begin{corollary} \label{three paths}
  If the intersection of three longest paths $P,Q,R$ is empty, then $\#(V(P)\cap V(Q))\ge 5$.
\end{corollary}

\begin{proof}
  If $\#(V(P)\cap V(Q))< 5$, then by Theorem~\ref{Teorema k menor igual a 4} we know that in $BT(P,Q)$ all ESU's are NC,
  or $BT(P,Q)$ is TD.
  Proposition~\ref{tres caminos NC} concludes the proof.
\end{proof}

\begin{remark}
  The following representation of a BT-graph with $\ell=V(P)\cap V(Q)=5$ shows that the method for $\ell\le 4$ cannot be carried over to this case,
  since the highlighted ESU is not NC. In fact, the blue path $R$ connects two components of different colors in the same highlighted ESU,
  but cannot be completed to two paths whose lengths sum $2L(R)+2L(P)$, contradicting Definition~\ref{def NC}.
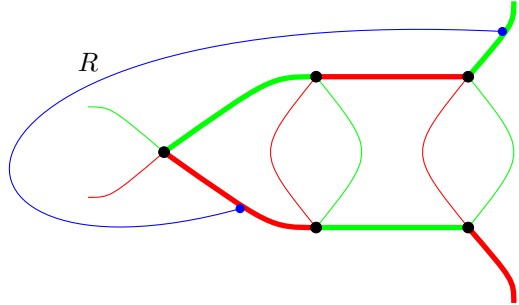
  $$
\begin{tikzpicture}[scale=1]
\draw[-,green] (2,2)..controls (2.8,3)..(2,4);
\draw[-,line width=2pt,red] (0,3)..controls (1.4,2)..(2,2);
\draw[-,red] (2,2)..controls (1.2,3)..(2,4);
\draw[-,line width=2pt,green] (2,4)..controls (1.4,4)..(0,3);
\draw[-,red] (4,2)..controls (3.2,3)..(4,4);
\draw[-,green] (0,3)..controls (-0.7,3.6)..(-1,3.6);
\draw[-,red] (0,3)..controls (-0.7,2.4)..(-1,2.4);
\draw[-,line width=2pt,red] (4.6,1)..controls (4.6,1.3)..(4,2);
\draw[-,green] (4,2)..controls (4.8,3)..(4,4);
\draw[-,line width=2pt,green] (4,4)..controls (4.6,4.7)..(4.6,5);
\draw[-,line width=2pt,red] (2,4)--(4,4);
\draw[-,line width=2pt,green] (2,2)--(4,2);
\filldraw [black]  (2,2)    circle (2pt)
[black]  (4,2)    circle (2pt)
[black]  (2,4)    circle (2pt)
[black]  (0,3)    circle (2pt)
[black]  (4,4)    circle (2pt);
\filldraw [blue]  (4.45,4.6)    circle (1.5pt)
[blue]  (1,2.25)    circle (1.5pt);
\draw[-,blue] (4.45,4.6)..controls (-3.5,5) and (-3.5,1)..(1,2.25);
\draw (-1,4.2) node {$R$};
\end{tikzpicture}
$$
However, we will prove in Corollary~\ref{Corollary Hippchen 6},
that this graph is not a counterexample to the Hippchen conjecture. We will even prove that
in this graph $V(P)\cap V(Q)$ is a separator in Theorem~\ref{teorema principal}.
\end{remark}

\section{Bi-traceable graphs with $\#(V(P)\cap V(Q))=5$}
\label{section table}
In this section we will prove in Theorem~\ref{k igual a 5} that if $G$ is a simple graph, $P,Q$ are longest paths and $\# V(P)\cap V(Q) =5$,
then $BT(P,Q)$ is either a concatenation of LNC building blocks,
or it is TD, or it is one of the three cases 6., 13. or 14. in the following table.

\medskip

\begin{tabular}{|c|c|c|c|}
  \hline
Case&  Permutations & Representation & Conn.\\ \hline
 1.&  $(1,2,3,4,5),(5,4,3,2,1)$&
   \begin{tikzpicture}[scale=0.3]
\draw[-,green] (1,2.6)..controls (1.3,2.6)..(2,2);
\draw[-,green] (2,2)..controls (3,2.8)..(4,2);
\draw[-,green] (4,2)..controls (5,2.8)..(6,2);
\draw[-,green] (6,2)..controls (7,2.8)..(8,2);
\draw[-,green] (8,2)..controls (9,2.8)..(10,2);
\draw[-,green] (10,2)..controls (10.7,2.6)..(11,2.6);
\draw[-,red] (1,1.4)..controls (1.3,1.4)..(2,2);
\draw[-,red] (2,2)..controls (3,1.2)..(4,2);
\draw[-,red] (4,2)..controls (5,1.2)..(6,2);
\draw[-,red] (6,2)..controls (7,1.2)..(8,2);
\draw[-,red] (8,2)..controls (9,1.2)..(10,2);
\draw[-,red] (10,2)..controls (10.7,1.4)..(11,1.4);
\filldraw [black]  (2,2)    circle (2pt)
[black]  (4,2)    circle (2pt)
[black]  (6,2)    circle (2pt)
[black]  (8,2)    circle (2pt)
[black]  (10,2)    circle (2pt);
\draw (2,3)node{};
\end{tikzpicture}
& LNC \\  \hline
2.& $(1,2,3,5,4),(4,5,3,2,1),(2,1,3,4,5),(5,4,3,1,2)$&
\begin{tikzpicture}[scale=0.3]
\draw[-,green] (-1,2.6)..controls (-0.7,2.6)..(0,2);
\draw[-,green] (0,2)..controls (1,2.8)..(2,2);
\draw[-,green] (2,2)..controls (3,2.8)..(4,2);
\draw[-,green] (4,2)..controls (5.6,3)..(6,3);
\draw[-,green] (6,3)..controls (6.8,2)..(6,1);
\draw[-,red] (6,3)..controls (5.2,2)..(6,1);
\draw[-,red] (-1,1.4)..controls (-0.7,1.4)..(0,2);
\draw[-,red] (0,2)..controls (1,1.2)..(2,2);
\draw[-,red] (2,2)..controls (3,1.2)..(4,2);
\draw[-,red] (4,2)..controls (5.6,1)..(6,1);
\draw[-,red] (6,3)..controls (6.7,4)..(7,4);
\draw[-,green] (6,1)..controls (6.7,0)..(7,0);
\filldraw [black]  (0,2)    circle (2pt)
[black]  (2,2)    circle (2pt)
[black]  (4,2)    circle (2pt)
[black]  (6,3)    circle (2pt)
[black]  (6,1)    circle (2pt);
\end{tikzpicture}
& LNC \\  \hline
3. &$(1,2,4,3,5),(5,3,4,2,1),(1,3,2,4,5),(5,4,2,3,1)$&
\begin{tikzpicture}[scale=0.3]
\draw[-,green] (-1,2.6)..controls (-0.7,2.6)..(0,2);
\draw[-,green] (0,2)..controls (1,2.8)..(2,2);
\draw[-,green] (2,2)..controls (3.6,3)..(4,3);
\draw[-,green] (4,3)..controls (3.2,2)..(4,1);
\draw[-,red] (-1,1.4)..controls (-0.7,1.4)..(0,2);
\draw[-,red] (0,2)..controls (1,1.2)..(2,2);
\draw[-,red] (2,2)..controls (3.6,1)..(4,1);
\draw[-,red] (7,2.6)..controls (6.7,2.6)..(6,2);
\draw[-,red] (6,2)..controls (4.4,3)..(4,3);
\draw[-,red] (4,3)..controls (4.8,2)..(4,1);
\draw[-,green] (7,1.4)..controls (6.7,1.4)..(6,2);
\draw[-,green] (6,2)..controls (4.4,1)..(4,1);
\filldraw [black]  (0,2)    circle (2pt)
[black]  (2,2)    circle (2pt)
[black]  (4,3)    circle (2pt)
[black]  (6,2)    circle (2pt)
[black]  (4,1)    circle (2pt);
\end{tikzpicture}
& LNC \\  \hline
4. & $\begin{array}{c}
  (1,2,4,5,3),(3,5,4,2,1),(1,2,5,3,4),(4,3,5,2,1),\\
  (3,1,2,4,5),(5,4,2,1,3),(2,3,1,4,5),(5,4,1,3,2)
\end{array}$&
\begin{tikzpicture}[scale=0.3]
\draw[-,green] (-1,2.6)..controls (-0.7,2.6)..(0,2);
\draw[-,green] (0,2)..controls (1,2.8)..(2,2);
\draw[-,green] (2,2)..controls (3.6,3)..(4,3);
\draw[-,green] (4,3) -- (4,1);
\draw[-,red] (4,1)..controls (5,2)..(6,2);
\draw[-,red] (-1,1.4)..controls (-0.7,1.4)..(0,2);
\draw[-,red] (0,2)..controls (1,1.2)..(2,2);
\draw[-,red] (2,2)..controls (3.6,1)..(4,1);
\draw[-,green] (7,2.6)..controls (6.7,2)..(6,2);
\draw[-,red] (6,2)..controls (4.4,3)..(4,3);
\draw[-,red] (4,3)..controls (4.8,3.3)..(5.6,3);
\draw[-,green] (6,2)..controls (5.4,1)..(4,1);
\filldraw [black]  (0,2)    circle (2pt)
[black]  (2,2)    circle (2pt)
[black]  (4,3)    circle (2pt)
[black]  (6,2)    circle (2pt)
[black]  (4,1)    circle (2pt);
\draw (2,3.5) node{};
\end{tikzpicture}
& LNC \\  \hline
5. &$(1,2,5,4,3),(3,4,5,2,1),(3,2,1,4,5),(5,4,1,2,3)$&
\begin{tikzpicture}[scale=0.3]
\draw[-,red] (2,2)..controls (2.8,3)..(2,4);
\draw[-,red] (2,4)..controls (2.8,5)..(2,6);
\draw[-,red] (2,2)..controls (1.2,2)..(0,4);
\draw[-,green] (2,6)..controls (1.2,6)..(0,4);
\draw[-,green] (0,4)..controls (-1,4.8)..(-2,4);
\draw[-,red] (0,4)..controls (-1,3.2)..(-2,4);
\draw[-,green] (-3,4.6)..controls (-2.6,4.6)..(-2,4);
\draw[-,red] (-3,3.4)..controls (-2.6,3.4)..(-2,4);
\draw[-,green] (3,2.5)..controls (2.7,2)..(2,2);
\draw[-,green] (2,2)..controls (1.2,3)..(2,4);
\draw[-,green] (2,4)..controls (1.2,5)..(2,6);
\draw[-,red] (2,6)..controls (2.7,6)..(3,5.5);
\filldraw [black]  (2,2)    circle (2pt)
[black]  (2,4)    circle (2pt)
[black]  (2,6)    circle (2pt)
[black]  (0,4)    circle (2pt)
[black]  (-2,4)    circle (2pt);
\end{tikzpicture}
& LNC \\  \hline
6. & $(1,3,2,5,4),(4,5,2,3,1),(2,1,4,3,5),(5,3,4,1,2)$&
\begin{tikzpicture}[scale=0.3]
\draw[-,green] (2,2)..controls (2.8,3)..(2,4);
\draw[-,red] (0,3)..controls (1.4,2)..(2,2);
\draw[-,red] (2,2)..controls (1.2,3)..(2,4);
\draw[-,green] (2,4)..controls (1.4,4)..(0,3);
\draw[-,red] (4,2)..controls (3.2,3)..(4,4);
\draw[-,green] (0,3)..controls (-0.7,3.6)..(-1,3.6);
\draw[-,red] (0,3)..controls (-0.7,2.4)..(-1,2.4);
\draw[-,red] (5,2.3)..controls (4.5,2)..(4,2);
\draw[-,green] (4,2)..controls (4.8,3)..(4,4);
\draw[-,green] (4,4)..controls (4.5,4)..(5,3.7);
\draw[-,red] (2,4)--(4,4);
\draw[-,green] (2,2)--(4,2);
\filldraw [black]  (2,2)    circle (2pt)
[black]  (4,2)    circle (2pt)
[black]  (2,4)    circle (2pt)
[black]  (0,3)    circle (2pt)
[black]  (4,4)    circle (2pt);
\end{tikzpicture}
&--
\\  \hline
7. & $(1,3,4,2,5),(5,2,4,3,1),(1,4,2,3,5),(5,3,2,4,1)$&
\begin{tikzpicture}[scale=0.15]
\draw[-,red] (0,2)..controls (4,0)..(5,0);
\draw[-,green] (0,2)..controls (2,4)..(3,4);
\draw[-,green] (10,2)..controls (6,0)..(5,0);
\draw[-,red] (10,2)..controls (8,4)..(7,4);
\draw[-,red] (5,0)..controls (3,3)..(3,4);
\draw[-,green] (5,0)..controls (7,3)..(7,4);
\draw[-,red] (3,4)..controls(5,5.4)..(7,4);
\draw[-,green] (3,4)..controls(5,2.6)..(7,4);
\draw[-,green] (0,2)--(1.5,3.5);
\draw[-,red] (7,4)..controls (7.8,4)..(8.5,3.5);
\draw (14,6) node{};
\draw[-,green] (0,2)..controls (-1.4,3.2)..(-2,3.2);
\draw[-,red] (0,2)..controls (-1.4,0.8)..(-2,0.8);
\draw[-,green] (10,2)..controls (11.4,0.8)..(12,0.8);
\draw[-,red] (10,2)..controls (11.4,3.2)..(12,3.2);
\filldraw [black] (3,4) circle (4pt)
[black] (7,4) circle (4pt)
[black] (0,2) circle (4pt)
[black] (10,2) circle (4pt)
[black] (5,0) circle (4pt);
\end{tikzpicture}
& LNC \\  \hline
8. & $\begin{array}{c}
  (1,3,4,5,2),(2,5,4,3,1),(4,1,2,3,5),(5,3,2,1,4),\\
  (1,5,2,3,4),(4,3,2,5,1),(2,3,4,1,5),(5,1,4,3,2)
\end{array}$
&
\begin{tikzpicture}[scale=0.3]
\draw[-,green] (1,2.6)..controls (1.3,2.6)..(2,2);
\draw[-,green] (2,2)..controls (3.6,3)..(4,3);
\draw[-,green] (4,3) -- (4,1);
\draw[-,red] (8,3)..controls (7,3)..(6,2);
\draw[-,red] (4,1)..controls (5,2)..(6,2);
\draw[-,red] (1,1.4)..controls (1.3,1.4)..(2,2);
\draw[-,red] (2,2)..controls (3.6,1)..(4,1);
\draw[-,green] (9,3.6)..controls (8.7,3)..(8,3);
\draw[-,red] (8,3)..controls (7,3.3)..(4,3);
\draw[-,red] (4,3)..controls (4,3.7)..(4.6,4);
\draw[-,green] (6,2)..controls (5.4,1)..(4,1);
\draw[-,green] (8,3)..controls (7.4,2)..(6,2);
\filldraw [black]  (2,2)    circle (2pt)
[black]  (4,3)    circle (2pt)
[black]  (6,2)    circle (2pt)
[black]  (8,3)    circle (2pt)
[black]  (4,1)    circle (2pt);
\end{tikzpicture}
& LNC \\  \hline
9. & $\begin{array}{c}
(1,3,5,2,4),(4,2,5,3,1),(2,4,1,3,5),(5,3,1,4,2),\\
(1,4,2,5,3),(3,5,2,4,1),(3,1,4,2,5),(5,2,4,1,3)
\end{array}$
&
\begin{tikzpicture}[scale=0.3]
\draw[-,green] (2,2) --(4,4);
\draw[-,red] (2,2)..controls(1.5,2).. (-1,4);
\draw[-,red] (2,6) --(4,4);
\draw[-,green] (2,6) ..controls(1.5,6).. (-1,4);
\draw[-,red] (6,4) --(4,4);
\draw[-,green] (3,4) --(4,4);
\draw[-,red] (6,4) --(7,4);
\draw[-,green] (2,2)..controls (4.5,2)..(6,4);
\draw[-,green] (2,6)..controls (4.5,6)..(6,4);
\draw[-,red] (2,2)..controls (1,4)..(2,6);
\draw[-,green] (-2,4.6)..controls (-1.7,4.6)..(-1,4);
\draw[-,red] (-2,3.4)..controls (-1.7,3.4)..(-1,4);
\filldraw [black]  (2,2)    circle (2pt)
[black]  (2,6)    circle (2pt)
[black]  (4,4)    circle (2pt)
[black]  (-1,4)    circle (2pt)
[black]  (6,4)    circle (2pt);
\end{tikzpicture}
& LNC \\  \hline
10.&$\begin{array}{c}
(1,3,5,4,2),(2,4,5,3,1),(4,2,1,3,5),(5,3,1,2,4),\\
(1,5,2,4,3),(3,4,2,5,1),(3,2,4,1,5),(5,1,4,2,3)
\end{array}$ &
\begin{tikzpicture}[scale=0.3]
\draw[-,green] (1,2.6)..controls (1.3,2.6)..(2,2);
\draw[-,green] (2,2)..controls (3.6,3)..(4,3);
\draw[-,green] (4,3) -- (4,1);
\draw[-,green] (4,1)..controls (5,2.3)..(6,3);
\draw[-,red] (1,1.4)..controls (1.3,1.4)..(2,2);
\draw[-,red] (2,2)..controls (3.6,1)..(4,1);
\draw[-,red] (6,3)..controls (6.5,2)..(6,1);
\draw[-,green] (6,3)..controls (5.5,2)..(6,1);
\draw[-,green] (7,1.6)..controls (6.7,1)..(6,1);
\draw[-,red] (6,3)--(4,3);
\draw[-,red] (4,3)..controls (3.3,3.3)..(3,3.3);
\draw[-,red] (6,1)--(4,1);
\filldraw [black]  (6,1)    circle (2pt)
[black]  (2,2)    circle (2pt)
[black]  (4,3)    circle (2pt)
[black]  (6,3)    circle (2pt)
[black]  (4,1)    circle (2pt);
\draw (2,3.5) node{};
\end{tikzpicture}
& LNC \\  \hline
11.&$(1,4,3,2,5),(5,2,3,4,1)$
&
\begin{tikzpicture}[scale=0.3]
\draw[-,green] (2,2)..controls (2.8,3)..(2,4);
\draw[-,green] (2,4)..controls (2.8,5)..(2,6);
\draw[-,green] (2,2)..controls (2.8,2)..(4,4);
\draw[-,red] (2,6)..controls (2.8,6)..(4,4);
\draw[-,red] (5,4.6)..controls (4.6,4.6)..(4,4);
\draw[-,green] (5,3.4)..controls (4.6,3.4)..(4,4);
\draw[-,green] (-1,4.6)..controls (-0.6,4.6)..(0,4);
\draw[-,red] (-1,3.4)..controls (-0.6,3.4)..(0,4);
\draw[-,red] (0,4)..controls (1.2,2)..(2,2);
\draw[-,red] (2,2)..controls (1.2,3)..(2,4);
\draw[-,red] (2,4)..controls (1.2,5)..(2,6);
\draw[-,green] (2,6)..controls (1.2,6)..(0,4);
\filldraw [black]  (2,2)    circle (2pt)
[black]  (2,4)    circle (2pt)
[black]  (2,6)    circle (2pt)
[black]  (4,4)    circle (2pt)
[black]  (0,4)    circle (2pt);
\end{tikzpicture}
& LNC \\  \hline
12. &$\begin{array}{c}
 (1,4,3,5,2),(2,5,3,4,1),(4,1,3,2,5),(5,2,3,1,4),\\
 (1,5,3,2,4),(4,2,3,5,1),(2,4,3,1,5),(5,1,3,4,2)
 \end{array}$
 &
\begin{tikzpicture}[scale=0.3]
\draw[-,green] (2,4)..controls (2.8,5)..(2,6);
\draw[-,green] (2,2)..controls (2.8,2)..(4,4);
\draw[-,red] (2,6)..controls (2.8,6)..(4,4);
\draw[-,green] (2,2)..controls (2.5,1.7)..(3,2);
\draw[-,green] (-1,4.6)..controls (-0.6,4.6)..(0,4);
\draw[-,red] (-1,3.4)..controls (-0.6,3.4)..(0,4);
\draw[-,red] (0,4)..controls (1.2,2)..(2,2);
\draw[-,red] (2,2)--(2,4);
\draw[-,green] (4,4)--(2,4);
\draw[-,red] (4,4)--(5,4);
\draw[-,red] (2,4)..controls (1.2,5)..(2,6);
\draw[-,green] (2,6)..controls (1.2,6)..(0,4);
\filldraw [black]  (2,2)    circle (2pt)
[black]  (2,4)    circle (2pt)
[black]  (2,6)    circle (2pt)
[black]  (4,4)    circle (2pt)
[black]  (0,4)    circle (2pt);
\end{tikzpicture}
& LNC \\  \hline
13. & $(1,4,5,2,3),(3,2,5,4,1),(3,4,1,2,5),(5,2,1,4,3)$
&
\begin{tikzpicture}[scale=0.3]
\draw[-,green] (2,4)..controls (2.8,5)..(2,6);
\draw[-,red] (2,2)..controls (3,2.5)..(4,4);
\draw[-,green] (2,2)..controls (3,3.5)..(4,4);
\draw[-,red] (2,6)..controls (2.8,6)..(4,4);
\draw[-,red] (2,4)..controls (1.5,3.7)..(1,4);
\draw[-,green] (-1,4.6)..controls (-0.6,4.6)..(0,4);
\draw[-,red] (-1,3.4)..controls (-0.6,3.4)..(0,4);
\draw[-,red] (0,4)..controls (1.2,2)..(2,2);
\draw[-,green] (2,2)--(2,4);
\draw[-,green] (4,4)--(5,4);
\draw[-,red] (2,4)..controls (1.2,5)..(2,6);
\draw[-,green] (2,6)..controls (1.2,6)..(0,4);
\filldraw [black]  (2,2)    circle (2pt)
[black]  (2,4)    circle (2pt)
[black]  (2,6)    circle (2pt)
[black]  (4,4)    circle (2pt)
[black]  (0,4)    circle (2pt);
\end{tikzpicture}
& -- \\  \hline
14. & $\begin{array}{c}
(1,4,5,3,2),(2,3,5,4,1),(4,3,1,2,5),(5,2,1,3,4),\\
(1,5,4,2,3),(3,2,4,5,1),(3,4,2,1,5),(5,1,2,4,3)
\end{array} $
&\begin{tikzpicture}[scale=0.3]
\draw[-,green] (2,2)..controls (3,2.5)..(4,2);
\draw[-,red] (0,3)..controls (1.4,2)..(2,2);
\draw[-,red] (2,2)..controls (3,1.5)..(4,2);
\draw[-,green] (2,4)..controls (1.4,4)..(0,3);
\draw[-,red] (2,4)..controls (3,3.5)..(4,4);
\draw[-,green] (0,3)..controls (-0.7,3.6)..(-1,3.6);
\draw[-,red] (0,3)..controls (-0.7,2.4)..(-1,2.4);
\draw[-,green] (5,2.3)..controls (4.5,2)..(4,2);
\draw[-,green] (2,4)..controls (3,4.5)..(4,4);
\draw[-,red] (2,4)..controls (1.5,4.6)..(1,4.4);
\draw[-,red] (4,2)--(4,4);
\draw[-,green] (2,2)..controls (2,2.6) and (4,3.4)..(4,4);
\filldraw [black]  (2,2)    circle (2pt)
[black]  (4,2)    circle (2pt)
[black]  (2,4)    circle (2pt)
[black]  (0,3)    circle (2pt)
[black]  (4,4)    circle (2pt);
\end{tikzpicture}
& -- \\  \hline
15. & $(1,5,3,4,2),(2,4,3,5,1),(4,2,3,1,5),(5,1,3,2,4)$
&
\begin{tikzpicture}[scale=0.3]
\draw[-,red] (2,2)..controls (1,3.5)..(1,4);
\draw[-,green] (1,4)..controls (1,4.5)..(2,6);
\draw[-,green] (2,2)..controls (2.5,1.8)..(3,2);
\draw[-,red] (2,6)..controls (2.5,6.2)..(3,6);
\draw[-,green] (-1,4.6)..controls (-0.6,4.6)..(0,4);
\draw[-,red] (-1,3.4)..controls (-0.6,3.4)..(0,4);
\draw[-,red] (0,4)..controls (1.2,2)..(2,2);
\draw[-,green] (1,4)..controls (2,3.4)..(3,4);
\draw[-,red] (1,4)..controls (2,4.6)..(3,4);
\draw[-,green] (2,2)..controls (3,3.5)..(3,4);
\draw[-,red] (3,4)..controls (3,4.5)..(2,6);
\draw[-,green] (2,6)..controls (1.2,6)..(0,4);
\filldraw [black]  (2,2)    circle (2pt)
[black]  (1,4)    circle (2pt)
[black]  (2,6)    circle (2pt)
[black]  (3,4)    circle (2pt)
[black]  (0,4)    circle (2pt);
\end{tikzpicture}
& LNC \\  \hline
\end{tabular}

\begin{tabular}{|c|c|c|c|}
  \hline
Case&  Permutations & Representation & Con.\\ \hline
16. &
$(1,5,4,3,2),(2,3,4,5,1),(4,3,2,1,5),(5,1,2,3,4)$
&
\begin{tikzpicture}[scale=0.3]
\draw[-,red] (1,2)..controls (2,2.8)..(3,2);
\draw[-,red] (3,2)..controls (4,2.8)..(5,2);
\draw[-,red] (5,2)..controls (6,2.8)..(7,2);
\draw[-,red] (1,2)..controls (1,2.8)..(4,4);
\draw[-,green] (7,2)..controls (7,2.8)..(4,4);
\draw[-,green] (5,5)..controls (5,4.6)..(4,4);
\draw[-,red] (3,5)..controls (3,4.6)..(4,4);
\draw[-,green] (0,1.4)..controls (0.3,1.4)..(1,2);
\draw[-,green] (1,2)..controls (2,1.2)..(3,2);
\draw[-,green] (3,2)..controls (4,1.2)..(5,2);
\draw[-,green] (5,2)..controls (6,1.2)..(7,2);
\draw[-,red] (7,2)..controls (7.7,1.4)..(8,1.4);
\filldraw [black]  (1,2)    circle (2pt)
[black]  (3,2)    circle (2pt)
[black]  (5,2)    circle (2pt)
[black]  (7,2)    circle (2pt)
[black]  (4,4)    circle (2pt);
\end{tikzpicture}
& LNC\\  \hline
17. &
$(2,1,3,5,4),(4,5,3,1,2)$
&
\begin{tikzpicture}[scale=0.3]
\draw[-,green] (2,3)..controls (0.6,2)..(0,2);
\draw[-,red] (0,4)..controls (0.6,4)..(2,3);
\draw[-,green] (0,2)..controls (0.8,3)..(0,4);
\draw[-,red] (-1,2.3)..controls (-0.5,2)..(0,2);
\draw[-,red] (0,2)..controls (-0.8,3)..(0,4);
\draw[-,green] (0,4)..controls (-0.5,4)..(-1,3.7);
\draw[-,green] (2,3)..controls (3.4,2)..(4,2);
\draw[-,red] (4,4)..controls (3.4,4)..(2,3);
\draw[-,green] (4,2)..controls (3.2,3)..(4,4);
\draw[-,red] (5,2.3)..controls (4.5,2)..(4,2);
\draw[-,red] (4,2)..controls (4.8,3)..(4,4);
\draw[-,green] (4,4)..controls (4.5,4)..(5,3.7);
\filldraw [black]  (0,2)    circle (2pt)
[black]  (4,2)    circle (2pt)
[black]  (0,4)    circle (2pt)
[black]  (2,3)    circle (2pt)
[black]  (4,4)    circle (2pt);
\end{tikzpicture}
& LNC\\  \hline
18. & $\begin{array}{c}
(2,1,4,5,3),(3,5,4,1,2),(3,1,2,5,4),(4,5,2,1,3),\\
(2,1,5,3,4),(4,3,5,1,2),(2,3,1,5,4),(4,5,1,3,2)
\end{array} $
&
\begin{tikzpicture}[scale=0.3]
\draw[-,red] (2,3)--(4,3);
\draw[-,red] (4,3) -- (4,1);
\draw[-,green] (2,1)..controls (2.6,2)..(2,3);
\draw[-,red] (2,1)..controls (1.4,2)..(2,3);
\draw[-,red] (2,1)..controls (1.6,1)..(1,1.3);
\draw[-,green] (2,3)..controls (1.6,3)..(1,2.7);
\draw[-,green] (4,1)..controls (5,2)..(6,2);
\draw[-,green] (2,1)--(4,1);
\draw[-,red] (7,2.6)..controls (6.7,2)..(6,2);
\draw[-,green] (6,2)..controls (4.4,3)..(4,3);
\draw[-,green] (4,3)..controls (4.8,3.3)..(5.6,3);
\draw[-,red] (6,2)..controls (5.4,1)..(4,1);
\filldraw [black]  (2,1)    circle (2pt)
[black]  (2,3)    circle (2pt)
[black]  (4,3)    circle (2pt)
[black]  (6,2)    circle (2pt)
[black]  (4,1)    circle (2pt);
\draw (2,3.5) node{};
\end{tikzpicture}
& LNC\\  \hline
19. &
$(2,1,5,4,3),(3,4,5,1,2),(3,2,1,5,4),(4,5,1,2,3)$
&
\begin{tikzpicture}[scale=0.3]
\draw[-,green] (2,2)..controls (2.8,3)..(2,4);
\draw[-,green] (2,4)..controls (2.8,5)..(2,6);
\draw[-,green] (2,2)..controls (1.2,2)..(0,3);
\draw[-,red] (2,6)..controls (1.2,6)..(0,5);
\draw[-,green] (0,3)..controls (0.8,4)..(0,5);
\draw[-,red] (0,3)..controls (-0.8,4)..(0,5);
\draw[-,red] (0,3)..controls (-0.5,2.8)..(-1,3);
\draw[-,green] (0,5)..controls (-0.5,5.2)..(-1,5);
\draw[-,red] (3,2.5)..controls (2.7,2)..(2,2);
\draw[-,red] (2,2)..controls (1.2,3)..(2,4);
\draw[-,red] (2,4)..controls (1.2,5)..(2,6);
\draw[-,green] (2,6)..controls (2.7,6)..(3,5.5);
\filldraw [black]  (2,2)    circle (2pt)
[black]  (2,4)    circle (2pt)
[black]  (2,6)    circle (2pt)
[black]  (0,3)    circle (2pt)
[black]  (0,5)    circle (2pt);
\end{tikzpicture}
& LNC\\  \hline
20. & $\begin{array}{c}
(2,3,5,1,4),(4,1,5,3,2),(2,5,1,3,4),(4,3,1,5,2),\\
(4,1,2,5,3),(3,5,2,1,4),(3,1,4,5,2),(2,5,4,1,3)
\end{array} $
&
\begin{tikzpicture}[scale=0.3]
\draw[-,red] (4,3)-- (3,4);
\draw[-,red] (4,3) -- (4,1);
\draw[-,red] (2,1)--(2,3);
\draw[-,red] (2,1)..controls (1.6,1)..(1,1.3);
\draw[-,green] (2,3)..controls (1.6,3)..(1,2.7);
\draw[-,red] (4,1)..controls (4.5,0.7)..(5,1);
\draw[-,green] (2,1)--(4,1);
\draw[-,green] (2,3)..controls (2.3,3.7)..(3,4);
\draw[-,red] (2,3)..controls (2.7,3.3)..(3,4);
\draw[-,green] (4,1)..controls (7,2)and(6,4)..(3,4);
\draw[-,green] (4,3)..controls (4.5,3.3)..(5,3);
\draw[-,green] (2,1)--(4,3);
\filldraw [black]  (2,1)    circle (2pt)
[black]  (2,3)    circle (2pt)
[black]  (4,3)    circle (2pt)
[black]  (3,4)    circle (2pt)
[black]  (4,1)    circle (2pt);
\draw (2,3.5) node{};
\end{tikzpicture}
& LNC\\  \hline
21. &
$(2,4,1,5,3),(3,5,1,4,2),(3,1,5,2,4),(4,2,5,1,3)$
&
\begin{tikzpicture}[scale=0.3]
\draw[-,red] (2,3)--(4,3);
\draw[-,red] (4,3) -- (4,1);
\draw[-,green] (2,1)--(2,3);
\draw[-,red] (2,1)..controls (1.6,1)..(1,1.3);
\draw[-,green] (2,3)..controls (1.6,3)..(1,2.7);
\draw[-,green] (4,3)..controls (4.4,3)..(5,2.7);
\draw[-,red] (4,1)..controls (4.4,1)..(5,1.3);
\draw[-,green] (2,1)--(4,1);
\draw[-,red] (2,1)--(3,2);
\draw[-,green] (3,2)--(4,1);
\draw[-,red] (3,2)--(2,3);
\draw[-,green] (3,2)--(4,3);
\filldraw [black]  (2,1)    circle (2pt)
[black]  (2,3)    circle (2pt)
[black]  (4,3)    circle (2pt)
[black]  (3,2)    circle (2pt)
[black]  (4,1)    circle (2pt);
\draw (2,3.5) node{};
\end{tikzpicture}
& TD\\  \hline
22. & $\begin{array}{c}
(2,4,5,1,3),(3,1,5,4,2),(3,5,1,2,4),(4,2,1,5,3),\\
(4,1,5,2,3),(3,2,5,1,4),(3,4,1,5,2),(2,5,1,4,3)
       \end{array}$
&
\begin{tikzpicture}[scale=0.3]
\draw[-,green] (4,3)-- (3,4);
\draw[-,green] (4,3) -- (4,1);
\draw[-,red] (2,1)--(2,3);
\draw[-,green] (2,1)..controls (1.6,1)..(1,1.3);
\draw[-,green] (2,3)..controls (1.6,3)..(1,2.7);
\draw[-,red] (4,1)..controls (4.5,0.7)..(5,1);
\draw[-,green] (2,1)--(4,1);
\draw[-,red] (2,3)..controls (2.3,3.7)..(3,4);
\draw[-,green] (2,3)..controls (2.7,3.3)..(3,4);
\draw[-,red] (4,1)..controls (7,2)and(6,4)..(3,4);
\draw[-,red] (4,3)..controls (4.5,3.3)..(5,3);
\draw[-,red] (2,1)--(4,3);
\filldraw [black]  (2,1)    circle (2pt)
[black]  (2,3)    circle (2pt)
[black]  (4,3)    circle (2pt)
[black]  (3,4)    circle (2pt)
[black]  (4,1)    circle (2pt);
\draw (2,3.5) node{};
\end{tikzpicture}
& LNC\\  \hline
23. &
$(2,5,3,1,4),(4,1,3,5,2)$
&
\begin{tikzpicture}[scale=0.3]
\draw[-,green] (2,3)--(4,3);
\draw[-,red] (4,3) -- (4,1);
\draw[-,red] (2,1)--(2,3);
\draw[-,red] (2,1)..controls (1.6,1)..(1,1.3);
\draw[-,green] (2,3)..controls (1.6,3)..(1,2.7);
\draw[-,red] (4,3)..controls (4.4,3)..(5,2.7);
\draw[-,green] (4,1)..controls (4.4,1)..(5,1.3);
\draw[-,green] (2,1)--(4,1);
\draw[-,green] (2,1)--(3,2);
\draw[-,red] (3,2)--(4,1);
\draw[-,red] (3,2)--(2,3);
\draw[-,green] (3,2)--(4,3);
\filldraw [black]  (2,1)    circle (2pt)
[black]  (2,3)    circle (2pt)
[black]  (4,3)    circle (2pt)
[black]  (3,2)    circle (2pt)
[black]  (4,1)    circle (2pt);
\draw (2,3.5) node{};
\end{tikzpicture}
& TD\\  \hline
\end{tabular}\vspace{0.5cm}

\noindent Now we will verify that for all cases in the table $BT(P,Q)$ is made of LNC blocks, when the last entry is LNC,
and that it is TD, when the last entry is TD.
We know by Proposition~\ref{elementary block is NC} that elementary blocks are LNC, so Case 1. is clear. Cases 2., 5., 15., 16. and 17. follow from
Proposition~\ref{lista}(2) with different embedded LNC IBU's, Cases 3. and 11. from Proposition~\ref{lista}(1),
Cases 4. and 8. from Proposition~\ref{lista}(4), and Case 7. from Proposition~\ref{lista}(3).

\noindent {\bf Case 9.:} In this case we have one elementary EBB and one EBB with 10 components and no embedded IBB.

\noindent 


\noindent
  One verifies directly that in this ESU all the pairs of components of different colors except seven, are adjacent or both extremal, so
  Propositions~\ref{prop adyacentes} and~\ref{PROP conexion entre extremos} prove that they are LNC.
  The two pairs $1b$ and $1e$ are proven to be LNC by the paths in the second and third diagrams above,
  where $X_1$ and $Y_1$ are in black and $X_2$ and $Y_2$ are in blue.
  The remaining pairs $2e,3c,4a,5a,5d$ are proven to be LNC by the paths in the following diagrams,
  where $X_1$ and $Y_1$ are in black and $X_2$ and $Y_2$ are in blue.

\noindent %

\begin{remark}
  From now on, when we identify in a graph $X_1,X_2,Y_1,Y_2$ and $x,y$ satisfying Definition~\ref{def LNC},
  then we will draw the paths $X_1$ and $Y_1$ in black, the paths $X_2$ and $Y_2$ in blue and $x$ and $y$ in red.
  We don't need to write the names in the graph and we won't do it.
\end{remark}

\noindent {\bf Case 10.:} In this case we have one elementary EBB, and one bigger EBB, that
has one embedded IBB and that has an ESU with 8 components:

%

One verifies directly that in this ESU all the pairs of components of different colors except four, are adjacent or both extremal, so
  Propositions~\ref{prop adyacentes} and~\ref{PROP conexion entre extremos} prove that they are LNC.
  The remaining pairs $1d,2a,3d,4c$ are proven to be LNC by the paths in the following diagrams.
  The paths $X_1$ and $Y_1$ are in black and the paths $X_2$ and $Y_2$ are in blue.

%

\noindent {\bf Case 12.:} In this case we have one elementary EBB, and one bigger EBB, that
has one embedded IBB and that has an ESU with 8 components:
$$
%
$$
One verifies directly that in this ESU all the pairs of components of different colors except four, are adjacent or both extremal, so
  Propositions~\ref{prop adyacentes} and~\ref{PROP conexion entre extremos} prove that they are LNC.
  The remaining pairs $1b,2a,3d,4a$ are proven to be LNC by the paths in the following diagrams.
  The paths $X_1$ and $Y_1$ are in black and the paths $X_2$ and $Y_2$ are in blue.

%

\noindent {\bf Case 18.:} In this case we have one BB with two embedded elementary IBB's and an ESU with 8 components:
$$
%
$$
One verifies directly that in this ESU all the pairs of components of different colors except four, are adjacent or both extremal, so
  Propositions~\ref{prop adyacentes} and~\ref{PROP conexion entre extremos} prove that they are LNC.
  The remaining pairs $1c,2b,2d,3a$ are proven to be LNC by the paths in the following diagrams.
  The paths $X_1$ and $Y_1$ are in black and the paths $X_2$ and $Y_2$ are in blue.

%

\noindent {\bf Case 19.:} In this case we have one BB with three embedded elementary IBB's and an ESU with 6 components:

%

\noindent One verifies directly that in this ESU all the pairs of components of different colors except one, are adjacent or both extremal, so
   Propositions~\ref{prop adyacentes} and~\ref{PROP conexion entre extremos} prove that they are LNC.
  The remaining pair is proven to be LNC by the paths in the second diagram, where
  the paths $X_1$ and $Y_1$ are in black and the paths $X_2$ and $Y_2$ are in blue.

\noindent {\bf Case 20.:} In this case we have one BB with one embedded elementary IBB and an ESU with 10 components:

%

\noindent
  One verifies directly that in this ESU all the pairs of components of different colors except seven, are adjacent or both extremal, so
  Propositions~\ref{prop adyacentes}   and~\ref{PROP conexion entre extremos} prove that they are LNC.
  The two pairs $1b$ and $2b$ are proven to be LNC by the paths in the second and third diagrams above,
  where $X_1$ and $Y_1$ are in black and $X_2$ and $Y_2$ are in blue.
  The remaining pairs $2e,3a,3c,4a,5d$ are proven to be LNC by the paths in the following diagrams,
  where $X_1$ and $Y_1$ are in black and $X_2$ and $Y_2$ are in blue.

%

\noindent {\bf Case 21.} and {\bf Case 23.} In both cases we have the same bi-traceable graph, which is TD. In order to see this,
we first use Lemmas~\ref{lema 4.1} and~\ref{conexion entre extremos} and the symmetry of the graph,
in order to verify that adjacent pairs and pairs of extremal components are NDC. Using again the symmetry of the graph,
 in order to check
the remaining pairs, it suffices to prove
that the five pairs of components $ab$, $ac$, $ad$, $be$ and $ce$ are NDC, which follows from the paths in the diagrams below,
where $X_1$ and $Y_1$ are in black and $X_2$ and $Y_2$ are in blue. In fact, the paths $\widehat P$ and $\widehat Q$ of
Definition~\ref{def NDC} are given in each case by
$$
\widehat P=X_1+R+Y_1\quad \text{and}\quad \widehat Q=X_2+R+Y_2.
$$

\begin{tikzpicture}[scale=1]
\draw[-] (2,3)--(4,3);
\draw[-] (4,3) -- (4,1);
\draw[-] (2,1)--(2,3);
\draw[-] (2,1)..controls (1.6,1)..(1,1.3);
\draw[-] (2,3)..controls (1.6,3)..(1,2.7);
\draw[-] (4,3)..controls (4.4,3)..(5,2.7);
\draw[-] (4,1)..controls (4.4,1)..(5,1.3);
\draw[-] (2,1)--(4,1);
\draw[-] (2,1)--(3,2);
\draw[-] (3,2)--(4,1);
\draw[-] (3,2)--(2,3);
\draw[-] (3,2)--(4,3);
\filldraw [black]  (2,1)    circle (2pt)
[black]  (2,3)    circle (2pt)
[black]  (4,3)    circle (2pt)
[black]  (3,2)    circle (2pt)
[black]  (4,1)    circle (2pt);
\draw (0.9,2.85) node{a};
\draw (2.35,1.65) node{b};
\draw (3,1.15) node{c};
\draw (3.65,1.65) node{d};
\draw (3,3.2) node{e};
\draw (3,0.4)node{Bi-traceable graph};
\draw (5.6,3.5) node{};
\end{tikzpicture}
\begin{tikzpicture}[scale=1]
\draw[-,cyan,line width=2pt] (2,3)--(4,3);
\draw[-,line width=2pt] (4,3) -- (4,1);
\draw[-,cyan,line width=2pt] (2,1)--(2,3);
\draw[-,line width=2pt] (2,1)..controls (1.6,1)..(1,1.3);
\draw[-,line width=2pt] (2,3)..controls (1.7,3)..(1.5,2.93);
\draw[-,cyan,line width=2pt] (1.5,2.93)--(1,2.7);
\draw[-,line width=2pt] (4,3)..controls (4.4,3)..(5,2.7);
\draw[-,cyan,line width=2pt] (4,1)..controls (4.4,1)..(5,1.3);
\draw[-,cyan,line width=2pt] (2,1)--(4,1);
\draw[-,line width=2pt] (2,1)--(2.5,1.5);
\draw[-,cyan,line width=2pt] (2.5,1.5)--(3,2);
\draw[-,line width=2pt] (3,2)--(4,1);
\draw[-,line width=2pt] (3,2)--(2,3);
\draw[-,cyan,line width=2pt] (3,2)--(4,3);
\filldraw [red]  (1.5,2.93)    circle (1.5pt)
[red]  (2.5,1.5)    circle (1.5pt);
\filldraw [black]  (2,1)    circle (2pt)
[black]  (2,3)    circle (2pt)
[black]  (4,3)    circle (2pt)
[black]  (3,2)    circle (2pt)
[black]  (4,1)    circle (2pt);
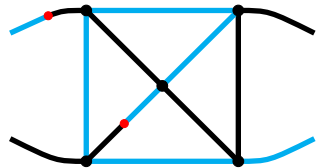
\draw (3,0.4)node{The case $ab$};
\end{tikzpicture}
\begin{tikzpicture}[scale=1]
\draw[-,cyan,line width=2pt] (2,3)--(4,3);
\draw[-,line width=2pt] (4,3) -- (4,1);
\draw[-,cyan,line width=2pt] (2,1)--(2,3);
\draw[-,line width=2pt] (2,1)..controls (1.6,1)..(1,1.3);
\draw[-,line width=2pt] (2,3)..controls (1.7,3)..(1.5,2.93);
\draw[-,cyan,line width=2pt] (1.5,2.93)--(1,2.7);
\draw[-,line width=2pt] (4,3)..controls (4.4,3)..(5,2.7);
\draw[-,cyan,line width=2pt] (4,1)..controls (4.4,1)..(5,1.3);
\draw[-,cyan,line width=2pt] (2,1)--(3,1);
\draw[-,line width=2pt] (3,1)--(4,1);
\draw[-,line width=2pt] (2,1)--(3,2);
\draw[-,cyan,line width=2pt] (3,2)--(4,1);
\draw[-,line width=2pt] (3,2)--(2,3);
\draw[-,cyan,line width=2pt] (3,2)--(4,3);
\filldraw [red]  (1.5,2.93)    circle (1.5pt)
[red]  (3,1)    circle (1.5pt);
\filldraw [black]  (2,1)    circle (2pt)
[black]  (2,3)    circle (2pt)
[black]  (4,3)    circle (2pt)
[black]  (3,2)    circle (2pt)
[black]  (4,1)    circle (2pt);
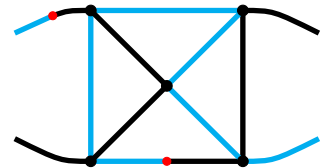
\draw (3,0.4)node{The case $ac$};
\draw(0,2)node{};
\end{tikzpicture}

\begin{tikzpicture}[scale=1]
\draw[-,line width=2pt] (2,3)--(4,3);
\draw[-,line width=2pt] (4,3) -- (4,1);
\draw[-,cyan,line width=2pt] (2,1)--(2,3);
\draw[-,line width=2pt] (2,1)..controls (1.6,1)..(1,1.3);
\draw[-,cyan,line width=2pt] (2,3)..controls (1.7,3)..(1.5,2.93);
\draw[-,line width=2pt] (1.5,2.93)--(1,2.7);
\draw[-,cyan,line width=2pt] (4,3)..controls (4.4,3)..(5,2.7);
\draw[-,cyan,line width=2pt] (4,1)..controls (4.4,1)..(5,1.3);
\draw[-,line width=2pt] (2,1)--(4,1);
\draw[-,cyan,line width=2pt] (2,1)--(3,2);
\draw[-,line width=2pt] (3,2)--(3.5,1.5);
\draw[-,cyan,line width=2pt] (3.5,1.5)--(4,1);
\draw[-,line width=2pt] (3,2)--(2,3);
\draw[-,cyan,line width=2pt] (3,2)--(4,3);
\filldraw [red]  (1.5,2.93)    circle (1.5pt)
[red]  (3.5,1.5)    circle (1.5pt);
\filldraw [black]  (2,1)    circle (2pt)
[black]  (2,3)    circle (2pt)
[black]  (4,3)    circle (2pt)
[black]  (3,2)    circle (2pt)
[black]  (4,1)    circle (2pt);
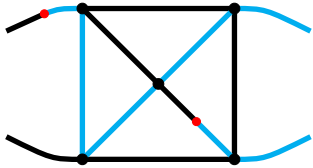
\draw (3,0.4)node{The case $ad$};
\draw(4,4)node{};
\end{tikzpicture}
\begin{tikzpicture}[scale=1]
\draw[-,line width=2pt] (2,3)--(3,3);
\draw[-,cyan,line width=2pt] (3,3)--(4,3);
\draw[-,line width=2pt] (4,3) -- (4,1);
\draw[-,cyan,line width=2pt] (2,1)--(2,3);
\draw[-,line width=2pt] (2,1)..controls (1.6,1)..(1,1.3);
\draw[-,line width=2pt] (2,3)..controls (1.6,3)..(1,2.7);
\draw[-,cyan,line width=2pt] (4,3)..controls (4.4,3)..(5,2.7);
\draw[-,cyan,line width=2pt] (4,1)..controls (4.4,1)..(5,1.3);
\draw[-,line width=2pt] (2,1)--(4,1);
\draw[-,cyan,line width=2pt] (2,1)--(2.5,1.5);
\draw[-,line width=2pt] (2.5,1.5)--(3,2);
\draw[-,cyan,line width=2pt] (3,2)--(4,1);
\draw[-,cyan,line width=2pt] (3,2)--(2,3);
\draw[-,line width=2pt] (3,2)--(4,3);
\filldraw [red]  (3,3)    circle (1.5pt)
[red]  (2.5,1.5)    circle (1.5pt);
\filldraw [black]  (2,1)    circle (2pt)
[black]  (2,3)    circle (2pt)
[black]  (4,3)    circle (2pt)
[black]  (3,2)    circle (2pt)
[black]  (4,1)    circle (2pt);
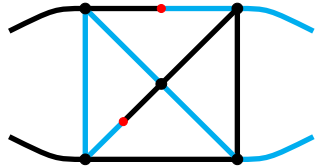
\draw (3,0.4)node{The case $be$};
\draw(0.3,2)node{};
\end{tikzpicture}
\begin{tikzpicture}[scale=1]
\draw[-,line width=2pt] (2.82,3)--(4,3);
\draw[-,cyan,line width=2pt] (2,3)--(2.82,3);
\draw[-,cyan,line width=2pt] (4,3) -- (4,1);
\draw[-,line width=2pt] (2,1)--(2,3);
\draw[-,cyan,line width=2pt] (2,1)..controls (1.6,1)..(1,1.3);
\draw[-,line width=2pt] (2,3)..controls (1.6,3)..(1,2.7);
\draw[-,cyan,line width=2pt] (4,3)..controls (4.4,3)..(5,2.7);
\draw[-,line width=2pt] (4,1)..controls (4.4,1)..(5,1.3);
\draw[-,line width=2pt] (2,1)--(2.83,1);
\draw[-,cyan,line width=2pt] (2.83,1)--(4,1);
\draw[-,cyan,line width=2pt] (2,1)--(3,2);
\draw[-,line width=2pt] (3,2)--(4,1);
\draw[-,cyan,line width=2pt] (3,2)--(2,3);
\draw[-,line width=2pt] (3,2)--(4,3);
\filldraw [red]  (2.83,3)    circle (1.5pt)
[red]  (2.83,1)    circle (1.5pt);
\filldraw [black]  (2,1)    circle (2pt)
[black]  (2,3)    circle (2pt)
[black]  (4,3)    circle (2pt)
[black]  (3,2)    circle (2pt)
[black]  (4,1)    circle (2pt);
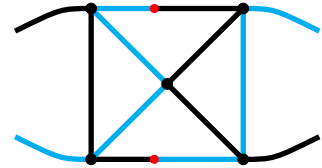
\draw (3,0.4) node{The case $ce$};
\draw(0.3,2)node{};
\end{tikzpicture}

\noindent {\bf Case 22.:} In this case we have one BB with one embedded elementary IBB and an ESU with 10 components:

\begin{tikzpicture}[scale=0.6]
\draw[-,green,line width=2pt] (4,3)-- (3,4);
\draw[-,green,line width=2pt] (4,3) -- (4,1);
\draw[-,red,line width=2pt] (2,1)--(2,3);
\draw[-,green,line width=2pt] (2,1)..controls (1.6,1)..(1,1.3);
\draw[-,green,line width=2pt] (2,3)..controls (1.6,3)..(1,2.7);
\draw[-,red,line width=2pt] (4,1)..controls (4.5,0.7)..(5,1);
\draw[-,green,line width=2pt] (2,1)--(4,1);
\draw[-,red] (2,3)..controls (2.3,3.7)..(3,4);
\draw[-,green] (2,3)..controls (2.7,3.3)..(3,4);
\draw[-,red,line width=2pt] (4,1)..controls (7,2)and(6,4)..(3,4);
\draw[-,red,line width=2pt] (4,3)..controls (4.5,3.3)..(5,3);
\draw[-,red,line width=2pt] (2,1)--(4,3);
\filldraw [black]  (2,1)    circle (2pt)
[black]  (2,3)    circle (2pt)
[black]  (4,3)    circle (2pt)
[black]  (3,4)    circle (2pt)
[black]  (4,1)    circle (2pt);
\draw[green] (1.5,3.25) node{$5$};
\draw[red] (6,3) node{$b$};
\draw[green] (3,0.7) node{$2$};
\draw[red] (2.75,2.25) node{$d$};
\draw[red] (4.7,2.8) node{$e$};
\draw[green] (0.8,1.3) node{$1$};
\draw[red] (1.7,2) node{$c$};
\draw[green] (3.25,3.25) node{$4$};
\draw[green] (4.3,2) node{$3$};
\draw[red] (5.2,1) node{$a$};
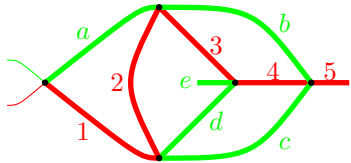
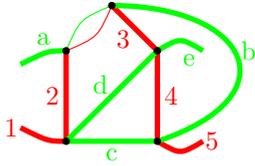
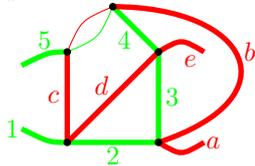
\draw (3,0) node{ESU with 10 components};
\end{tikzpicture}
\begin{tikzpicture}[scale=0.6]
\draw[-,cyan,line width=2pt] (4,3)-- (3,4);
\draw[-,line width=2pt] (4,3) -- (4,1);
\draw[-,cyan,line width=2pt] (2,1)--(2,3);
\draw[-,line width=2pt] (2,1)..controls (1.7,1)..(1.5,1.06);
\draw[-,cyan,line width=2pt] (1.5,1.06)--(1,1.3);
\draw[-,line width=2pt] (2,3)..controls (1.6,3)..(1,2.7);
\draw[-,line width=2pt] (4,1)..controls (4.5,0.7)..(5,1);
\draw[-,cyan,line width=2pt] (2,1)--(4,1);
\draw[-,cyan,line width=2pt] (2,3)..controls (2.3,3.7)..(3,4);
\draw[-,line width=2pt] (2,3)..controls (2.7,3.3)..(3,4);
\draw[-,cyan,line width=2pt] (4,1)..controls (7,2)and(6,4)..(3,4);
\draw[-,line width=2pt] (5,3.6)..controls (4.5,3.87)and (3.7,4)..(3,4);
\draw[-,cyan,line width=2pt] (4,3)..controls (4.5,3.3)..(5,3);
\draw[-,line width=2pt] (2,1)--(4,3);
\filldraw [red]  (1.5,1.06)    circle (1.5pt)
[red]  (5,3.6)    circle (1.5pt);
\filldraw [black]  (2,1)    circle (2pt)
[black]  (2,3)    circle (2pt)
[black]  (4,3)    circle (2pt)
[black]  (3,4)    circle (2pt)
[black]  (4,1)    circle (2pt);
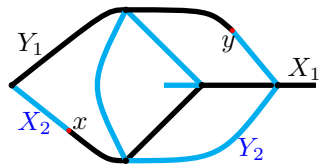
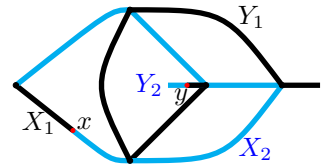
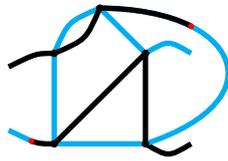
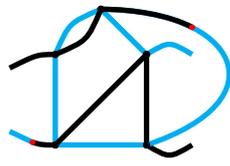
\draw (3,0) node{The case $1b$};
\end{tikzpicture}
\begin{tikzpicture}[scale=0.6]
\draw[-,cyan,line width=2pt] (4,3)-- (3,4);
\draw[-,cyan,line width=2pt] (4,3) -- (4,1);
\draw[-,cyan,line width=2pt] (2,1)--(2,3);
\draw[-,line width=2pt] (2,1)..controls (1.6,1)..(1,1.3);
\draw[-,line width=2pt] (2,3)..controls (1.6,3)..(1,2.7);
\draw[-,cyan,line width=2pt] (4,1)..controls (4.5,0.7)..(5,1);
\draw[-,cyan,line width=2pt] (2,1)--(3,1);
\draw[-,line width=2pt] (3,1)--(4,1);
\draw[-,cyan,line width=2pt] (2,3)..controls (2.3,3.7)..(3,4);
\draw[-,line width=2pt] (2,3)..controls (2.7,3.3)..(3,4);
\draw[-,line width=2pt] (4,1)..controls (7,2)and(6,4)..(3,4);
\draw[-,line width=2pt] (4,3)..controls (4.4,3.225)..(4.5,3.225);
\draw[-,cyan,line width=2pt] (4.5,3.225)..controls (4.6,3.225)..(5,3);
\draw[-,line width=2pt] (2,1)--(4,3);
\filldraw [red]  (4.5,3.225)    circle (1.5pt)
[red]  (3,1)    circle (1.5pt);
\filldraw [black]  (2,1)    circle (2pt)
[black]  (2,3)    circle (2pt)
[black]  (4,3)    circle (2pt)
[black]  (3,4)    circle (2pt)
[black]  (4,1)    circle (2pt);
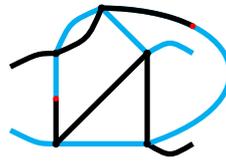
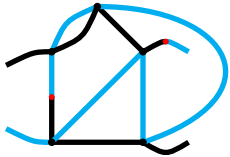
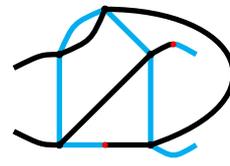
\draw (3,0) node{The case $2e$};
\end{tikzpicture}

\noindent
  One verifies directly that in this ESU all the pairs of components of different colors except seven, are adjacent or both extremal, so
  Propositions~\ref{prop adyacentes}   and~\ref{PROP conexion entre extremos} prove that they are LNC.
  The two pairs $1b$ and $2e$ are proven to be LNC by the paths in the second and third diagrams above,
  where $X_1$ and $Y_1$ are in black and $X_2$ and $Y_2$ are in blue.
  The remaining pairs $3c,4a,4c,5b,5d$ are proven to be LNC by the paths in the following diagrams,
  where $X_1$ and $Y_1$ are in black and $X_2$ and $Y_2$ are in blue.

\begin{tikzpicture}[scale=0.6]
\draw[-,cyan,line width=2pt] (4,3)-- (3,4);
\draw[-,line width=2pt] (4,3) -- (4,2);
\draw[-,cyan,line width=2pt] (4,2) -- (4,1);
\draw[-,line width=2pt] (2,1)--(2,2);
\draw[-,cyan,line width=2pt] (2,2)--(2,3);
\draw[-,cyan,line width=2pt] (2,1)..controls (1.6,1)..(1,1.3);
\draw[-,line width=2pt] (2,3)..controls (1.6,3)..(1,2.7);
\draw[-,cyan,line width=2pt] (4,1)..controls (4.5,0.7)..(5,1);
\draw[-,line width=2pt] (2,1)--(4,1);
\draw[-,cyan,line width=2pt] (2,3)..controls (2.3,3.7)..(3,4);
\draw[-,line width=2pt] (2,3)..controls (2.7,3.3)..(3,4);
\draw[-,line width=2pt] (4,1)..controls (7,2)and(6,4)..(3,4);
\draw[-,line width=2pt] (4,3)..controls (4.5,3.3)..(5,3);
\draw[-,cyan,line width=2pt] (2,1)--(4,3);
\filldraw [red]  (2,2.03)    circle (1.5pt)
[red]  (4,2.03)    circle (1.5pt);
\filldraw [black]  (2,1)    circle (2pt)
[black]  (2,3)    circle (2pt)
[black]  (4,3)    circle (2pt)
[black]  (3,4)    circle (2pt)
[black]  (4,1)    circle (2pt);
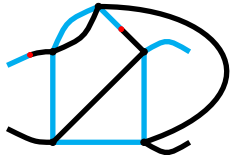
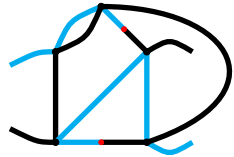
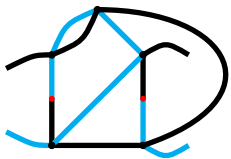
\draw (3,0) node{The case $3c$};
\draw(3,5)node{};
\end{tikzpicture}
\begin{tikzpicture}[scale=0.6]
\draw[-,line width=2pt] (4,3)-- (3.5,3.5);
\draw[-,cyan,line width=2pt] (3.5,3.5)-- (3,4);
\draw[-,cyan,line width=2pt] (4,3) -- (4,1);
\draw[-,cyan,line width=2pt] (2,1)--(2,3);
\draw[-,line width=2pt] (2,1)..controls (1.6,1)..(1,1.3);
\draw[-,line width=2pt] (2,3)..controls (1.6,3)..(1,2.7);
\draw[-,line width=2pt] (4,1)..controls (4.4,0.775)..(4.5,0.775);
\draw[-,cyan,line width=2pt] (4.5,0.775)..controls (4.6,0.775)..(5,1);
\draw[-,cyan,line width=2pt] (2,1)--(4,1);
\draw[-,cyan,line width=2pt] (2,3)..controls (2.3,3.7)..(3,4);
\draw[-,line width=2pt] (2,3)..controls (2.7,3.3)..(3,4);
\draw[-,line width=2pt] (4,1)..controls (7,2)and(6,4)..(3,4);
\draw[-,cyan,line width=2pt] (4,3)..controls (4.5,3.3)..(5,3);
\draw[-,line width=2pt] (2,1)--(4,3);
\filldraw [red]  (4.5,0.775)    circle (1.5pt)
[red]  (3.5,3.5)    circle (1.5pt);
\filldraw [black]  (2,1)    circle (2pt)
[black]  (2,3)    circle (2pt)
[black]  (4,3)    circle (2pt)
[black]  (3,4)    circle (2pt)
[black]  (4,1)    circle (2pt);
\draw (3,0) node{The case $4a$};
\end{tikzpicture}
\begin{tikzpicture}[scale=0.6]
\draw[-,cyan,line width=2pt] (4,3)-- (3.5,3.5);
\draw[-,line width=2pt] (3.5,3.5)-- (3,4);
\draw[-,line width=2pt] (4,3) -- (4,1);
\draw[-,line width=2pt] (2,1)--(2,2);
\draw[-,cyan,line width=2pt] (2,2)--(2,3);
\draw[-,cyan,line width=2pt] (2,1)..controls (1.6,1)..(1,1.3);
\draw[-,line width=2pt] (2,3)..controls (1.6,3)..(1,2.7);
\draw[-,line width=2pt] (4,1)..controls (4.5,0.7)..(5,1);
\draw[-,cyan,line width=2pt] (2,1)--(4,1);
\draw[-,cyan,line width=2pt] (2,3)..controls (2.3,3.7)..(3,4);
\draw[-,line width=2pt] (2,3)..controls (2.7,3.3)..(3,4);
\draw[-,cyan,line width=2pt] (4,1)..controls (7,2)and(6,4)..(3,4);
\draw[-,cyan,line width=2pt] (4,3)..controls (4.5,3.3)..(5,3);
\draw[-,line width=2pt] (2,1)--(4,3);
\filldraw [red]  (2,2)    circle (1.5pt)
[red]  (3.5,3.5)    circle (1.5pt);
\filldraw [black]  (2,1)    circle (2pt)
[black]  (2,3)    circle (2pt)
[black]  (4,3)    circle (2pt)
[black]  (3,4)    circle (2pt)
[black]  (4,1)    circle (2pt);
\draw (3,0) node{The case $4c$};
\end{tikzpicture}

\begin{tikzpicture}[scale=0.6]
\draw[-,cyan,line width=2pt] (4,3)-- (3,4);
\draw[-,line width=2pt] (4,3) -- (4,1);
\draw[-,line width=2pt] (2,1)--(2,3);
\draw[-,cyan,line width=2pt] (2,1)..controls (1.6,1)..(1,1.3);
\draw[-,line width=2pt] (1.5,2.93)--(1,2.7);
\draw[-,cyan,line width=2pt] (2,3)..controls (1.75,3)..(1.5,2.93);
\draw[-,line width=2pt] (4,1)..controls (4.5,0.7)..(5,1);
\draw[-,cyan,line width=2pt] (2,1)--(4,1);
\draw[-,cyan,line width=2pt] (2,3)..controls (2.3,3.7)..(3,4);
\draw[-,line width=2pt] (2,3)..controls (2.7,3.3)..(3,4);
\draw[-,cyan,line width=2pt] (4,1)..controls (7,2)and(6,4)..(3,4);
\draw[-,line width=2pt] (5,3.6)..controls (4.5,3.87)and (3.7,4)..(3,4);
\draw[-,cyan,line width=2pt] (4,3)..controls (4.5,3.3)..(5,3);
\draw[-,line width=2pt] (2,1)--(4,3);
\filldraw [red]  (5,3.6)    circle (1.5pt)
[red]  (1.5,2.93)    circle (1.5pt);
\filldraw [black]  (2,1)    circle (2pt)
[black]  (2,3)    circle (2pt)
[black]  (4,3)    circle (2pt)
[black]  (3,4)    circle (2pt)
[black]  (4,1)    circle (2pt);
\draw (3,0) node{The case $5b$};
\draw(3,5)node{};
\end{tikzpicture}
\begin{tikzpicture}[scale=0.6]
\draw[-,cyan,line width=2pt] (4,3)-- (3,4);
\draw[-,cyan,line width=2pt] (4,3) -- (4,1);
\draw[-,cyan,line width=2pt] (2,1)--(2,3);
\draw[-,line width=2pt] (2,1)..controls (1.6,1)..(1,1.3);
\draw[-,cyan,line width=2pt] (1.5,2.93)--(1,2.7);
\draw[-,line width=2pt] (2,3)..controls (1.75,3)..(1.5,2.93);
\draw[-,cyan,line width=2pt] (4,1)..controls (4.5,0.7)..(5,1);
\draw[-,line width=2pt] (2,1)--(4,1);
\draw[-,cyan,line width=2pt] (2,3)..controls (2.3,3.7)..(3,4);
\draw[-,line width=2pt] (2,3)..controls (2.7,3.3)..(3,4);
\draw[-,line width=2pt] (4,1)..controls (7,2)and(6,4)..(3,4);
\draw[-,line width=2pt] (4,3)..controls (4.5,3.3)..(5,3);
\draw[-,cyan,line width=2pt] (2,1)--(3,2);
\draw[-,line width=2pt] (3,2)--(4,3);
\filldraw [red]  (1.5,2.93)    circle (1.5pt)
[red]  (3,2)    circle (1.5pt);
\filldraw [black]  (2,1)    circle (2pt)
[black]  (2,3)    circle (2pt)
[black]  (4,3)    circle (2pt)
[black]  (3,4)    circle (2pt)
[black]  (4,1)    circle (2pt);
\draw (3,0) node{The case $5d$};
\end{tikzpicture}

\noindent This
  finishes all the cases in the table, and thus we have proven the following theorem.
  \begin{theorem}\label{k igual a 5}
    If $G$ is a simple graph, $P,Q$ are longest paths and $\# V(P)\cap V(Q) =5$,
    then $BT(P,Q)$ is either a concatenation of LNC blocks,
    or it is TD, or it is one of the three cases 6., 13. or 14. in the table above.
  \end{theorem}

\section{In the three exceptional cases $V(P)\cap V(Q)$ is a separator}

In this section we will prove that none of the three exceptional cases of the previous section is
a counterexample to the Hippchen conjecture, which proves the Hippchen conjecture for $k=6$. Then we prove that
in the three cases $V(P)\cap V(Q)$ is a separator.

\begin{lemma}\label{caso coincidentes}
  Assume $G$ is $\ell+1$-connected and let $a_1,b_1,P_0',Q_0'$ be as in Notation~\ref{notation permutation},
  in particular we assume that $\# (V(P)\cap V(Q)))=\ell$.  If $a_1=b_1$, then
  $P_0'\ne \emptyset$ and  $Q_0'\ne \emptyset$.
\end{lemma}

\begin{proof}
  Since $a_1=b_1$, we can interchange $P_0$ and $Q_0$, and $P_0'=\emptyset$ if and only if $Q_0'=\emptyset$.
   Assume by contradiction that $P_0'=Q_0'=\emptyset$, and so $a_1=b_1$ is an endpoint of $P$ and of $Q$.
  Since $G$ is $\ell+1$-connected, there is an edge connecting $a_1$ with a point $t\not\in\{a_2,\dots,a_{\ell}\}=(V(P)\cap V(Q))$,
  which we call $a_1t$.
  If $t\notin V(P)$, then $L(P+a_1 t)>L(P)$ which contradicts the fact that $P$ is a longest path; and if  $t\notin V(Q)$, then $L(Q+a_1 t)>L(Q)$,
  which contradicts the fact that $Q$ is a longest path, concluding the proof.
\end{proof}

\begin{proposition}\label{casos excepcionales con puntas}
  In the three cases 6. 13. and 14., we can assume $a_1=b_1$. In any of the three cases, if $P_0'\ne \emptyset$ and $Q_0'\ne \emptyset$,
  then $V(P)\cap V(Q)$ is a separator.
\end{proposition}

\begin{proof}
  Note that in the three exceptional cases can assume $a_1=b_1$, changing the directions of $P$ and/or $Q$, if necessary.

\noindent {\bf Case 6.:} The following diagrams show that $P_0'$ can connect directly only with $Q_3'$ and $P_3'$, and that $Q_3'$ cannot be 
connected directly with $P_3'$.

  \begin{tikzpicture}[scale=0.6]
\draw[-,line width=2pt] (2,2)..controls (2.8,3)..(2,4);
\draw[-,cyan,line width=2pt] (0,3)..controls (1.4,2)..(2,2);
\draw[-,cyan,line width=2pt] (2,2)..controls (1.4,2.8)..(1.4,3);
\draw[-,line width=2pt] (1.4,3)..controls (1.4,3.3)..(2,4);
\draw[-,cyan,line width=2pt] (2,4)..controls (1.4,4)..(0,3);
\draw[-,cyan,line width=2pt] (4,2)..controls (3.2,3)..(4,4);
\draw[-,line width=2pt] (0,3)--(-0.5,3.42);
\draw[-,line width=2pt,cyan] (-0.5,3.42)..controls (-0.8,3.6)..(-1,3.6);
\draw[-,line width=2pt] (0,3)..controls (-0.7,2.4)..(-1,2.4);
\draw[-,cyan,line width=2pt] (5,2.3)..controls (4.5,2)..(4,2);
\draw[-,line width=2pt] (4,2)..controls (4.8,3)..(4,4);
\draw[-,line width=2pt] (4,4)..controls (4.5,4)..(5,3.7);
\draw[-,cyan,line width=2pt] (2,4)--(4,4);
\draw[-,line width=2pt] (2,2)--(4,2);
\draw[-,white,line width=3pt] (-0.4,3.52)..controls (0,3.8)..(1.2,3.1);
\draw[-,cyan] (-0.5,3.47)..controls (0,3.83)..(1.4,3.03);
\draw[-] (-0.5,3.39)..controls (0,3.77)..(1.4,2.97);
\filldraw [red]  (-0.5,3.42)    circle (1.5pt)
[red]  (1.4,3)    circle (1.5pt);
\filldraw [black]  (2,2)    circle (2pt)
[black]  (4,2)    circle (2pt)
[black]  (2,4)    circle (2pt)
[black]  (0,3)    circle (2pt)
[black]  (4,4)    circle (2pt);
\draw(-1,4)node{$P_0'$};
\draw(3,4.35)node{$Q_3'$};
\draw(3,1.6)node{$P_3'$};
\draw(6,2)node{};
\end{tikzpicture}
  \begin{tikzpicture}[scale=0.6]
\draw[-,line width=2pt] (2,2)..controls (2.8,3)..(2,4);
\draw[-,line width=2pt] (0,3)..controls (1.4,2)..(2,2);
\draw[-,cyan,line width=2pt] (2,2)..controls (1.2,3)..(2,4);
\draw[-,cyan,line width=2pt] (2,4)..controls (1.4,4)..(0,3);
\draw[-,cyan,line width=2pt] (4,2)..controls (3.2,3)..(4,4);
\draw[-,line width=2pt] (0,3)--(-0.5,3.42);
\draw[-,line width=2pt,cyan] (-0.5,3.42)..controls (-0.8,3.6)..(-1,3.6);
\draw[-,cyan,line width=2pt] (0,3)..controls (-0.7,2.4)..(-1,2.4);
\draw[-,line width=2pt] (5,2.3)..controls (4.5,2)..(4,2);
\draw[-,line width=2pt] (4,2)..controls (4.6,2.7)..(4.6,3);
\draw[-,cyan,line width=2pt] (4.6,3)..controls (4.6,3.3)..(4,4);
\draw[-,line width=2pt] (4,4)..controls (4.5,4)..(5,3.7);
\draw[-,line width=2pt] (2,4)--(4,4);
\draw[-,cyan,line width=2pt] (2,2)--(4,2);
\draw[-,cyan] (-0.5,3.42)..controls (1,5.5)and(5.5,5.5)..(5.5,4);
\draw[-,cyan] (5.5,4)..controls (5.5,3.5)..(4.6,3);
\draw[-] (-0.53,3.45)..controls (1,5.55)and(5.52,5.56)..(5.54,4);
\draw[-] (5.54,4)..controls (5.54,3.47)..(4.6,2.96);
\filldraw [red]  (-0.5,3.42)    circle (1.5pt)
[red]  (4.6,3)    circle (1.5pt);
\filldraw [black]  (2,2)    circle (2pt)
[black]  (4,2)    circle (2pt)
[black]  (2,4)    circle (2pt)
[black]  (0,3)    circle (2pt)
[black]  (4,4)    circle (2pt);
\draw(-1,4)node{$P_0'$};
\draw(3,4.35)node{$Q_3'$};
\draw(3,1.6)node{$P_3'$};
\draw(6,2)node{};
\end{tikzpicture}
  \begin{tikzpicture}[scale=0.6]
\draw[-,line width=2pt] (2,2)..controls (2.8,3)..(2,4);
\draw[-,cyan,line width=2pt] (0,3)..controls (1.4,2)..(2,2);
\draw[-,cyan,line width=2pt] (2,2)..controls (1.2,3)..(2,4);
\draw[-,line width=2pt] (2,4)..controls (1.4,4)..(0,3);
\draw[-,cyan,line width=2pt] (4,2)..controls (3.2,3)..(4,4);
\draw[-,line width=2pt] (0,3)..controls (-0.7,3.6)..(-1,3.6);
\draw[-,cyan,line width=2pt] (0,3)..controls (-0.7,2.4)..(-1,2.4);
\draw[-,line width=2pt] (5,2.3)..controls (4.5,2)..(4,2);
\draw[-,line width=2pt] (4,2)..controls (4.8,3)..(4,4);
\draw[-,cyan,line width=2pt] (4,4)..controls (4.5,4)..(5,3.7);
\draw[-,cyan,line width=2pt] (2,4)--(3,4);
\draw[-,line width=2pt] (2,2)--(3,2);
\draw[-,line width=2pt] (3,4)--(4,4);
\draw[-,cyan,line width=2pt] (3,2)--(4,2);
\draw[-] (2.97,4)--(2.97,2);
\draw[-,cyan] (3.03,4)--(3.03,2);
\filldraw [red]  (3,4)    circle (1.5pt)
[red]  (3,2)    circle (1.5pt);
\filldraw [black]  (2,2)    circle (2pt)
[black]  (4,2)    circle (2pt)
[black]  (2,4)    circle (2pt)
[black]  (0,3)    circle (2pt)
[black]  (4,4)    circle (2pt);
\draw(-1,4)node{$P_0'$};
\draw(3,4.35)node{$Q_3'$};
\draw(3,1.6)node{$P_3'$};
\end{tikzpicture}

\noindent In fact, the blue and black path show that the connections shown in the diagrams are impossible. By symmetry of the elementary ESU's,
 $P_0'$ cannot connect with each of the other components of the elementary ESU's. Moreover, the adjacent components cannot be connected with
 $P_0'$ by Lemma~\ref{lema 4.1} and the extremal components cannot be connected by
 Lemma~\ref{conexion entre extremos}. So $Q_3'$ and $P_3'$ are the only left. By symmetry
 $Q_0'$ can connect directly only with $Q_3'$ and $P_3'$. But $P_0'$ and $Q_0'$ form an NC pair, so they cannot connect to the same component
 $X\in \{P_3',Q_3'\}$.
On the other hand, $Q_3'$ can connect only with $P_0'$ and $Q_0'$, since any other pair that contains $Q_3'$  but not $P_3'$,
   is either adjacent, or can be discarded by the paths in the following diagrams.

     
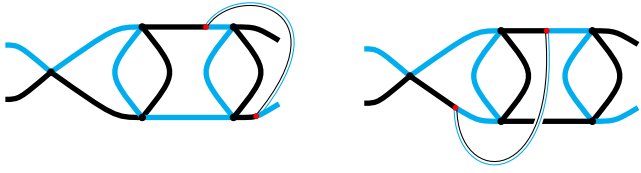
\begin{tikzpicture}[scale=0.6]
\draw[-,line width=2pt] (2,2)..controls (2.8,3)..(2,4);
\draw[-,line width=2pt] (0,3)..controls (1.4,2)..(2,2);
\draw[-,cyan,line width=2pt] (2,2)..controls (1.2,3)..(2,4);
\draw[-,cyan,line width=2pt] (2,4)..controls (1.4,4)..(0,3);
\draw[-,cyan,line width=2pt] (4,2)..controls (3.2,3)..(4,4);
\draw[-,cyan,line width=2pt] (0,3)..controls (-0.7,3.6)..(-1,3.6);
\draw[-,line width=2pt] (0,3)..controls (-0.7,2.4)..(-1,2.4);
\draw[-,cyan,line width=2pt] (5,2.3)..controls (4.6,2.06)..(4.5,2.03);
\draw[-,line width=2pt] (4.5,2.03)..controls (4.4,2)..(4,2);
\draw[-,line width=2pt] (4,2)..controls (4.8,3)..(4,4);
\draw[-,line width=2pt] (4,4)..controls (4.5,4)..(5,3.7);
\draw[-,line width=2pt] (2,4)--(3.4,4);
\draw[-,cyan,line width=2pt] (3.4,4)--(4,4);
\draw[-,cyan,line width=2pt] (2,2)--(4,2);
\draw[-,cyan] (4.54,2)..controls (6.6,4.55)and(4.05,5.05)..(3.4,4.06);
\draw[-] (4.5,2.06)..controls (6.5,4.5)and(4,5)..(3.4,3.97);
\filldraw [red]  (3.4,4)    circle (1.5pt)
[red]  (4.5,2.03)    circle (1.5pt);
\filldraw [black]  (2,2)    circle (2pt)
[black]  (4,2)    circle (2pt)
[black]  (2,4)    circle (2pt)
[black]  (0,3)    circle (2pt)
[black]  (4,4)    circle (2pt);
\draw(6,0.5)node{};
\end{tikzpicture}
  \begin{tikzpicture}[scale=0.6]
\draw[-,line width=2pt] (2,2)..controls (2.8,3)..(2,4);
\draw[-,cyan,line width=2pt] (0,3)..controls (1.4,4)..(2,4);
\draw[-,cyan,line width=2pt] (2,4)..controls (1.2,3)..(2,2);
\draw[-,cyan,line width=2pt] (2,2)..controls (1.5,2)..(1,2.3);
\draw[-,line width=2pt] (1,2.3)--(0,3);
\draw[-,cyan,line width=2pt] (4,4)..controls (3.2,3)..(4,2);
\draw[-,line width=2pt] (0,3)..controls (-0.7,2.4)..(-1,2.4);
\draw[-,cyan,line width=2pt] (0,3)..controls (-0.7,3.6)..(-1,3.6);
\draw[-,line width=2pt] (5,3.7)..controls (4.5,4)..(4,4);
\draw[-,line width=2pt] (4,4)..controls (4.8,3)..(4,2);
\draw[-,cyan,line width=2pt] (4,2)..controls (4.5,2)..(5,2.3);
\draw[-,line width=2pt] (2,2)--(4,2);
\draw[-,line width=2pt] (2,4)--(3,4);
\draw[-,cyan,line width=2pt] (3,4)--(4,4);
\draw[-,white,line width=3pt] (1,2.1)..controls (1,0.5)and(3.3,0.5)..(2.95,3.8);
\draw[-,cyan] (0.97,2.3)..controls (0.97,0.4)and(3.33,0.4)..(3.03,4);
\draw[-] (1.03,2.3)..controls (1.03,0.5)and(3.27,0.5)..(2.97,4);
\filldraw [red]  (1,2.3)    circle (1.5pt)
[red]  (3,4)    circle (1.5pt);
\filldraw [black]  (2,2)    circle (2pt)
[black]  (4,2)    circle (2pt)
[black]  (2,4)    circle (2pt)
[black]  (0,3)    circle (2pt)
[black]  (4,4)    circle (2pt);
\draw(6,1)node{};
\end{tikzpicture}

   \noindent The symmetric argument shows that $P_3'$ can connect only with $P_0'$ and $Q_0'$.
 Hence $P_0'$ and $Q_0'$ cannot be connected in $G\setminus (V(P)\cap V(Q))$, which shows that $V(P)\cap V(Q)$ is a separator, as desired.

\noindent {\bf Case 13.:} The following diagrams show that $P_0'$ cannot be connected to any other component.

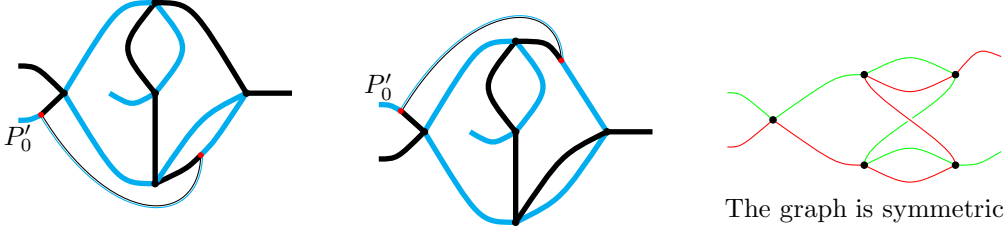
\begin{tikzpicture}[scale=0.6]
\draw[-,cyan,line width=2pt] (2,4)..controls (2.8,5)..(2,6);
\draw[-,line width=2pt] (2,2)..controls (2.7,2.33)..(3,2.63);
\draw[-,cyan,line width=2pt] (3,2.63)..controls (3.3,2.93)..(4,4);
\draw[-,cyan,line width=2pt] (2,2)..controls (3,3.5)..(4,4);
\draw[-,line width=2pt] (2,6)..controls (2.8,6)..(4,4);
\draw[-,cyan,line width=2pt] (2,4)..controls (1.5,3.7)..(1,4);
\draw[-,line width=2pt] (-1,4.6)..controls (-0.6,4.6)..(0,4);
\draw[-,line width=2pt,cyan] (-1,3.4)..controls (-0.75,3.4)..(-0.5,3.52);
\draw[-,line width=2pt] (-0.5,3.52)--(0,4);
\draw[-,cyan,line width=2pt] (0,4)..controls (1.2,2)..(2,2);
\draw[-,line width=2pt] (2,2)--(2,4);
\draw[-,line width=2pt] (4,4)--(5,4);
\draw[-,line width=2pt] (2,4)..controls (1.2,5)..(2,6);
\draw[-,cyan,line width=2pt] (2,6)..controls (1.2,6)..(0,4);
\draw[-] (-0.5,3.55)..controls (1,1.02)and (2.96,1)..(3,2.63);
\draw[-,cyan] (-0.52,3.5)..controls (1,0.96)and (3,0.96)..(3.04,2.63);
\filldraw [red]  (-0.5,3.52)    circle (1.5pt)
[red]  (3,2.63)    circle (1.5pt);
\filldraw [black]  (2,2)    circle (2pt)
[black]  (2,4)    circle (2pt)
[black]  (2,6)    circle (2pt)
[black]  (4,4)    circle (2pt)
[black]  (0,4)    circle (2pt);
\draw (-1,3)node{$P_0'$};
\draw (6,2)node{};
\end{tikzpicture}
\begin{tikzpicture}[scale=0.6]
\draw[-,cyan,line width=2pt] (2,4)..controls (2.8,5)..(2,6);
\draw[-,cyan,line width=2pt] (2,2)..controls (3,2.5)..(4,4);
\draw[-,line width=2pt] (2,2)..controls (3,3.5)..(4,4);
\draw[-,line width=2pt] (2,6)..controls (2.7,6)..(3,5.575);
\draw[-,cyan,line width=2pt] (3,5.575)--(4,4);
\draw[-,cyan,line width=2pt] (2,4)..controls (1.5,3.7)..(1,4);
\draw[-,cyan,line width=2pt] (-1,4.6)..controls (-0.75,4.6)..(-0.5,4.45);
\draw[-,line width=2pt] (-0.5,4.45)--(0,4);
\draw[-,line width=2pt] (-1,3.4)..controls (-0.6,3.4)..(0,4);
\draw[-,cyan,line width=2pt] (0,4)..controls (1.2,2)..(2,2);
\draw[-,line width=2pt] (2,2)--(2,4);
\draw[-,line width=2pt] (4,4)--(5,4);
\draw[-,line width=2pt] (2,4)..controls (1.2,5)..(2,6);
\draw[-,cyan,line width=2pt] (2,6)..controls (1.2,6)..(0,4);
\draw[-] (-0.5,4.45)..controls (1,6.98)and (2.96,7)..(3,5.575);
\draw[-,cyan] (-0.52,4.5)..controls (1,7.04)and (3,7.04)..(3.04,5.575);
\filldraw [red]  (-0.52,4.46)    circle (1.5pt)
[red]  (3,5.575)    circle (1.5pt);
\filldraw [black]  (2,2)    circle (2pt)
[black]  (2,4)    circle (2pt)
[black]  (2,6)    circle (2pt)
[black]  (4,4)    circle (2pt)
[black]  (0,4)    circle (2pt);
\draw (-1,5)node{$P_0'$};
\draw (6,2)node{};
\end{tikzpicture}
\begin{tikzpicture}[scale=0.6]
\draw[-,green] (2,2)..controls (3,2.5)..(4,2);
\draw[-,red] (0,3)..controls (1.4,2)..(2,2);
\draw[-,red] (2,2)..controls (3,1.5)..(4,2);
\draw[-,green] (2,4)..controls (1.4,4)..(0,3);
\draw[-,red] (2,4)..controls (3,3.5)..(4,4);
\draw[-,green] (0,3)..controls (-0.7,3.6)..(-1,3.6);
\draw[-,red] (0,3)..controls (-0.7,2.4)..(-1,2.4);
\draw[-,green] (5,2.3)..controls (4.5,2)..(4,2);
\draw[-,green] (2,4)..controls (3,4.5)..(4,4);
\draw[-,red] (4,4)..controls (4.5,4.6)..(5,4.4);
\draw[-,green] (2,2)..controls (2,2.6) and (4,3.4)..(4,4);
\draw[-,white,line width=2pt] (4,2)..controls (4,2.5)and(2,3.5)..(2,4);
\draw[-,red] (4,2)..controls (4,2.5)and(2,3.5)..(2,4);

\filldraw [black]  (2,2)    circle (2pt)
[black]  (4,2)    circle (2pt)
[black]  (2,4)    circle (2pt)
[black]  (0,3)    circle (2pt)
[black]  (4,4)    circle (2pt);
\draw (2,1)node{The graph is symmetric};
\end{tikzpicture}

\noindent In fact, the first two diagrams show paths that prevent $P_0'$ to be connected to two components, and by symmetry also to
the other component of the elementary ESU. The third diagram shows that
the graph is symmetric, and so we can discard three more components. The remaining components are either adjacent to $P_0'$ or are extremal,
so they cannot be connected to $P_0'$ by Lemmas~\ref{lema 4.1} and~\ref{conexion entre extremos}. Hence there cannot be a path
from $P_0'$ to $Q_0'$ in $G\setminus (V(P)\cap V(Q))$, and so $V(P)\cap V(Q)$ is a separator, as desired.

\noindent {\bf Case 14.:} The following diagrams show that $P_0'$ cannot be connected to any other component.

\noindent 
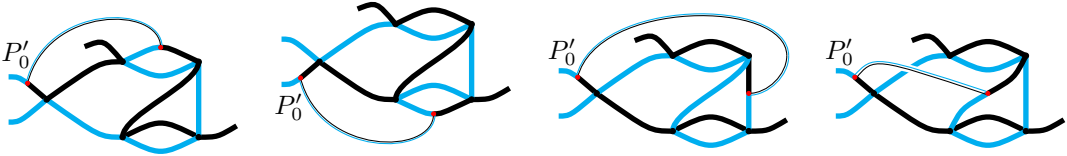
\begin{tikzpicture}[scale=0.5]
\draw[-,line width=2pt] (2,2)..controls (3,2.5)..(4,2);
\draw[-,cyan,line width=2pt] (0,3)..controls (1.4,2)..(2,2);
\draw[-,cyan,line width=2pt] (2,2)..controls (3,1.5)..(4,2);
\draw[-,line width=2pt] (2,4)..controls (1.4,4)..(0,3);
\draw[-,cyan,line width=2pt] (2,4)..controls (3,3.5)..(4,4);
\draw[-,line width=2pt] (0,3)--(-0.5,3.43);
\draw[-,cyan,line width=2pt] (-0.5,3.43)..controls (-0.75,3.6)..(-1,3.6);
\draw[-,cyan,line width=2pt] (0,3)..controls (-0.7,2.4)..(-1,2.4);
\draw[-,line width=2pt] (5,2.3)..controls (4.5,2)..(4,2);
\draw[-,cyan,line width=2pt] (2,4)..controls (2.8,4.38)..(3,4.38);
\draw[-,line width=2pt] (3,4.38)..controls (3.2,4.38)..(4,4);
\draw[-,line width=2pt] (2,4)..controls (1.5,4.6)..(1,4.4);
\draw[-,cyan,line width=2pt] (4,2)--(4,4);
\draw[-,line width=2pt] (2,2)..controls (2,2.6) and (4,3.4)..(4,4);
\draw[-] (-0.5,3.43)..controls (0.02,5.5) and (3,5.5)..(3,4.38);
\draw[-,cyan] (-0.54,3.43)..controls (-0.03,5.56) and (3,5.55)..(3.04,4.4);
\filldraw [red]  (-0.5,3.43)    circle (1.5pt)
[red]  (3,4.38)    circle (1.5pt);
\filldraw [black]  (2,2)    circle (2pt)
[black]  (4,2)    circle (2pt)
[black]  (2,4)    circle (2pt)
[black]  (0,3)    circle (2pt)
[black]  (4,4)    circle (2pt);
\draw (-0.8,4.2)node{$P_0'$};
\draw (5.3,2)node{};
\end{tikzpicture}
\begin{tikzpicture}[scale=0.5]
\draw[-,cyan,line width=2pt] (2,2)..controls (3,2.5)..(4,2);
\draw[-,line width=2pt] (0,3)..controls (1.4,2)..(2,2);
\draw[-,cyan,line width=2pt] (2,2)..controls (2.8,1.62)..(3,1.62);
\draw[-,line width=2pt] (3,1.62)..controls (3.2,1.62)..(4,2);
\draw[-,cyan,line width=2pt] (2,4)..controls (1.4,4)..(0,3);
\draw[-,cyan,line width=2pt] (2,4)..controls (3,3.5)..(4,4);
\draw[-,cyan,line width=2pt] (0,3)..controls (-0.7,3.6)..(-1,3.6);
\draw[-,line width=2pt] (0,3)--(-0.5,2.57);
\draw[-,cyan,line width=2pt] (-0.5,2.57)..controls (-0.75,2.4)..(-1,2.4);
\draw[-,line width=2pt] (5,2.3)..controls (4.5,2)..(4,2);
\draw[-,line width=2pt] (2,4)..controls (3,4.5)..(4,4);
\draw[-,line width=2pt] (2,4)..controls (1.5,4.6)..(1,4.4);
\draw[-,cyan,line width=2pt] (4,2)--(4,4);
\draw[-,line width=2pt] (2,2)..controls (2,2.6) and (4,3.4)..(4,4);
\draw[-] (-0.5,2.57)..controls (0.02,0.5) and (3,0.5)..(3,1.62);
\draw[-,cyan] (-0.54,2.57)..controls (-0.03,0.44) and (3,0.45)..(3.04,1.6);
\filldraw [red]  (-0.5,2.57)    circle (1.5pt)
[red]  (3,1.62)    circle (1.5pt);
\filldraw [black]  (2,2)    circle (2pt)
[black]  (4,2)    circle (2pt)
[black]  (2,4)    circle (2pt)
[black]  (0,3)    circle (2pt)
[black]  (4,4)    circle (2pt);
\draw (-0.8,1.8)node{$P_0'$};
\draw (5.3,4.5)node{};
\end{tikzpicture}
\begin{tikzpicture}[scale=0.5]
\draw[-,line width=2pt] (2,2)..controls (3,2.5)..(4,2);
\draw[-,line width=2pt] (0,3)..controls (1.4,2)..(2,2);
\draw[-,cyan,line width=2pt] (2,2)..controls (3,1.5)..(4,2);
\draw[-,cyan,line width=2pt] (2,4)..controls (1.4,4)..(0,3);
\draw[-,cyan,line width=2pt] (2,4)..controls (3,3.5)..(4,4);
\draw[-,line width=2pt] (0,3)--(-0.5,3.43);
\draw[-,cyan,line width=2pt] (-0.5,3.43)..controls (-0.75,3.6)..(-1,3.6);
\draw[-,cyan,line width=2pt] (0,3)..controls (-0.7,2.4)..(-1,2.4);
\draw[-,line width=2pt] (5,2.3)..controls (4.5,2)..(4,2);
\draw[-,line width=2pt] (2,4)..controls (3,4.5)..(4,4);
\draw[-,line width=2pt] (2,4)..controls (1.5,4.6)..(1,4.4);
\draw[-,cyan,line width=2pt] (4,2)--(4,3);
\draw[-,line width=2pt] (4,3)--(4,4);
\draw[-,cyan,line width=2pt] (2,2)..controls (2,2.6) and (4,3.4)..(4,4);
\draw[-] (-0.5,3.43)..controls (0.02,5.5) and (5,5.5)..(5,4);
\draw[-] (4,3)..controls (4.5,3)and(5,3.5)..(5,4);
\draw[-,cyan] (-0.54,3.43)..controls (-0.03,5.56) and (5.04,5.57)..(5.05,4);
\draw[-,cyan] (4,2.95)..controls (4.53,2.95)and(5.06,3.48)..(5.05,4);
\filldraw [red]  (-0.5,3.43)    circle (1.5pt)
[red]  (4,3)    circle (1.5pt);
\filldraw [black]  (2,2)    circle (2pt)
[black]  (4,2)    circle (2pt)
[black]  (2,4)    circle (2pt)
[black]  (0,3)    circle (2pt)
[black]  (4,4)    circle (2pt);
\draw (-0.9,4.1)node{$P_0'$};
\draw (5.3,1.5)node{};
\end{tikzpicture}
\begin{tikzpicture}[scale=0.5]
\draw[-,line width=2pt] (2,2)..controls (3,2.5)..(4,2);
\draw[-,line width=2pt] (0,3)..controls (1.4,2)..(2,2);
\draw[-,cyan,line width=2pt] (2,2)..controls (3,1.5)..(4,2);
\draw[-,cyan,line width=2pt] (2,4)..controls (1.4,4)..(0,3);
\draw[-,cyan,line width=2pt] (2,4)..controls (3,3.5)..(4,4);
\draw[-,line width=2pt] (0,3)--(-0.5,3.43);
\draw[-,cyan,line width=2pt] (-0.5,3.43)..controls (-0.75,3.6)..(-1,3.6);
\draw[-,cyan,line width=2pt] (0,3)..controls (-0.7,2.4)..(-1,2.4);
\draw[-,line width=2pt] (5,2.3)..controls (4.5,2)..(4,2);
\draw[-,line width=2pt] (2,4)..controls (3,4.5)..(4,4);
\draw[-,line width=2pt] (2,4)..controls (1.5,4.6)..(1,4.4);
\draw[-,cyan,line width=2pt] (4,2)--(4,4);
\draw[-,cyan,line width=2pt] (2,2)..controls (2.3,2.6) ..(3,3);
\draw[-,line width=2pt] (3,3)..controls (3.7,3.4)..(4,4);
\draw[-,white,line width=3pt] (-0.4,3.53)..controls (0,3.96)..(2.8,3.1);
\draw[-,cyan] (-0.53,3.46)..controls (0,4.03)..(3.03,3.03);
\draw[-] (-0.47,3.4)..controls (0,3.97)..(2.97,2.97);
\filldraw [red]  (-0.5,3.43)    circle (1.5pt)
[red]  (3,3)    circle (1.5pt);
\filldraw [black]  (2,2)    circle (2pt)
[black]  (4,2)    circle (2pt)
[black]  (2,4)    circle (2pt)
[black]  (0,3)    circle (2pt)
[black]  (4,4)    circle (2pt);
\draw (-0.9,4.1)node{$P_0'$};
\draw (5.5,1.5)node{};
\end{tikzpicture}

\noindent In fact, the diagrams show paths that prevent $P_0'$ to be connected to four components, and by symmetry also to
the other component of the elementary ESU's. The remaining components are either adjacent to $P_0'$ or are extremal,
so they cannot be connected to $P_0'$ by Lemmas~\ref{lema 4.1} and~\ref{conexion entre extremos}. Hence there cannot be a path
from $P_0'$ to $Q_0'$ in $G\setminus (V(P)\cap V(Q))$, and so $V(P)\cap V(Q)$ is a separator, as desired. This finishes the three
cases and concludes the proof.
\end{proof}

\begin{corollary}\label{Corollary Hippchen 6}
  Assume that $P$ and $Q$ are two longest paths in a $6$-connected simple graph $G$. Then $\# (V(P)\cap V(Q))\ge 6$.
\end{corollary}

\begin{proof}
  We know that $\# V(P)\ge 6$, so we can assume that $V(P)\ne V(Q)$.
  Since a $6$-connected graph is also $5$-connected, by Corollary~\ref{Corollary Hippchen 5}
   we know that $\# (V(P)\cap V(Q))\ge 5$.
  Assume by contradiction that $\# (V(P)\cap V(Q))<6$, i.e., that $\# (V(P)\cap V(Q))= 5$. Since $G$ is $6$-connected,
  the complement of $V(P)\cap V(Q)$ is connected, which contradicts the fact that by Lemma~\ref{caso coincidentes}
  and Proposition~\ref{casos excepcionales con puntas} we know that
  $V(P)\cap V(Q)$ is a separator. This contradiction concludes the proof.
\end{proof}

Now we prove that in the three exceptional cases $V(P)\cap V(Q)$ is always a separator.
\begin{lemma} \label{final vacio}
  Let  $P_0'$, $Q_0'$, $P_{\ell}'$ and $Q_{\ell}'$ be as in Notation~\ref{notation permutation},
  and assume that one of $P_0'$, $Q_0'$, $P_{\ell}'$ or $Q_{\ell}'$ is empty. Then the two adjacent partial paths of the other longest path
  have length $1$.
\end{lemma}

\begin{proof}
  Clear, since otherwise one can extend $P$ (respectively $Q$) using the first part of the adjacent partial path.
\end{proof}

We will show that in each of the three cases, if $P_0'=Q_0'= \emptyset$, then $BT(P,Q)$ is weakly
disconnected (WD), according to the following definition.

\begin{definition}\label{def WD}
  The graph $BT(P,Q)$  is called \textbf{weakly disconnected (WD)} if each pair $X,Y$ of
  components of different colors  is either NDC, or is \textbf{weakly non directly connectable (WNDC)}, which means that
    for every path $R$ that connects $X$ with $Y$, such that $R$ is internally disjoint from $BT(P,Q)$, we can  construct
    a path $\widehat P$ and a
   cycle $C$ in $BT(P,Q)\cup R$ such that
  $$
  L(\widehat P)+L(C)= 2L(R)+ L(P)+L(Q)=2 L(P)+2 L(R).
  $$
\end{definition}
Clearly in a WD graph with $V(P)\ne V(Q)$, the set $V(P)\cap V(Q)$ is a separator, since the length of $\widehat P$ and the length of the opened 
cycle $\widetilde C$
sum $L(\widehat P)+L(C)-1=2L(P)+2L(R)-1>2L(P)$, which shows that no pair of components of different colors can be connected.

\begin{proposition}\label{casos excepcionales con puntas2}
  Consider the three cases 6. 13. and 14. and assume $a_1=b_1$. If $P_0'=Q_0'= \emptyset$,
  then $V(P)\cap V(Q)$ is a separator.
\end{proposition}

\begin{proof}
{\bf Case 6.:} The paths in the following diagrams show that the graph is TD.

\noindent  
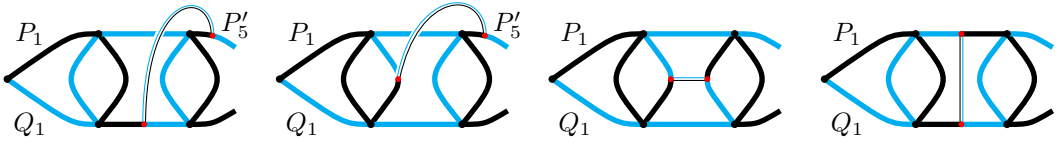
\begin{tikzpicture}[scale=0.6]
\draw[-,line width=2pt] (2,2)..controls (2.8,3)..(2,4);
\draw[-,cyan,line width=2pt] (0,3)..controls (1.4,2)..(2,2);
\draw[-,cyan,line width=2pt] (2,2)..controls (1.2,3)..(2,4);
\draw[-,line width=2pt] (2,4)..controls (1.4,4)..(0,3);
\draw[-,cyan,line width=2pt] (4,2)..controls (3.2,3)..(4,4);
\draw[-,line width=2pt] (5,2.3)..controls (4.5,2)..(4,2);
\draw[-,line width=2pt] (4,2)..controls (4.8,3)..(4,4);
\draw[-,line width=2pt] (4,4)..controls (4.3,4)..(4.5,3.96);
\draw[-,cyan,line width=2pt] (4.5,3.96)..controls (4.7,3.9)..(5,3.7);
\draw[-,cyan,line width=2pt] (2,4)--(4,4);
\draw[-,line width=2pt] (2,2)--(3,2);
\draw[-,cyan,line width=2pt] (3,2)--(4,2);
\draw[-,white,line width=3pt] (3,2.1)..controls (3,5)and (4.5,5)..(4.5,4.1);
\draw[-,cyan] (2.97,2)..controls (2.96,5.06)and (4.53,5.05)..(4.53,3.96);
\draw[-] (3.03,2)..controls (3.04,4.97)and (4.47,5)..(4.47,3.96);
\filldraw [red]  (4.5,3.96)    circle (1.5pt)
[red]  (3,2)    circle (1.5pt);
\filldraw [black]  (2,2)    circle (2pt)
[black]  (4,2)    circle (2pt)
[black]  (2,4)    circle (2pt)
[black]  (0,3)    circle (2pt)
[black]  (4,4)    circle (2pt);
\draw (0.5,4) node{$P_1$};
\draw (5,4.2) node{$P_5'$};
\draw (0.5,2) node{$Q_1$};
\draw (5.5,2)node{};
\end{tikzpicture}
\begin{tikzpicture}[scale=0.6]
\draw[-,line width=2pt] (2,2)..controls (2.6,2.8)..(2.6,3);
\draw[-,cyan,line width=2pt] (2.6,3)..controls (2.6,3.2)..(2,4);
\draw[-,cyan,line width=2pt] (0,3)..controls (1.4,2)..(2,2);
\draw[-,line width=2pt] (2,2)..controls (1.2,3)..(2,4);
\draw[-,line width=2pt] (2,4)..controls (1.4,4)..(0,3);
\draw[-,cyan,line width=2pt] (4,2)..controls (3.2,3)..(4,4);
\draw[-,line width=2pt] (5,2.3)..controls (4.5,2)..(4,2);
\draw[-,line width=2pt] (4,2)..controls (4.8,3)..(4,4);
\draw[-,line width=2pt] (4,4)..controls (4.3,4)..(4.5,3.96);
\draw[-,cyan,line width=2pt] (4.5,3.96)..controls (4.7,3.9)..(5,3.7);
\draw[-,cyan,line width=2pt] (2,4)--(4,4);
\draw[-,cyan,line width=2pt] (2,2)--(4,2);
\draw[-,white,line width=3pt] (2.6,3)..controls (3,5)and (4.5,5)..(4.5,4.1);
\draw[-,cyan] (2.57,3)..controls (2.96,5.06)and (4.53,5.05)..(4.53,3.96);
\draw[-] (2.63,3)..controls (3.04,4.97)and (4.47,5)..(4.47,3.96);
\filldraw [red]  (4.5,3.96)    circle (1.5pt)
[red]  (2.6,3)    circle (1.5pt);
\filldraw [black]  (2,2)    circle (2pt)
[black]  (4,2)    circle (2pt)
[black]  (2,4)    circle (2pt)
[black]  (0,3)    circle (2pt)
[black]  (4,4)    circle (2pt);
\draw (0.5,4) node{$P_1$};
\draw (0.5,2) node{$Q_1$};
\draw (5,4.2) node{$P_5'$};
\draw (5.5,2)node{};
\end{tikzpicture}
\begin{tikzpicture}[scale=0.6]
\draw[-,line width=2pt] (2,2)..controls (2.6,2.8)..(2.6,3);
\draw[-,cyan,line width=2pt] (2.6,3)..controls (2.6,3.2)..(2,4);
\draw[-,cyan,line width=2pt] (0,3)..controls (1.4,2)..(2,2);
\draw[-,line width=2pt] (2,2)..controls (1.2,3)..(2,4);
\draw[-,line width=2pt] (2,4)..controls (1.4,4)..(0,3);
\draw[-,cyan,line width=2pt] (4,2)..controls (3.4,2.8)..(3.4,3);
\draw[-,line width=2pt] (3.4,3)..controls (3.4,3.2)..(4,4);
\draw[-,line width=2pt] (5,2.3)..controls (4.5,2)..(4,2);
\draw[-,line width=2pt] (4,2)..controls (4.8,3)..(4,4);
\draw[-,cyan,line width=2pt] (4,4)..controls (4.5,4)..(5,3.7);
\draw[-,cyan,line width=2pt] (2,4)--(4,4);
\draw[-,cyan,line width=2pt] (2,2)--(4,2);
\draw[-,cyan] (2.6,3.03)--(3.4,3.03);
\draw[-] (2.6,2.97)--(3.4,2.97);
\filldraw [red]  (3.4,3)    circle (1.5pt)
[red]  (2.6,3)    circle (1.5pt);
\filldraw [black]  (2,2)    circle (2pt)
[black]  (4,2)    circle (2pt)
[black]  (2,4)    circle (2pt)
[black]  (0,3)    circle (2pt)
[black]  (4,4)    circle (2pt);
\draw (0.5,4) node{$P_1$};
\draw (0.5,2) node{$Q_1$};
\draw (5.5,2)node{};
\end{tikzpicture}
\begin{tikzpicture}[scale=0.6]
\draw[-,line width=2pt] (2,2)..controls (2.8,3)..(2,4);
\draw[-,cyan,line width=2pt] (0,3)..controls (1.4,2)..(2,2);
\draw[-,cyan,line width=2pt] (2,2)..controls (1.2,3)..(2,4);
\draw[-,line width=2pt] (2,4)..controls (1.4,4)..(0,3);
\draw[-,cyan,line width=2pt] (4,2)..controls (3.2,3)..(4,4);
\draw[-,line width=2pt] (5,2.3)..controls (4.5,2)..(4,2);
\draw[-,line width=2pt] (4,2)..controls (4.8,3)..(4,4);
\draw[-,cyan,line width=2pt] (4,4)..controls (4.5,4)..(5,3.7);
\draw[-,cyan,line width=2pt] (2,4)--(3,4);
\draw[-,line width=2pt] (2,2)--(3,2);
\draw[-,line width=2pt] (3,4)--(4,4);
\draw[-,cyan,line width=2pt] (3,2)--(4,2);
\draw[-] (2.97,4)--(2.97,2);
\draw[-,cyan] (3.03,4)--(3.03,2);
\filldraw [red]  (3,4)    circle (1.5pt)
[red]  (3,2)    circle (1.5pt);
\filldraw [black]  (2,2)    circle (2pt)
[black]  (4,2)    circle (2pt)
[black]  (2,4)    circle (2pt)
[black]  (0,3)    circle (2pt)
[black]  (4,4)    circle (2pt);
\draw (0.5,4) node{$P_1$};
\draw (0.5,2) node{$Q_1$};
\end{tikzpicture}

\noindent In fact, the two first diagrams show that one cannot connect $P_5'$ with any other component. Some components are either adjacent,
or extremal, $P_1'$ and $Q_1'$ are empty, and for the three remaining components, the two diagrams show that they cannot be connected. Note that two
of them are components in an elementary ESU, so by symmetry it suffices to show it for one of them.
By symmetry the same holds for $Q_5'$, and the remaining pairs that are not adjacent are shown to be NDC in the last two diagrams.
Note that by symmetry it suffices to prove it for one horizontal connection.

\noindent {\bf Case 13.:} The paths and cycles in the following diagrams show that the graph is WD.

\noindent \begin{tikzpicture}[scale=0.6]
\draw[-,cyan,line width=2pt] (2,4)..controls (2.8,5)..(2,6);
\draw[-,cyan,line width=2pt] (2,2)..controls (3,2.5)..(4,4);
\draw[-,line width=2pt] (2,2)..controls (3,3.5)..(4,4);
\draw[-,cyan,line width=2pt] (2,6)..controls (2.8,6)..(4,4);
\draw[-,cyan,line width=2pt] (2,4)..controls (1.5,3.7)..(1,4);
\draw[-,line width=2pt] (0,4)..controls (1.2,2)..(2,2);
\draw[-,cyan,line width=2pt] (2,2)--(2,3);
\draw[-,line width=2pt] (2,3)--(2,4);
\draw[-,line width=2pt] (4,4)--(4.5,4);
\draw[-,line width=2pt,cyan] (4.5,4)--(5,4);
\draw[-,line width=2pt] (2,4)..controls (1.2,5)..(2,6);
\draw[-,line width=2pt] (2,6)..controls (1.2,6)..(0,4);
\draw[-,white,line width=3pt] (4.4,4.1)..controls (4,5)..(2.2,3.2);
\draw[-,cyan] (4.5,4.05)..controls (4,5.05)..(2,3.05);
\draw[-] (4.48,3.97)..controls (4,5)..(1.96,2.94);
\filldraw [red]  (4.5,4)    circle (1.5pt)
[red]  (2,3)    circle (1.5pt);
\filldraw [black]  (2,2)    circle (2pt)
[black]  (2,4)    circle (2pt)
[black]  (2,6)    circle (2pt)
[black]  (4,4)    circle (2pt)
[black]  (0,4)    circle (2pt);
\draw (5,4.3)node{$P_5'$};
\draw (2.5,1)node{Cycle $C$ in black};
\draw (5.5,2)node{};
\end{tikzpicture}
\begin{tikzpicture}[scale=0.6]
\draw[-,line width=2pt] (2,4)..controls (2.6,4.8)..(2.6,5);
\draw[-,cyan,line width=2pt] (2.6,5)..controls (2.6,5.2)..(2,6);
\draw[-,cyan,line width=2pt] (2,2)..controls (3,2.5)..(4,4);
\draw[-,line width=2pt] (2,2)..controls (3,3.5)..(4,4);
\draw[-,cyan,line width=2pt] (2,6)..controls (2.8,6)..(4,4);
\draw[-,cyan,line width=2pt] (2,4)..controls (1.5,3.7)..(1,4);
\draw[-,line width=2pt] (0,4)..controls (1.2,2)..(2,2);
\draw[-,cyan,line width=2pt] (2,2)--(2,4);
\draw[-,line width=2pt] (4,4)--(4.5,4);
\draw[-,line width=2pt,cyan] (4.5,4)--(5,4);
\draw[-,line width=2pt] (2,4)..controls (1.2,5)..(2,6);
\draw[-,line width=2pt] (2,6)..controls (1.2,6)..(0,4);
\draw[-,white,line width=3pt] (4.4,4.1)..controls (4.5,5)..(2.7,5);
\draw[-,cyan] (4.53,4.05)..controls (4.5,5.05)..(2.63,5.03);
\draw[-] (4.48,3.97)..controls (4.48,5)..(2.56,4.97);
\filldraw [red]  (4.5,4)    circle (1.5pt)
[red]  (2.6,5)    circle (1.5pt);
\filldraw [black]  (2,2)    circle (2pt)
[black]  (2,4)    circle (2pt)
[black]  (2,6)    circle (2pt)
[black]  (4,4)    circle (2pt)
[black]  (0,4)    circle (2pt);
\draw (5,4.3)node{$P_5'$};
\draw (2.5,1)node{Cycle $C$ in black};
\draw (5.5,2)node{};
\end{tikzpicture}
\begin{tikzpicture}[scale=0.6]
\draw[-,line width=2pt] (2,4)..controls (2.6,4.8)..(2.6,5);
\draw[-,cyan,line width=2pt] (2.6,5)..controls (2.6,5.2)..(2,6);
\draw[-,line width=2pt] (2,2)..controls (3,2.5)..(4,4);
\draw[-,line width=2pt] (2,2)..controls (2.8,3.18)..(3,3.38);
\draw[-,cyan,line width=2pt] (3,3.38)..controls (3.2,3.58)..(4,4);
\draw[-,line width=2pt] (2,6)..controls (2.8,6)..(4,4);
\draw[-,cyan,line width=2pt] (2,4)..controls (1.5,3.7)..(1,4);
\draw[-,cyan,line width=2pt] (0,4)..controls (1.2,2)..(2,2);
\draw[-,cyan,line width=2pt] (2,2)--(2,4);
\draw[-,cyan,line width=2pt] (4,4)--(5,4);
\draw[-,line width=2pt] (2,4)..controls (1.2,5)..(2,6);
\draw[-,cyan,line width=2pt] (2,6)..controls (1.2,6)..(0,4);
\draw[-,white,line width=3pt] (3,3.6)..controls (3.1,5)..(2.7,5);
\draw[-,cyan] (3.03,3.38)..controls (3.13,5.03)..(2.6,5.03);
\draw[-] (2.97,3.38)..controls (3.08,5)..(2.56,4.97);
\filldraw [red]  (3,3.38)    circle (1.5pt)
[red]  (2.6,5)    circle (1.5pt);
\filldraw [black]  (2,2)    circle (2pt)
[black]  (2,4)    circle (2pt)
[black]  (2,6)    circle (2pt)
[black]  (4,4)    circle (2pt)
[black]  (0,4)    circle (2pt);
\draw (2.5,1)node{Cycle $C$ in black};
\draw (5.5,2)node{};
\end{tikzpicture}
\begin{tikzpicture}[scale=0.6]
\draw[-,cyan,line width=2pt] (2,4)..controls (2.8,5)..(2,6);
\draw[-,line width=2pt] (2,2)..controls (3,2.5)..(4,4);
\draw[-,cyan,line width=2pt] (2,2)..controls (3,3.5)..(4,4);
\draw[-,cyan,line width=2pt] (2,6)..controls (2.8,6)..(3.5,4.84);
\draw[-,line width=2pt] (3.5,4.84)--(4,4);
\draw[-,cyan,line width=2pt] (2,4)..controls (1.5,3.7)..(1,4);
\draw[-,line width=2pt] (0,4)..controls (1.2,2)..(2,2);
\draw[-,cyan,line width=2pt] (2,2)--(2,3);
\draw[-,line width=2pt] (2,3)--(2,4);
\draw[-,cyan,line width=2pt] (4,4)--(5,4);
\draw[-,line width=2pt] (2,4)..controls (1.2,5)..(2,6);
\draw[-,line width=2pt] (2,6)..controls (1.2,6)..(0,4);
\draw[-] (2,3.03)--(3.5,4.87);
\draw[-,cyan] (2.03,2.97)--(3.53,4.81);
\filldraw [red]  (2,3)    circle (1.5pt)
[red]  (3.5,4.84)    circle (1.5pt);
\filldraw [black]  (2,2)    circle (2pt)
[black]  (2,4)    circle (2pt)
[black]  (2,6)    circle (2pt)
[black]  (4,4)    circle (2pt)
[black]  (0,4)    circle (2pt);
\draw (2.5,1)node{Cycle $C$ in black};
\draw (5.5,6.5)node{};
\end{tikzpicture}

\noindent In fact, the two first diagrams show that one cannot connect $P_5'$ with any other component. Some components are either adjacent,
or extremal, $P_1'$ and $Q_1'$ are empty, and for the three remaining components, the two diagrams show that they cannot be connected. Note that two
of them are components in an elementary ESU, so by symmetry it suffices to show it for one of them.
By symmetry the same holds for $Q_5'$, and the remaining pairs that are not adjacent are shown to be WNDC in the last two diagrams.
Note that by symmetry it suffices to prove it for one pair of components of different elementary ESU's.

\noindent {\bf Case 14.:} The paths and cycles in the following six diagrams show that the graph is WD.

\noindent 
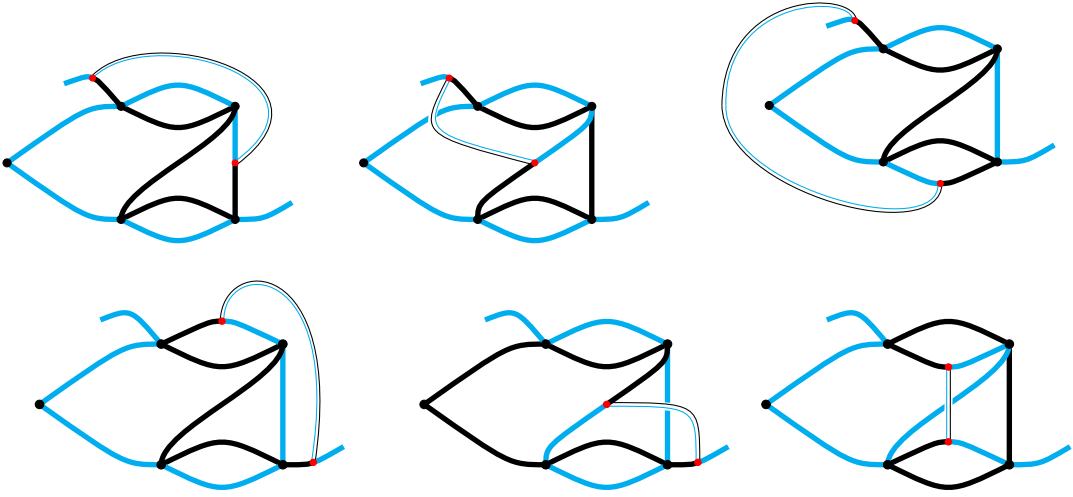
\begin{tikzpicture}[scale=0.75]
\draw[-,line width=2pt] (2,2)..controls (3,2.5)..(4,2);
\draw[-,cyan,line width=2pt] (0,3)..controls (1.4,2)..(2,2);
\draw[-,cyan,line width=2pt] (2,2)..controls (3,1.5)..(4,2);
\draw[-,cyan,line width=2pt] (2,4)..controls (1.4,4)..(0,3);
\draw[-,line width=2pt] (2,4)..controls (3,3.5)..(4,4);
\draw[-,cyan,line width=2pt] (5,2.3)..controls (4.5,2)..(4,2);
\draw[-,cyan,line width=2pt] (2,4)..controls (3,4.5)..(4,4);
\draw[-,line width=2pt] (2,4)..controls (1.6,4.46)..(1.5,4.5);
\draw[-,cyan,line width=2pt] (1.5,4.5)..controls (1.4,4.54)..(1,4.4);
\draw[-,line width=2pt] (4,2)--(4,3);
\draw[-,cyan,line width=2pt] (4,3)--(4,4);
\draw[-,line width=2pt] (2,2)..controls (2,2.6) and (4,3.4)..(4,4);
\draw[-] (1.5,4.55)..controls (2.5,5.55) and (6.05,4.5)..(4.03,2.97);
\draw[-,cyan] (1.5,4.5)..controls (2.5,5.5) and (5.95,4.45)..(4,3.03);
\filldraw [red]  (4,3)    circle (1.5pt)
[red]  (1.5,4.5)    circle (1.5pt);
\filldraw [black]  (2,2)    circle (2pt)
[black]  (4,2)    circle (2pt)
[black]  (2,4)    circle (2pt)
[black]  (0,3)    circle (2pt)
[black]  (4,4)    circle (2pt);
\draw (5,2)node{};
\end{tikzpicture}
\begin{tikzpicture}[scale=0.75]
\draw[-,line width=2pt] (2,2)..controls (3,2.5)..(4,2);
\draw[-,cyan,line width=2pt] (0,3)..controls (1.4,2)..(2,2);
\draw[-,cyan,line width=2pt] (2,2)..controls (3,1.5)..(4,2);
\draw[-,cyan,line width=2pt] (2,4)..controls (1.4,4)..(0,3);
\draw[-,line width=2pt] (2,4)..controls (3,3.5)..(4,4);
\draw[-,cyan,line width=2pt] (5,2.3)..controls (4.5,2)..(4,2);
\draw[-,cyan,line width=2pt] (2,4)..controls (3,4.5)..(4,4);
\draw[-,line width=2pt] (2,4)..controls (1.6,4.46)..(1.5,4.5);
\draw[-,cyan,line width=2pt] (1.5,4.5)..controls (1.4,4.54)..(1,4.4);
\draw[-,line width=2pt] (4,2)--(4,4);
\draw[-,cyan,line width=2pt] (3,3)..controls (4,3.7)..(4,4);
\draw[-,line width=2pt] (2,2)..controls (2,2.3)..(3,3);
\draw[-,white,line width=3pt] (1.45,4.4)..controls (1,3.5)..(2.8,3.25);
\draw[-,cyan] (1.53,4.5)..controls (1.04,3.53)..(3.03,3.03);
\draw[-] (1.47,4.5)..controls (0.96,3.5)..(2.97,2.97);
\filldraw [red]  (3,3)    circle (1.5pt)
[red]  (1.5,4.5)    circle (1.5pt);
\filldraw [black]  (2,2)    circle (2pt)
[black]  (4,2)    circle (2pt)
[black]  (2,4)    circle (2pt)
[black]  (0,3)    circle (2pt)
[black]  (4,4)    circle (2pt);
\draw (6,2)node{};
\end{tikzpicture}
\begin{tikzpicture}[scale=0.75]
\draw[-,line width=2pt] (2,2)..controls (3,2.5)..(4,2);
\draw[-,cyan,line width=2pt] (0,3)..controls (1.4,2)..(2,2);
\draw[-,cyan,line width=2pt] (2,2)..controls (2.8,1.62)..(3,1.62);
\draw[-,line width=2pt] (3,1.62)..controls (3.2,1.62)..(4,2);
\draw[-,cyan,line width=2pt] (2,4)..controls (1.4,4)..(0,3);
\draw[-,line width=2pt] (2,4)..controls (3,3.5)..(4,4);
\draw[-,cyan,line width=2pt] (5,2.3)..controls (4.5,2)..(4,2);
\draw[-,cyan,line width=2pt] (2,4)..controls (3,4.5)..(4,4);
\draw[-,line width=2pt] (2,4)..controls (1.6,4.46)..(1.5,4.5);
\draw[-,cyan,line width=2pt] (1.5,4.5)..controls (1.4,4.54)..(1,4.4);
\draw[-,cyan,line width=2pt] (4,2)--(4,4);
\draw[-,line width=2pt] (2,2)..controls (2,2.6) and (4,3.4)..(4,4);
\draw[-] (-0.83,3)..controls (-0.8,5.05) and (1.5,5.05)..(1.53,4.5);
\draw[-] (-0.83,3)..controls (-0.8,1.45) and (3,0.45)..(3.03,1.62);
\draw[-,cyan] (-0.76,3)..controls (-0.8,4.96) and (1.5,4.96)..(1.46,4.5);
\draw[-,cyan] (-0.76,3)..controls (-0.8,1.55) and (3,0.55)..(2.96,1.62);

\noindent 
\filldraw [red]  (3,1.62)    circle (1.5pt)
[red]  (1.5,4.5)    circle (1.5pt);
\filldraw [black]  (2,2)    circle (2pt)
[black]  (4,2)    circle (2pt)
[black]  (2,4)    circle (2pt)
[black]  (0,3)    circle (2pt)
[black]  (4,4)    circle (2pt);
\draw (6,2)node{};
\end{tikzpicture}

\begin{tikzpicture}[scale=0.8]
\draw[-,line width=2pt] (2,2)..controls (3,2.5)..(4,2);
\draw[-,cyan,line width=2pt] (0,3)..controls (1.4,2)..(2,2);
\draw[-,cyan,line width=2pt] (2,2)..controls (3,1.5)..(4,2);
\draw[-,cyan,line width=2pt] (2,4)..controls (1.4,4)..(0,3);
\draw[-,line width=2pt] (2,4)..controls (3,3.5)..(4,4);
\draw[-,cyan,line width=2pt] (5,2.3)..controls (4.6,2.07)..(4.5,2.04);
\draw[-,line width=2pt] (4.5,2.04)..controls (4.4,2)..(4,2);
\draw[-,line width=2pt] (2,4)..controls (2.8,4.38)..(3,4.38);
\draw[-,cyan,line width=2pt] (3,4.38)..controls (3.2,4.38)..(4,4);
\draw[-,cyan,line width=2pt] (2,4)..controls (1.5,4.6)..(1,4.4);
\draw[-,cyan,line width=2pt] (4,2)--(4,4);
\draw[-,line width=2pt] (2,2)..controls (2,2.6) and (4,3.4)..(4,4);
\draw[-] (4.53,2.04)..controls (5.03,5.5) and (3,5.5)..(2.97,4.38);
\draw[-,cyan] (4.47,2.04)..controls (5,5.45) and (3,5.4)..(3.03,4.38);%
\filldraw [red]  (4.5,2.04)    circle (1.5pt)
[red]  (3,4.38)    circle (1.5pt);
\filldraw [black]  (2,2)    circle (2pt)
[black]  (4,2)    circle (2pt)
[black]  (2,4)    circle (2pt)
[black]  (0,3)    circle (2pt)
[black]  (4,4)    circle (2pt);
\draw (6,2)node{};
\end{tikzpicture}
\begin{tikzpicture}[scale=0.8]
\draw[-,line width=2pt] (2,2)..controls (3,2.5)..(4,2);
\draw[-,line width=2pt] (0,3)..controls (1.4,2)..(2,2);
\draw[-,cyan,line width=2pt] (2,2)..controls (3,1.5)..(4,2);
\draw[-,line width=2pt] (2,4)..controls (1.4,4)..(0,3);
\draw[-,line width=2pt] (2,4)..controls (3,3.5)..(4,4);
\draw[-,cyan,line width=2pt] (5,2.3)..controls (4.6,2.07)..(4.5,2.04);
\draw[-,line width=2pt] (4.5,2.04)..controls (4.4,2)..(4,2);
\draw[-,cyan,line width=2pt] (2,4)..controls (3,4.5)..(4,4);
\draw[-,cyan,line width=2pt] (2,4)..controls (1.5,4.6)..(1,4.4);
\draw[-,cyan,line width=2pt] (4,2)--(4,4);
\draw[-,line width=2pt] (3,3)..controls (4,3.7)..(4,4);
\draw[-,cyan,line width=2pt] (2,2)..controls (2,2.3)..(3,3);
\draw[-,white,line width=3pt] (3.3,3)..controls (4.5,2.96)..(4.5,2.24);
\draw[-] (3,3.03)..controls (4.53,3.03)..(4.53,2.04);
\draw[-,cyan] (3,2.97)..controls (4.47,2.97)..(4.47,2.04);
\filldraw [red]  (4.5,2.04)    circle (1.5pt)
[red]  (3,3)    circle (1.5pt);
\filldraw [black]  (2,2)    circle (2pt)
[black]  (4,2)    circle (2pt)
[black]  (2,4)    circle (2pt)
[black]  (0,3)    circle (2pt)
[black]  (4,4)    circle (2pt);
\draw (5.3,2)node{};
\end{tikzpicture}
\begin{tikzpicture}[scale=0.8]
\draw[-,line width=2pt] (2,2)..controls (2.8,2.38)..(3,2.38);
\draw[-,cyan,line width=2pt] (3,2.38)..controls (3.2,2.38)..(4,2);
\draw[-,cyan,line width=2pt] (0,3)..controls (1.4,2)..(2,2);
\draw[-,line width=2pt] (2,2)..controls (3,1.5)..(4,2);
\draw[-,cyan,line width=2pt] (2,4)..controls (1.4,4)..(0,3);
\draw[-,line width=2pt] (2,4)..controls (3,4.5)..(4,4);
\draw[-,cyan,line width=2pt] (5,2.3)..controls (4.5,2)..(4,2);
\draw[-,line width=2pt] (2,4)..controls (2.8,3.62)..(3,3.62);
\draw[-,cyan,line width=2pt] (3,3.62)..controls (3.2,3.62)..(4,4);
\draw[-,cyan,line width=2pt] (2,4)..controls (1.5,4.6)..(1,4.4);
\draw[-,line width=2pt] (4,2)--(4,4);
\draw[-,cyan,line width=2pt] (2,2)..controls (2,2.6) and (4,3.4)..(4,4);
\draw[-,white,line width=3pt] (3,2.6)--(3,3.4);
\draw[-] (3.03,2.38)--(3.03,3.62);
\draw[-,cyan] (2.97,2.38)--(2.97,3.62);
\filldraw [red]  (3,2.38)    circle (1.5pt)
[red]  (3,3.62)    circle (1.5pt);
\filldraw [black]  (2,2)    circle (2pt)
[black]  (4,2)    circle (2pt)
[black]  (2,4)    circle (2pt)
[black]  (0,3)    circle (2pt)
[black]  (4,4)    circle (2pt);
\draw (6,2)node{};
\end{tikzpicture}

\noindent In fact, one uses the fact that $P_1'$ and $Q_1'$ are empty, discard all adjacent pairs and pairs with two extremal components,
and the remaining pairs are connected in the six diagrams above, where by symmetry we consider the connection with only one of the components of each
elementary ESU.

Thus in all three exceptional cases $V(P)\cap V(Q)$ is a separator, as desired.
\end{proof}

\begin{theorem}\label{teorema principal}
  If $G$ is a simple graph, $P,Q$ are longest paths, $V(P)\ne V(Q)$ and $\# V(P)\cap V(Q) \le 5$, then $V(P)\cap V(Q)$ is a separator.
\end{theorem}

\begin{proof}
  If $\# V(P)\cap V(Q) \le 4$, then it is true by Corollary~\ref{k menor igual a 4}.
   If $\# V(P)\cap V(Q) = 5$ and $BT(P,Q)$ is none of the exceptional
  cases, then by Theorem~\ref{k igual a 5} the graph $BT(P,Q)$ is a concatenation of LNC BB, or it is TD. Hence all ESU's are NC, and so by
  Proposition~\ref{isu nc implica articulation set} and Theorem~\ref{concatenation of LNC blocks} the set $V(P)\cap V(Q)$ is a separator.
  Finally, if $BT(P,Q)$ is in one of the exceptional cases, then we can assume $a_1=b_1$. If $P_0'\ne \emptyset$ and $Q_0'\ne \emptyset$,
  then Proposition~\ref{casos excepcionales con puntas} yields the result. Otherwise necessarily $P_0'=Q_0'=\emptyset$, and then
  Proposition~\ref{casos excepcionales con puntas2} concludes the proof.
\end{proof}

Note that Theorem~\ref{teorema principal} implies Corollary~\ref{Corollary Hippchen 6}.

\begin{corollary}\label{Corollary Gr4}
  Assume that $P$ and $Q$ are two longest paths in a simple graph $G$. If
   $V(Q)\ne V(P)$ and $n=|V(G)|\le 7$ then $V(Q)\cap V(P)$ is a separator.
\end{corollary}

\begin{proof}
  Since $V(Q)\ne V(P)$ and $|V(G)|\le 7$, it follows that $\# (V(P)\cap V(Q))\le 5$, and the result follows from the previous theorem.
\end{proof}

\section{Three longest paths in the exceptional cases}

In this section we will show that in none of the three exceptional cases 6. 13. and 14., we can have three disjoint longest paths.
In order to do this, we will analyze the number of times and the sequential order
in which a third longest path $R$ intersects the components of $BT(P,Q)$.
We will show that in the three cases there is only one pair of components of different colors in $BT(P,Q)$
  connected directly by $R$,
and in the next lemma we prove that this is impossible.

\begin{lemma}\label{lema tecnico tres caminos}
  Assume $P,Q,R$ are disjoint longest paths. Then it is impossible that there is only one pair of components of different colors in $BT(P,Q)$
  connected directly by $R$.
\end{lemma}

\begin{proof}
We will prove the following three statements.
  \begin{enumerate}
     \item Given two components of different colors in $BT(P,Q)$, there cannot be two disjoint subpaths of $R$ joining directly one component with
     the other.
     \item If there is only one pair of components of different colors that are connected directly by $R$, then there is only one subpath of $R$
     connecting directly $P$ and $Q$.
     \item It is impossible that there is only one subpath of $R$ connecting directly $P$ and $Q$.
  \end{enumerate}
Clearly the lemma follows from statements~(2) and~(3).

\medskip

  \noindent \textbf{(1)} The corresponding result for cycles is well known (See for example~\cite{Ch}*{p.145, Claim 1}).
  Assume by contradiction that $R_1$ has one endpoint $x_1$ in $P$ and the other endpoint $y_1$ in $Q$, and
  $R_2$ has one endpoint $x_2$ in $P$ and the other endpoint $y_2$ in $Q$. Then $R_1+Q_{[y_1,y_2]}+R_2$ is internally disjoint from $P$
  and $R_1+P_{[x_1,x_2]}+R_2$ is internally disjoint from $Q$. Then
  $$
  \widehat P=P_{\le x_1}+R_1+Q_{[y_1,y_2]}+R_2+P_{\ge x_2}\quad \text{and}\quad
  \widehat Q=Q_{\le y_1}+R_1+P_{[x_1,x_2]}+R_2+Q_{\ge y_2}
  $$

\noindent
    
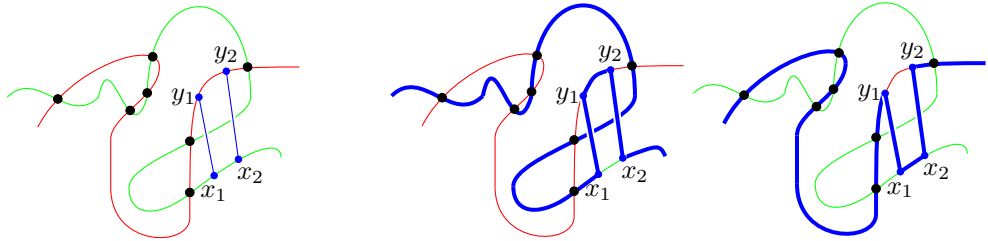
\begin{tikzpicture}[scale=0.8]
     \draw[-,green] (0,2)..controls (0.5,2.5)and (1,1.5)..(1.5,2);
     \draw[-,green] (1.5,2)..controls (1.7,3)and (2,1)..(2.3,2);
     \draw[-,green] (2.3,2)..controls (2.3,4)and (4,4)..(4,2);
     \draw[-,green] (4,2)..controls (4,1.5)and (2,1)..(2,0.5);
     \draw[-,green] (2,0.5)..controls (2,0)and (2.5,0)..(3,0.4);
     \draw[-,green] (3,0.4)..controls (4,1.2)and (4.5,1.3)..(4.5,1);
     \draw[-,red] (0.5,1.5)..controls (1,2.5)and (2.5,3)..(2.5,2.5);
     \draw[-,red] (2.5,2.5)..controls (2.5,2)and (1.7,1.7)..(1.7,1.3);
     \draw[-,red] (1.7,1.3)--(1.7,0.5);
     \draw[-,red] (1.7,0.5)..controls (1.7,-0.5)and (3,-0.5)..(3,0);
     \draw[-,red] (3,0)..controls (3,2.5)..(4.8,2.5);
     \draw[-,white,line width=2pt] (3.8,0.97)--(3.6,2.43);
     \draw[-,white,line width=2pt] (3.15,2)--(3.4,0.7);
     \draw[-,blue] (3.8,0.97)--(3.6,2.43);
     \draw[-,blue] (3.15,2)--(3.4,0.7);
     \filldraw [blue]  (3.6,2.43)    circle (1.5pt)
    [blue]  (3.15,2)    circle (1.5pt);
     \filldraw [blue]  (3.8,0.97)    circle (1.5pt)
    [blue]  (3.4,0.7)    circle (1.5pt);
\filldraw [black]  (0.83,1.97)    circle (2pt)
    [black]  (2.39,2.67)    circle (2pt)
    [black]  (2.3,2.07)    circle (2pt)
    [black]  (2.02,1.78)    circle (2pt)
    [black]  (3,0.42)    circle (2pt)
    [black]  (3,1.27)    circle (2pt)
    [black]  (3.95,2.5)    circle (2pt);
    \draw(3.4,0.4)node{$x_1$};
    \draw(4,0.7)node{$x_2$};
    \draw(2.9,2)node{$y_1$};
    \draw(3.6,2.7)node{$y_2$};
    \draw(2.5,-1)node{$P$ in green, $Q$ in red, $R_1$, $R_2$ in blue};
  \end{tikzpicture}
    \begin{tikzpicture}[scale=0.8]
     \draw[-,blue,line width=1.5pt] (0,2)..controls (0.5,2.5)and (1,1.5)..(1.5,2);
     \draw[-,blue,line width=1.5pt] (1.5,2)..controls (1.7,3)and (2,1)..(2.3,2);
     \draw[-,blue,line width=1.5pt] (2.3,2)..controls (2.3,4)and (4,4)..(4,2);
     \draw[-,blue,line width=1.5pt] (4,2)..controls (4,1.5)and (2,1)..(2,0.5);
     \draw[-,blue,line width=1.5pt] (2,0.5)..controls (2,0)and (2.5,0)..(3,0.4);
     \draw[-,green] (3,0.4)..controls (4,1.2)and (4.5,1.3)..(4.5,1);
     \draw[-,blue,line width=1.5pt] (3,0.4)--(3.4,0.7);
     \draw[-,blue,line width=1.5pt] (3.8,0.97)..controls (4.4,1.2)..(4.5,1);

     \draw[-,red] (0.5,1.5)..controls (1,2.5)and (2.5,3)..(2.5,2.5);
     \draw[-,red] (2.5,2.5)..controls (2.5,2)and (1.7,1.7)..(1.7,1.3);
     \draw[-,red] (1.7,1.3)--(1.7,0.5);
     \draw[-,red] (1.7,0.5)..controls (1.7,-0.5)and (3,-0.5)..(3,0);
     \draw[-,red] (3,0)..controls (3,2.5)..(4.8,2.5);
     \draw[-,blue,line width=1.5pt] (3.15,2)..controls (3.32,2.35)..(3.6,2.43);
     \draw[-,white,line width=3pt] (3.8,0.97)--(3.6,2.43);
     \draw[-,white,line width=3pt] (3.15,2)--(3.4,0.7);
     \draw[-,blue,line width=1.5pt] (3.8,0.97)--(3.6,2.43);
     \draw[-,blue,line width=1.5pt] (3.15,2)--(3.4,0.7);
     \filldraw [blue]  (3.6,2.43)    circle (1.5pt)
    [blue]  (3.15,2)    circle (1.5pt);
     \filldraw [blue]  (3.8,0.97)    circle (1.5pt)
    [blue]  (3.4,0.7)    circle (1.5pt);
\filldraw [black]  (0.83,1.97)    circle (2pt)
    [black]  (2.39,2.67)    circle (2pt)
    [black]  (2.3,2.07)    circle (2pt)
    [black]  (2.02,1.78)    circle (2pt)
    [black]  (3,0.42)    circle (2pt)
    [black]  (3,1.27)    circle (2pt)
    [black]  (3.95,2.5)    circle (2pt);
    \draw(3.4,0.4)node{$x_1$};
    \draw(4,0.7)node{$x_2$};
    \draw(2.9,2)node{$y_1$};
    \draw(3.6,2.7)node{$y_2$};
    \draw(2.5,-1)node{$\widehat P$ in blue};
  \end{tikzpicture}
    \begin{tikzpicture}[scale=0.8]
     \draw[-,green] (0,2)..controls (0.5,2.5)and (1,1.5)..(1.5,2);
     \draw[-,green] (1.5,2)..controls (1.7,3)and (2,1)..(2.3,2);
     \draw[-,green] (2.3,2)..controls (2.3,4)and (4,4)..(4,2);
     \draw[-,green] (4,2)..controls (4,1.5)and (2,1)..(2,0.5);
     \draw[-,green] (2,0.5)..controls (2,0)and (2.5,0)..(3,0.4);
     \draw[-,green] (3,0.4)..controls (4,1.2)and (4.5,1.3)..(4.5,1);
     \draw[-,blue,line width=1.5pt] (0.5,1.5)..controls (1,2.5)and (2.5,3)..(2.5,2.5);
     \draw[-,blue,line width=1.5pt] (2.5,2.5)..controls (2.5,2)and (1.7,1.7)..(1.7,1.3);
     \draw[-,blue,line width=1.5pt] (1.7,1.3)--(1.7,0.5);
     \draw[-,blue,line width=1.5pt] (1.7,0.5)..controls (1.7,-0.5)and (3,-0.5)..(3,0);
     \draw[-,blue,line width=1.5pt] (3,0)..controls (3,2.5)..(4.8,2.5);
     \draw[-,white,line width=3pt] (3.15,2)..controls (3.32,2.35)..(3.6,2.43);
     \draw[-,red] (3.15,2)..controls (3.32,2.35)..(3.6,2.43);
     \draw[-,white,line width=3pt] (3.8,0.97)--(3.6,2.43);
     \draw[-,white,line width=3pt] (3.15,2)--(3.4,0.7);
     \draw[-,blue,line width=1.5pt] (3.8,0.97)--(3.6,2.43);
     \draw[-,blue,line width=1.5pt] (3.15,2)--(3.4,0.7);
     \draw[-,blue,line width=1.5pt] (3.8,0.97)--(3.4,0.7);
     \filldraw [blue]  (3.6,2.43)    circle (1.5pt)
    [blue]  (3.15,2)    circle (1.5pt);
     \filldraw [blue]  (3.8,0.97)    circle (1.5pt)
    [blue]  (3.4,0.7)    circle (1.5pt);
\filldraw [black]  (0.83,1.97)    circle (2pt)
    [black]  (2.39,2.67)    circle (2pt)
    [black]  (2.3,2.07)    circle (2pt)
    [black]  (2.02,1.78)    circle (2pt)
    [black]  (3,0.42)    circle (2pt)
    [black]  (3,1.27)    circle (2pt)
    [black]  (3.95,2.5)    circle (2pt);
    \draw(3.4,0.4)node{$x_1$};
    \draw(4,0.7)node{$x_2$};
    \draw(2.9,2)node{$y_1$};
    \draw(3.6,2.7)node{$y_2$};
    \draw(2.5,-1)node{$\widehat Q$ in blue};
  \end{tikzpicture}

  \noindent are two paths that have lengths that sum $L(P)+L(Q)+2L(R_1)+2L(R_2)$, a contradiction that proves the first statement.

\medskip

  \noindent \textbf{(2)} By item~(1) the only other possibility is that
  two consecutive subpaths of $R$ go back and forth between the components of the only pair connected by $R$.
  But then $R$ intersects one of the paths $P$ or $Q$ in only one point,
  which is impossible, for example by Corollary~\ref{three paths} applied to $V(R)\cap V(P)$ or $V(R)\cap V(Q)$.

\medskip

  \noindent \textbf{(3)} Assume there is only one subpath $R_2$ of $R$
  connecting directly $P$ and $Q$. Let the endpoint of that subpath of $R$ in $P$ be $a$ and the other
  endpoint in $Q$ be $b$.
  Then $a$ and $b$ partition $R$ into three subpaths $R_1,R_2,R_3$, the point $a$ partitions the path $P$ into $P_1$, $P_2$ and $b$ partitions
  the path $Q$ into $Q_1$, $Q_2$. Interchanging if necessary $P$ with $Q$, and $Q_1$ with $Q_2$, we can assume without loss of generality
  that $L(R_1)\ge L(R_3)$ and that $L(Q_1)\ge L(Q_2)$. But then the path
  $R_1+R_2+Q_1$ has length
  $$
  L(R_1)+L(R_2)+L(Q_1)>L(R)/2+L(Q)/2=L(Q)=L(R),
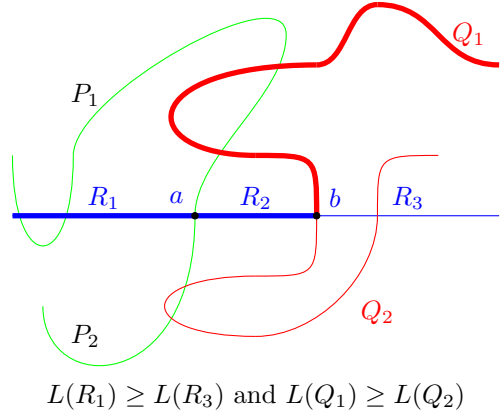
  $$
 $$
 \begin{tikzpicture}[scale=0.8]
     \draw[-,green] (0,4)..controls(0,2)and(1,2)..(1,4);
     \draw[-,green] (1,4)..controls(1,5)and(4.5,7)..(4.5,6);
     \draw[-,green] (4.5,6)..controls(4.5,5)and(3,4)..(3,3);
     \draw[-,green] (3,3)..controls(3,0)and(0.5,0)..(0.5,1.5);
     \draw[-,blue] (0,3)--(8,3);
     \draw[-,blue,line width=2pt] (0,3)--(5,3);
     \draw[-,red] (7,4)..controls(6,4)..(6,3);
     \draw[-,red] (6,3)..controls(6,2)and(5,1)..(4,1);
     \draw[-,red] (4,1)..controls(2,1)and(2,2)..(4,2);
     \draw[-,red] (4,2)..controls(5,2)..(5,3);
     \draw[-,red,line width=2pt] (5,3)..controls(5,4)..(4,4);
     \draw[-,red,line width=2pt] (4,4)..controls(2,4)and(2,5.5)..(5,5.5);
     \draw[-,red,line width=2pt] (5,5.5)..controls(5.5,5.5)and(5.5,6.5)..(6,6.5);
     \draw[-,red,line width=2pt] (6,6.5)..controls(7,6.5)and(7,5.5)..(8,5.5);
     \filldraw [black]  (3,3)    circle (1.5pt)
    [black]  (5,3)    circle (1.5pt);
    \draw[blue](1.5,3.3)node{$R_1$};
    \draw[blue](4,3.3)node{$R_2$};
    \draw[blue](6.5,3.3)node{$R_3$};
    \draw[blue](2.7,3.3)node{$a$};
    \draw[blue](5.3,3.3)node{$b$};
    \draw(1.2,5)node{$P_1$};
    \draw(1.2,1)node{$P_2$};
    \draw[red](7.5,6)node{$Q_1$};
    \draw[red](6,1.4)node{$Q_2$};
    \draw(4,0)node{$L(R_1)\ge L(R_3)$ and $L(Q_1)\ge L(Q_2)$};
    \draw(9,2)node{};
  \end{tikzpicture}
$$
\noindent contradicting that $R$ is a longest path and thus proving statement~(3).
\end{proof}

\begin{proposition} \label{prop caso 6}
  If $P,Q,R$ are disjoint longest paths, then $BT(P,Q)$ cannot have the configuration of Case 6 in the table of section~\ref{section table}.
\end{proposition}

\begin{proof}
  We can assume $a_1=b_1$.

  
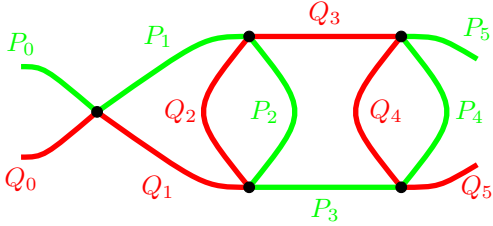
\begin{tikzpicture}[scale=1]
\draw[-,green,line width=2pt] (2,2)..controls (2.8,3)..(2,4);
\draw[-,red,line width=2pt] (0,3)..controls (1.4,2)..(2,2);
\draw[-,red,line width=2pt] (2,2)..controls (1.2,3)..(2,4);
\draw[-,green,line width=2pt] (2,4)..controls (1.4,4)..(0,3);
\draw[-,red,line width=2pt] (4,2)..controls (3.2,3)..(4,4);
\draw[-,green,line width=2pt] (0,3)..controls (-0.7,3.6)..(-1,3.6);
\draw[-,red,line width=2pt] (0,3)..controls (-0.7,2.4)..(-1,2.4);
\draw[-,red,line width=2pt] (5,2.3)..controls (4.5,2)..(4,2);
\draw[-,green,line width=2pt] (4,2)..controls (4.8,3)..(4,4);
\draw[-,green,line width=2pt] (4,4)..controls (4.5,4)..(5,3.7);
\draw[-,red,line width=2pt] (2,4)--(4,4);
\draw[-,green,line width=2pt] (2,2)--(4,2);
\filldraw [black]  (2,2)    circle (2pt)
[black]  (4,2)    circle (2pt)
[black]  (2,4)    circle (2pt)
[black]  (0,3)    circle (2pt)
[black]  (4,4)    circle (2pt);
\draw[green](-1,3.9)node{$P_0$};
\draw[red](-1,2.1)node{$Q_0$};
\draw[red](0.8,2)node{$Q_1$};
\draw[green](0.8,4)node{$P_1$};
\draw[green](3,1.7)node{$P_3$};
\draw[red](3,4.3)node{$Q_3$};
\draw[red](5,2)node{$Q_5$};
\draw[green](5,4.1)node{$P_5$};
\draw[red](3.8,3)node{$Q_4$};
\draw[green](4.9,3)node{$P_4$};
\draw[red](1.1,3)node{$Q_2$};
\draw[green](2.2,3)node{$P_2$};
\draw(5.6,2)node{};
\end{tikzpicture}

   From the proof of Proposition~\ref{casos excepcionales con puntas}
   we know that if $R$ touches $P_0'$, then it can touch only $P_3'$ and $Q_3'$, and that $P_3'$ cannot be connected directly with $Q_3'$, and that
   $Q_3'$ and $P_3'$ can be connected only with $P_0'$ or $Q_0'$.
   Since $P_0'$ and $Q_0'$ are NC, if $R$ touches one of $P_0'$, $Q_0'$, $P_3'$ or $Q_3'$, then the only possibilities for the set of components
   touched by $R$ are four sets with two elements and two sets with three elements, which are connected in the following way:
   $$
   P_0'\sim P_3',\quad Q_0'\sim Q_3', \quad Q_0'\sim P_3',\quad P_0'\sim Q_3', \quad P_3'\sim Q_0'\sim Q_3', \quad\text{or}\quad
   P_3'\sim P_0'\sim Q_3'.
   $$
   In all cases there is at most one pair of components of different colors connected by $R$, which is impossible by
   Lemma~\ref{lema tecnico tres caminos}.
   Hence $R$ cannot touch $P_0'$, $Q_0'$, $P_3'$ nor $Q_3'$.

   We are searching for pairs of components that can be connected directly by a subpath of $R$.
   We have discarded all pairs that contain $P_0'$, $Q_0'$, $P_3'$  or $Q_3'$,
   and we can also discard all pairs that are adjacent
   and have different colors, and all pairs of extremal components of different colors.
   Using the paths in the following 3 diagrams,



\noindent and using that by symmetry, if you cannot connect one component of an elementary ESU, then
    you cannot connect the other; we can discard all pairs except the following four.

\noindent  %

\noindent The blue and the black paths in the following diagram, whose lengths sum more than $2L(P)$, show that
$R$ cannot connect simultaneously $P_1'$ with $Q_5'$ and  $Q_1'$ with $P_5'$.

\noindent %

\noindent Thus there is only one pair of components of different colors in $BT(P,Q)$ connected by $R$, which contradicts
Lemma~\ref{lema tecnico tres caminos} and concludes the proof.
\end{proof}

\begin{proposition} \label{prop caso 13}
  If $P,Q,R$ are disjoint longest paths, then $BT(P,Q)$ cannot have the configuration of Case 13  in the table of section~\ref{section table}.
\end{proposition}

\begin{proof}
  We can assume $a_1=b_1$.

  %

  \noindent  As in the proof of Proposition~\ref{casos excepcionales con puntas}
   we know that $P_0'$ cannot be connected with any other component, and by symmetry the same holds for $Q_0'$.
   The paths in the following three diagrams show that the components in the elementary ESU's cannot be connected to any other component.

  \noindent   %

\noindent Note that we use the symmetry of the graph.
We discard the adjacent pairs and the pair with two extremal components, and are left with 9 possible pairs that $R$ can connect directly.
There are six pairs with the same color
$$
P_1'\sim P_3',\quad P_1'\sim P_5',\quad P_5'\sim P_3',\quad Q_1'\sim Q_3',\quad Q_1'\sim Q_5',\quad Q_5'\sim Q_3',
$$
and three pairs with different colors
$$
P_1'\sim Q_5',\quad P_3'\sim Q_3',\quad P_5'\sim Q_1'.
$$
The paths in the following diagrams show that $R$ can connect only one pair of different colors.

%

\noindent Thus Lemma~\ref{lema tecnico tres caminos} concludes the proof.
\end{proof}

\begin{proposition} \label{prop caso 14}
  If $P,Q,R$ are disjoint longest paths, then $BT(P,Q)$ cannot have the configuration of Case 14  in the table of section~\ref{section table}.
\end{proposition}

\begin{proof}
  We can assume $a_1=b_1$.

  %

  \noindent  As in the proof of Proposition~\ref{casos excepcionales con puntas}
   we know that $P_0'$ cannot be connected with any other component, and by symmetry the same holds for $Q_0'$. We discard also the pairs of adjacent
   components and if both components are extremal. There are 14 possible connections left,
   (where we count the connection to the two components of an elementary ESU only once), and 8 of them are discarded by the
   paths in the following diagrams.

\noindent
    %

  So we are left with 6 possible connections.

\noindent
    %

\noindent
Since $P_2$ and $Q_4$ are the components of an elementary ESU, we also have the connections $P_2'\sim Q_1'$ and $Q_4'\sim P_5$. So we have two
possibilities: either $R$ touches some of $P_2',P_5',Q_1',Q_4'$, or $R$ touches some of $P_1',P_3',Q_3',Q_5'$. In the first case, since
$P_2'$ and $Q_4'$ are NC, they cannot be connected simultaneously by $R$ to the same component, thus we have two possibilities
$$
Q_1'\sim P_2'\sim P_5'\quad \text{or}\quad Q_1'\sim Q_4'\sim P_5'.
$$
But both cases are impossible by Lemma~\ref{lema tecnico tres caminos}.

So $R$ must touch some of $P_1',P_3',Q_3',Q_5'$. The paths in the following diagram show that $R$ cannot connect simultaneously $P_1'\sim Q_3'$
and $P_3'\sim Q_5'$.

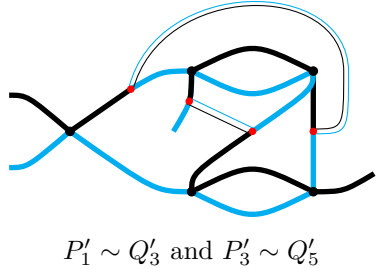
\noindent
    \begin{tikzpicture}[scale=0.8]
\draw[line width=2pt,-] (2,2)..controls (3,2.5)..(4,2);
\draw[line width=2pt,-,cyan] (0,3)..controls (1.4,2)..(2,2);
\draw[line width=2pt,-,cyan] (2,2)..controls (3,1.5)..(4,2);
\draw[line width=2pt,-] (0,3)--(1,3.7);
\draw[line width=2pt,-,cyan] (1,3.7)..controls (1.5,4)..(2,4);
\draw[line width=2pt,-,cyan] (2,4)..controls (3,3.5)..(4,4);
\draw[line width=2pt,-] (0,3)..controls (-0.7,3.6)..(-1,3.6);
\draw[line width=2pt,-,cyan] (0,3)..controls (-0.7,2.4)..(-1,2.4);
\draw[line width=2pt,-] (5,2.3)..controls (4.5,2)..(4,2);
\draw[line width=2pt,-] (2,4)..controls (3,4.5)..(4,4);
\draw[line width=2pt,-,cyan] (2,4)..controls (2,3.5)..(1.7,3);
\draw[line width=2pt,-,cyan] (4,2)--(4,3);
\draw[line width=2pt,-] (4,3)--(4,4);
\draw[-,line width=2pt] (2,4)..controls (1.99,3.6)..(1.96,3.5);
\draw[-,cyan,line width=2pt] (3,3)..controls (4,3.7)..(4,4);
\draw[-,line width=2pt] (2,2)..controls (2,2.3)..(3,3);
\draw[-] (1.03,3.67)..controls(1.5,5.5)and(4.5,5.5)..(4.5,4);
\draw[-] (4.5,4)..controls(4.5,3.04)..(4,3.04);
\draw[-,cyan] (0.97,3.73)..controls(1.44,5.58)and(4.56,5.58)..(4.58,4);
\draw[-,cyan] (4.58,4)..controls(4.58,2.96)..(4,2.96);
\draw[-,cyan] (2,3.54)--(3.04,3.04);
\draw[-] (1.92,3.46)--(2.96,2.96);
\filldraw [red]  (1,3.7)    circle (1.5pt)
[red]  (4,3)    circle (1.5pt);
\filldraw [red]  (1.96,3.5)    circle (1.5pt)
[red]  (3,3)    circle (1.5pt);
\filldraw [black]  (2,2)    circle (2pt)
[black]  (4,2)    circle (2pt)
[black]  (2,4)    circle (2pt)
[black]  (0,3)    circle (2pt)
[black]  (4,4)    circle (2pt);
\draw(2,1)node{$P_1'\sim Q_3'$ and $P_3'\sim Q_5'$};
\draw(-7,3)node{};
\end{tikzpicture}

\noindent Since $P_1$ and $Q_5$ are adjacent and $P_3$ and $Q_3$ are also adjacent,
 there is only one pair of components of different colors in $BT(P,Q)$ connected by $R$, which contradicts
Lemma~\ref{lema tecnico tres caminos} and concludes the proof.
\end{proof}

\begin{theorem} \label{teorema three paths mayor que 5}
  If the intersection of three longest paths $P,Q,R$ is empty, then $\#(V(P)\cap V(Q))\ge 6$.
\end{theorem}

\begin{proof}
  Assume that the intersection of three longest paths $P,Q,R$ is empty, then $\#(V(P)\cap V(Q))\ge 5$ by Corollary~\ref{three paths}.
  If  $\#(V(P)\cap V(Q))= 5$, then we know by
  Theorem~\ref{k igual a 5} that in $BT(P,Q)$ all the ESU's are NC or $BT(P,Q)$ is TD, or it has the representation type of
  one of the cases 6., 13 or 14. If all the ESU's are NC or $BT(P,Q)$ is TD, then Proposition~\ref{tres caminos NC} yields a contradiction.
  Else Propositions~\ref{prop caso 6}, \ref{prop caso 13} and~\ref{prop caso 14} lead to contradictions and thus
  finish the proof.
\end{proof}

\begin{remark}
  In each of the cases that couldn't be discarded constructing longer paths in the three exceptional cases, we proved
  that $R$ connects only one pair of components of different colors of $BT(P,Q)$, which
   is impossible by Lemma~\ref{lema tecnico tres caminos}.

  It would be interesting to characterize the configurations in which this strategy works. One could start trying to find all configurations
  (in which not all ESU's are NC), in which $R$ can connect only one pair of components of different colors of $BT(P,Q)$.
\end{remark}

\end{document}